%% file: root.tex
\documentclass[hidelinks,onefignum,onetabnum]{siamart220329}

\usepackage{cite} 

\usepackage{amsmath}
\usepackage{amsfonts}
\usepackage{amssymb}
\usepackage{amscd}

\usepackage{graphicx}
\usepackage{caption}
\usepackage{subcaption}

\usepackage{tikz}
\usepackage{pgfplots}
  \pgfplotsset{compat = 1.17}
  \pgfkeys{/pgf/number format/.cd, 1000 sep = {\,}}
  \usepgfplotslibrary{external}
  \tikzset{external/system call = {%
    pdflatex \tikzexternalcheckshellescape
      -halt-on-error
      -interaction=batchmode
      -jobname "\image" "\texsource"}}
  \tikzexternaldisable

\usepackage{url}
\usepackage{hyperref}    
\usepackage{nicefrac} % Compact symbols for 1/2, etc.
\usepackage{cancel}   % `\cancel{}' macro.
\usepackage{empheq}  
\usepackage{xspace}

\usepackage[linesnumbered, lined, algoruled, algosection, algo2e]{algorithm2e}
\SetKwRepeat{Do}{do}{while}%

\crefname{algocf}{alg.}{algs.}
\Crefname{algocf}{Algorithm}{Algorithms}

\input{macros}

\usepackage{amsfonts}
\usepackage{graphicx}
\usepackage{epstopdf}
\usepackage{algorithmic}
\ifpdf
  \DeclareGraphicsExtensions{.eps,.pdf,.png,.jpg}
\else
  \DeclareGraphicsExtensions{.eps}
\fi

\newsiamremark{remark}{Remark}
\newsiamremark{hypothesis}{Hypothesis}
\crefname{hypothesis}{Hypothesis}{Hypotheses}
\newsiamthm{claim}{Claim}

\headers{$\CH_2$-optimal model reduction of LQO systems}{S.~Reiter, I.~V. Gosea, I.~Pontes Duff, and S.~Gugercin}

\title{$\mathcal{H}_2$-optimal model reduction  of linear quadratic-output systems by multivariate rational interpolation
\thanks{Preprint.\funding{The work of Reiter and Gugercin  was supported in part by the US National Science Foundation grant AMPS-2318880.}}}

\author{Sean Reiter\thanks{Courant Institute of Mathematical Sciences, New York University, New York, NY 10012 USA
  (\email{s.reiter@nyu.edu}).}
\and Ion Victor Gosea\thanks{Max Planck Institute for Dynamics of Complex Technical Systems,
  Sandtorstr. 1, 39106 Magdeburg, Germany (\email{gosea@mpi-magdeburg.mpg.de}, \email{pontes@mpi-magdeburg.mpg.de}).}
\and Igor Pontes Duff\footnotemark[3] 
\and  Serkan Gugercin\thanks{Department of Mathematics and Division of Computational Modeling and Data
  Analytics, Academy of Data Science, Virginia Tech,
  Blacksburg, VA 24061, USA (\email{gugercin@vt.edu}).}
  }

\ifpdf
\hypersetup{
  pdftitle={},
  pdfauthor={}
}
\fi
\begin{document}

\maketitle

% REQUIRED
\begin{abstract}
\input{abstract}
\end{abstract}

% REQUIRED

% REQUIRED

%%%%%%%%%%%%%%%%%%%%%%%%%%%%%%%%%%%%%%%%%%%%%%%%%%%%%%%%%%%%%%%%%%%%%%%%%%%%%%%%
% INTRODUCTION.                                                                %
%%%%%%%%%%%%%%%%%%%%%%%%%%%%%%%%%%%%%%%%%%%%%%%%%%%%%%%%%%%%%%%%%%%%%%%%%%%%%%%%
\input{introduction}

%%%%%%%%%%%%%%%%%%%%%%%%%%%%%%%%%%%%%%%%%%%%%%%%%%%%%%%%%%%%%%%%%%%%%%%%%%%%%%%%
% BACKGROUND.                                                                %
%%%%%%%%%%%%%%%%%%%%%%%%%%%%%%%%%%%%%%%%%%%%%%%%%%%%%%%%%%%%%%%%%%%%%%%%%%%%%%%%
\input{background}

%%%%%%%%%%%%%%%%%%%%%%%%%%%%%%%%%%%%%%%%%%%%%%%%%%%%%%%%%%%%%%%%%%%%%%%%%%%%%%%%
% RESULTS.                                                                %
%%%%%%%%%%%%%%%%%%%%%%%%%%%%%%%%%%%%%%%%%%%%%%%%%%%%%%%%%%%%%%%%%%%%%%%%%%%%%%%%
\input{h2opt}

\input{lqoirka}
\input{numerics}

%%%%%%%%%%%%%%%%%%%%%%%%%%%%%%%%%%%%%%%%%%%%%%%%%%%%%%%%%%%%%%%%%%%%%%%%%%%%%%%%
% CONCLUSIONS.                                                                %
%%%%%%%%%%%%%%%%%%%%%%%%%%%%%%%%%%%%%%%%%%%%%%%%%%%%%%%%%%%%%%%%%%%%%%%%%%%%%%%%
\input{conc}
 
%%%%%%%%%%%%%%%%%%%%%%%%%%%%%%%%%%%%%%%%%%%%%%%%%%%%%%%%%%%%%%%%%%%%%%%%%%%%%%%%
% *** APPENDICES ***                                                     %
%%%%%%%%%%%%%%%%%%%%%%%%%%%%%%%%%%%%%%%%%%%%%%%%%%%%%%%%%%%%%%%%%%%%%%%%%%%%%%%%
\input{appendices}

%%%%%%%%%%%%%%%%%%%%%%%%%%%%%%%%%%%%%%%%%%%%%%%%%%%%%%%%%%%%%%%%%%%%%%%%%%%%%%%%
% *** ACKNOWLEDGEMENTS ***                                                     %
%%%%%%%%%%%%%%%%%%%%%%%%%%%%%%%%%%%%%%%%%%%%%%%%%%%%%%%%%%%%%%%%%%%%%%%%%%%%%%%%
\section*{Acknowledgments}
The authors thank Mark Embree and Steffen W.~R. Werner for helpful discussions during the preparation of this manuscript.

%%%%%%%%%%%%%%%%%%%%%%%%%%%%%%%%%%%%%%%%%%%%%%%%%%%%%%%%%%%%%%%%%%%%%%%%%%%%%%%%
% *** REFERENCES ***                                                     %
%%%%%%%%%%%%%%%%%%%%%%%%%%%%%%%%%%%%%%%%%%%%%%%%%%%%%%%%%%%%%%%%%%%%%%%%%%%%%%%%
\bibliographystyle{siamplain}
\bibliography{root}
\end{document}

%% file: macros.tex
\mathcode`\@="8000
{\catcode`\@=\active \gdef@{\mkern1mu}}
\newcommand{\ds}{\ensuremath{\mathrm d}}

\newcommand{\trans}{\ensuremath{\mkern-1.5mu\mathsf{T}}}
\newcommand*{\defeq}{\ensuremath{\stackrel{\mathsf{def}}{=}}}
\newcommand{\frob}{\ensuremath{\mathsf{F}}}

\DeclareMathOperator{\vecm}{vec}

\DeclareMathOperator{\trace}{tr}
\DeclareMathOperator{\real}{Re}
\DeclareMathOperator{\imag}{Im}
\DeclareMathOperator{\diag}{diag}

\DeclareMathOperator{\Res}{Res}

\DeclareMathOperator{\range}{Range}
\DeclareMathOperator{\sinc}{sinc}
\DeclareMathOperator{\expec}{Exp}
\DeclareMathOperator{\var}{Var}

\newcommand{\imunit}{{\ensuremath{\dot{\imath\hspace*{-0.2em}\imath}}}}
\def\Sys{\ensuremath{\mathcal{G}}}

\def\Syshat{\ensuremath{\check{\Sys}}}
\def\Sysred{\ensuremath{\wt\Sys}}
\def\Syscheck{\ensuremath{\check\Sys}}

\def\contourRone{\ensuremath{\Gamma_{R_1}}}
\def\contourRtwo{\ensuremath{\Gamma_{R_2}}}
\def\contourRi{\ensuremath{\Gamma_{R_i}}}

\def\wt#1{\ensuremath{\widetilde{#1}}}

\newcommand\R{\ensuremath{\mathbb{R}}}
\newcommand\C{\ensuremath{\mathbb{C}}}
\newcommand\Cm{\ensuremath{\C^{m}}}
\newcommand\Cp{\ensuremath{\C^{p}}}
\newcommand\Cn{\ensuremath{\C^{n}}}
\newcommand\Cr{\ensuremath{\C^{r}}}

\newcommand\Crr{\ensuremath{\C^{r\times r}}}
\newcommand\Cnr{\ensuremath{\C^{n\times r}}}
\newcommand\Cpm{\ensuremath{\C^{p\times m}}}

\newcommand\Rn{\ensuremath{\R^{n}}}
\newcommand\Rm{\ensuremath{\R^{m}}}
\newcommand\Rp{\ensuremath{\R^{p}}}
\newcommand\Rr{\ensuremath{\R^{r}}}

\newcommand\Rnn{\ensuremath{\R^{n\times n}}}
\newcommand\Rnm{\ensuremath{\R^{n\times m}}}
\newcommand\Rpn{\ensuremath{\R^{p\times n}}}
\newcommand\Rpm{\ensuremath{\R^{p\times m}}}

\newcommand\Rrr{\ensuremath{\R^{r\times r}}}
\newcommand\Rpr{\ensuremath{\R^{p\times r}}}
\newcommand\Rrm{\ensuremath{\R^{r\times m}}}
\newcommand\Rnr{\ensuremath{\R^{n\times r}}}

\newcommand{\BA}{\ensuremath{\boldsymbol{A}}}
\newcommand{\BB}{\ensuremath{\boldsymbol{B}}}
\newcommand{\BC}{\ensuremath{\boldsymbol{C}}}
\newcommand{\BD}{\ensuremath{\boldsymbol{D}}}
\newcommand{\BE}{\ensuremath{\boldsymbol{E}}}

\newcommand{\BG}{\ensuremath{\boldsymbol{G}}}
\newcommand{\BH}{\ensuremath{\boldsymbol{H}}}
\newcommand{\BI}{\ensuremath{\boldsymbol{I}}}

\newcommand{\BK}{\ensuremath{\boldsymbol{K}}}

\newcommand{\BM}{\ensuremath{\boldsymbol{M}}}
\newcommand{\BN}{\ensuremath{\boldsymbol{N}}}

\newcommand{\BS}{\ensuremath{\boldsymbol{S}}}
\newcommand{\BT}{\ensuremath{\boldsymbol{T}}}
\newcommand{\BU}{\ensuremath{\boldsymbol{U}}}
\newcommand{\BV}{\ensuremath{\boldsymbol{V}}}
\newcommand{\BW}{\ensuremath{\boldsymbol{W}}}
\newcommand{\BX}{\ensuremath{\boldsymbol{X}}}
\newcommand{\BY}{\ensuremath{\boldsymbol{Y}}}

\newcommand{\Bb}{\ensuremath{\boldsymbol{b}}}
\newcommand{\Bc}{\ensuremath{\boldsymbol{c}}}

\newcommand{\Bg}{\ensuremath{\boldsymbol{g}}}

\newcommand{\Bl}{\ensuremath{\boldsymbol{\ell}}}
\newcommand{\Bm}{\ensuremath{\boldsymbol{m}}}

\newcommand{\Bq}{\ensuremath{\boldsymbol{q}}}
\newcommand{\Br}{\ensuremath{\boldsymbol{r}}}
\newcommand{\Bs}{\ensuremath{\boldsymbol{s}}}
\newcommand{\Bt}{\ensuremath{\boldsymbol{t}}}
\newcommand{\Bu}{\ensuremath{\boldsymbol{u}}}
\newcommand{\Bv}{\ensuremath{\boldsymbol{v}}}
\newcommand{\Bw}{\ensuremath{\boldsymbol{w}}}
\newcommand{\Bx}{\ensuremath{\boldsymbol{x}}}
\newcommand{\By}{\ensuremath{\boldsymbol{y}}}
\newcommand{\Bz}{\ensuremath{\boldsymbol{z}}}

\newcommand{\CH}{\ensuremath{\mathcal{H}}}

\newcommand{\CL}{\ensuremath{\mathcal{L}}}

\newcommand{\bxi}{\ensuremath{\boldsymbol{\xi}}}
\newcommand{\Bzero}{\ensuremath{\mathbf{0}}}

\newcommand{\wtBv}{\ensuremath{\skew1\wt{\Bv}}}
\newcommand{\wtBw}{\ensuremath{\skew1\wt{\Bw}}}
\newcommand{\wtBvarphi}{\ensuremath{\wt{\Bvarphi}}}
\newcommand{\Bvarphi}{\ensuremath{\boldsymbol{\varphi}}}

\newcommand{\BVp}{\ensuremath{\BV_{\operatorname{p}}}}
\newcommand{\BWp}{\ensuremath{\BW_{\operatorname{p}}}}

\newcommand{\BAr}{\ensuremath{\skew5\wt{\BA}}} 
\newcommand{\BBr}{\ensuremath{\skew2\wt{\BB}}}
\newcommand{\BCr}{\ensuremath{\skew2\wt{\BC}}}
\newcommand{\BDr}{\ensuremath{\skew2\wt{\BD}}} 
\newcommand{\BEr}{\ensuremath{\skew4\wt{\BE}}}
\newcommand{\BGr}{\ensuremath{\skew4\wt{\BG}}}
\newcommand{\BMr}{\ensuremath{\skew4\wt{\BM}}} 
\newcommand{\BMk}{\ensuremath{\BM_k}}
\newcommand{\Bxr}{\ensuremath{\skew1\wt{\Bx}}} 
\newcommand{\BVr}{\ensuremath{\BV}} 
\newcommand{\BWr}{\ensuremath{\BW}}
\newcommand{\yr}{\ensuremath{\skew1\wt{y}}}
\newcommand{\Byr}{\ensuremath{\skew1\wt{\By}}}

\def\lambdared{\ensuremath{{\wt{\lambda}}}}
\def\Bbred{\ensuremath{{\skew3\wt{\Bb}}}}
\def\Bcred{\ensuremath{{\skew1\wt{\Bc}}}}
\def\Bmred{\ensuremath{{\wt{\Bm}}}}

\def\Bsred{\ensuremath{{\skew4\wt{\Bs}}}}
\def\Btred{\ensuremath{{\wt{\Bt}}}}
\def\BSred{\ensuremath{\skew2\wt{\BS}}}
\def\BTred{\ensuremath{\wt{\BT}}}

\newcommand{\BGcheck}{\ensuremath{\skew3\check{\BG}}}
\newcommand{\BGbar}{\ensuremath{\skew3\overline{\BG}}}

\newcommand{\BMkr}{\ensuremath{\wt{\BM}_{k}}}

\newcommand{\BGonered}{\ensuremath{\BGr_{1}}}
\newcommand{\BGtwored}{\ensuremath{\BGr_{2}}}

\newcommand{\BGonehat}{\ensuremath{{\BGcheck}_{1}}}
\newcommand{\BGtwohat}{\ensuremath{\BGcheck_{2}}}

\newcommand{\BGonebar}{\ensuremath{\BGbar_{1}}}
\newcommand{\BGtwobar}{\ensuremath{\BGbar_{2}}}

\newcommand{\BGoneRedBar}{\ensuremath{\skew3\overline{\BGr}_1}}
\newcommand{\BGtwoRedBar}{\ensuremath{\skew3\overline{\BGr}_2}}

\newcommand{\BGtwoCheckBar}{\ensuremath{\skew3\overline{\BGcheck}_2}}

\newcommand{\MOR}{\textsf{MOR}\xspace}
\newcommand{\ROM}{\textsf{ROM}\xspace}

\newcommand{\LTI}{\textsf{LTI}\xspace}

\newcommand{\LQO}{\textsf{LQO}\xspace}

\newcommand{\IRKA}{\textsf{IRKA}\xspace}

\newcommand{\SISO}{\textsf{SISO}\xspace}
\newcommand{\SIMO}{\textsf{SIMO}\xspace}
\newcommand{\MIMO}{\textsf{MIMO}\xspace}

\newcommand{\plotfontsize}{\footnotesize}

\newcommand{\LQOIRKA}{\ensuremath{\mathsf{LQO}\mbox{-}\mathsf{IRKA}}}
\newcommand{\LQOIRKAeigs}{\ensuremath{\mathsf{LQO}\mbox{-}\mathsf{IRKA_{eigs}}}}

\newcommand{\LQOIRKAimag}{\ensuremath{\mathsf{LQO}\mbox{-}\mathsf{IRKA_{imag}}}}
\newcommand{\LQOBT}{\ensuremath{\mathsf{LQO}\mbox{-}\mathsf{BT}}}
\newcommand{\relerr}{\ensuremath{\operatorname{relerr}}}

\newcommand{\interpOneStep}{\ensuremath{\mathsf{interp_{oneStep}}}}
\newcommand{\interpOneStepEigs}{\ensuremath{\mathsf{interp_{oneStep,\eigs}}}}

\newcommand{\interpOneStepImag}{\ensuremath{\mathsf{interp_{oneStep,\imagInit}}}}
\newcommand{\imagInit}{\ensuremath{\mathsf{imag}}}

\newcommand{\eigs}{\ensuremath{\mathsf{eigs}}}

\newcommand{\imagsf}{\ensuremath{\mathsf{imag}}}
\newcommand{\MATLAB}{\ensuremath{\mathsf{MATLAB}}}

\definecolor{matlabblue}{HTML}{0072BD}
\definecolor{matlaborange}{HTML}{D95319}
\definecolor{matlabyellow}{HTML}{EDB120}
\definecolor{matlabpurple}{HTML}{7E2F8E}
\definecolor{matlabgreen}{HTML}{77AC30}
\definecolor{matlablightblue}{HTML}{4DBEEE}
\definecolor{matlabred}{HTML}{A2142F}

\tikzstyle{sline} = [
  solid,
  line width = 1.5pt
]

\pgfplotscreateplotcyclelist{plotlist}{
  {black, sline},
  {matlabblue, sline, dashed},
  {matlabpurple, sline, dashdotted},
  {matlabgreen, sline, dotted},
  {matlaborange, sline, dashed},
  {matlabred, sline, dashdotted},
}

\pgfplotscreateplotcyclelist{errorplotlist}{
  {matlabblue, sline, mark=o},
  {matlabpurple, sline, mark=+},
  {matlabgreen, sline, mark=x},
}

\pgfplotscreateplotcyclelist{convplotlist}{
  {black, sline},
  {matlabblue, sline, mark=o},
  {matlabpurple, sline, mark=+},
}

%% file: abstract.tex
This paper addresses the $\CH_2$-optimal approximation of linear dynamical systems with quadratic-output functions, also known as linear quadratic-output systems.
Our major contributions are threefold. 
First, we derive interpolatory first-order optimality conditions for the linear quadratic-output $\CH_2$ minimization problem. 
These conditions correspond to the mixed-multipoint tangential interpolation of the full-order linear- and quadratic-output transfer functions, and generalize the Meier-Luenberger optimality framework for the $\CH_2$-optimal model reduction of linear time-invariant systems.
Second, given the optimal interpolation data, we show how to enforce the interpolatory optimality conditions explicitly by Petrov-Galerkin projection of the full-order model. 
Third, to find the optimal interpolation data, we build on this projection framework and propose a generalization of the iterative rational Krylov algorithm for the $\CH_2$-optimal model reduction of linear quadratic-output systems, called \LQOIRKA{}.
Upon convergence, \LQOIRKA{}  produces reduced linear quadratic-output systems that satisfy the interpolatory optimality conditions.  The method only requires solving shifted linear systems and matrix-vector products, thus making it suitable for large-scale problems. 
Numerical examples are included to illustrate the effectiveness of the proposed method.

%% file: introduction.tex
%%%%%%%%%%%%%%%%%%%%%%%%%%%%%%%%%%%%%%%%%%%%%%%%%%%%%%%%%%%%%%%%%%%
\section{Introduction}
\label{sec:intro}
%%%%%%%%%%%%%%%%%%%%%%%%%%%%%%%%%%%%%%%%%%%%%%%%%%%%%%%%%%%%%%%%%%%%
Mathematical models of dynamical systems are essential tools for understanding and forecasting the behavior of complex physical phenomena.
These systems, which are collections of ordinary differential equations arising from, e.g., discretizations of a partial differential equation, often have large dimension 
due to the fine spatial and temporal resolutions required for accurate predictions.
This in turn commands significant computational resources such as time and memory.
A remedy to this problem is model-order reduction (\MOR{}): the construction of low-order and cheap-to-evaluate surrogate models that are used as high-fidelity approximations in place of the original large-scale system for downstream computational tasks.
We refer to~\cite{Ant05,BenMS05,BenOCW17,AntBG20,BauBF14} and the references therein for a comprehensive overview.

In this work, we consider dynamical systems that evolve linearly in the state equation and contain (up to) quadratic terms in the output equation.
In state space, such systems are formulated as
\begin{align}\label{eq:lqoSys}
    \Sys: \begin{cases} \BE\dot\Bx(t)=\BA\Bx(t)+\BB \Bu(t),\quad\Bx(0)=\Bzero_n,\\[6pt]
    \hphantom{\BE}\By(t)=\underbrace{\BC\Bx(t)}_{\defeq\By_1(t)} +
    \underbrace{\BM\left(\Bx(t)\otimes\Bx(t)\right)}_{\defeq\By_2(t)},
\end{cases}
\end{align}
where $\BE,\BA\in\Rnn$, $\BB\in\Rnm$, $\BC\in\Rpn$ and $\BM\in\R^{p\times n^2}$ describe the time evolution of the internal state variables $\Bx\colon[0,\infty)\to\Rn$ and the outputs $\By\colon[0,\infty)\to\Rp$ under the influence of external inputs $\Bu\colon[0,\infty)\to\Rm$. The matrices $\BE,\BA,\BB,\BC,$ and $\BM$ are a \emph{state-space realization} of $\Sys$. 
We use $\Bzero_n\in\Rn$ to denote the $n$-dimensional vector of all zeros.
We refer to systems of the form~\cref{eq:lqoSys} as \emph{linear quadratic-output} (\LQO{}) systems.
Throughout this work, we assume that the system~\cref{eq:lqoSys} is \emph{asymptotically stable}, i.e., the eigenvalues of the matrix pencil $s\BE - \BA$ have strictly negative real parts, and that the matrix $\BE$ is nonsingular. 
For a discussion of \LQO{} systems with a singular $\BE$ term, see~\cite{PrzDGB24}.
Using standard identities of the Kronecker product, the quadratic term $\By_2$ in~\cref{eq:lqoSys} may also be expressed as
\begin{align}
    \label{eq:altQO}
    \By_2(t)=\BM\left(\Bx(t)\otimes\Bx(t)\right) = \begin{bmatrix}
        \Bx(t)^{\trans}\BM_1\Bx(t)\\[1pt]
        \vdots\\[1pt]
        \Bx(t)^{\trans}\BM_p\Bx(t)\\[1pt]
    \end{bmatrix}~~\mbox{where}~~\BM\defeq\begin{bmatrix}
        \vecm{(\BM_1)}^{\trans}\\[1pt]
        \vdots\\[1pt]
        \vecm{(\BM_p)}^{\trans}\\[1pt]
    \end{bmatrix}
\end{align}
and $\vecm\colon\R^{n\times n}\to\R^{n^2}$ is the vectorization operator.
The matrix $\BM_k\in\Rnn$ models the quadratic component of the $k$-th output.
Because $\BM_k$ can always be replaced by its symmetric part, we assume henceforth that $\BM_k$ is symmetric for each $k$.

Dynamical systems with quadratic-output functions such as~\cref{eq:lqoSys} appear naturally in applications where one is interested in observing quantities computed as the product of time- or frequency-components of the state.
For instance, in the study of structural dynamics or vibro-acoustic problems,
the root mean squared displacement~\cite{VanVNLM12,YueM12,AumW23,ReiW24} of the state $\Bx$ is used to model the vibrational character or average spatial deformation of a given surface.
Other prominent examples include observables that correspond to power or energy~\cite{VanVNLM12,Pul23,HolNSU25}, e.g., the internal energy functional of a port-Hamiltonian system~\cite{MehU23,Van06}, quadratic cost functions in optimal control or design problems~\cite{DiazHGA23,YueM13}, and the variance of a collection of random variables in stochastic modeling~\cite{PulA19}.

With regard to the model reduction of \LQO{} systems~\cref{eq:lqoSys}, our goal is the construction of a new, so-called \emph{reduced-order model} (\ROM{}) of the form
\begin{align}\label{eq:lqoSysRed}
    \Sysred: \begin{cases}  \BEr\dot{\Bxr}(t)=\BAr\Bxr(t)+\BBr \Bu(t),~~\Bxr(0)=\Bzero_{r},
    \\[6pt]
    \hphantom{\BEr}\Byr(t)={\BCr\Bxr(t)} +
    {\BMr\left(\Bxr(t)\otimes\Bxr(t)\right)},
    \end{cases}
\end{align}
with the reduced-order quantities $\BEr$, $\BAr\in\Rrr$, $\BBr\in\Rrm$, $\BCr \in \Rpr$, $\BMr\in\R^{p\times r^2}$, $\Bxr\colon[0,\infty)\to\Rr$ and $\Byr\colon[0,\infty)\to\Rp$ for $r \ll n$.
To be an effective surrogate, the \ROM{}~\cref{eq:lqoSysRed} should accurately reproduce the input-to-output response of the full-order system~\cref{eq:lqoSys} in the sense that $\Byr$ is a good approximation to $\By$ for all admissible inputs.
We consider here methods based on \emph{Petrov-Galerkin projection} for computing~\cref{eq:lqoSysRed}. In this regime, the model reduction task amounts to determining left and right approximation subspaces spanned by the model reduction bases $\BWr\in\Rnr$ and $\BVr\in\Rnr$ so that the reduced model~\cref{eq:lqoSysRed} is computed as
\begin{equation}
\label{eq:projMor}
\BEr=\BWr^{\trans}\BE\BVr,~~\BAr=\BWr^{\trans}\BA\BVr,~~\BBr=\BWr^{\trans}\BB,~~\BCr=\BC\BVr,~~\BMr=\BM\left(\BVr\otimes\BVr\right),
\end{equation}
or equivalently $\BMkr=\BV^{\trans}\BMk\BV$ for each $k$.
In essence, different projection-based model reduction methods amount to different strategies for choosing $\BVr$ and $\BWr$.

In the recent literature, much of the well-established technology for the approximation of purely \emph{linear time-invariant} (\LTI{}) systems~--~those with linear state \emph{and} output equations~--~has been extended to the \LQO{} setting~\cref{eq:lqoSys}.
For instance, generalizations of balancing-related \MOR{} are considered in~\cite{BenGPD21,VanM10,BalG24,PulA19,Pul23,PrzDGB24,SonZXUS24}.
Notably, Benner et al.~\cite{BenGPD21} introduce a novel algebraic quadratic-output Gramian and system $\CH_2$ norm based on the Volterra kernels of~\cref{eq:lqoSys}, and develop a related balanced truncation algorithm.
Interpolatory methods~\cite{VanVNLM12,GosA19,GosG20,DiazHGA23,ReiW24,Bu24} design $\BV$ and $\BW$ so that the rational transfer functions of the reduced-order system~\cref{eq:lqoSysRed} match those of the full-order system, or their derivatives, at specified points in the complex plane. 
Diaz et al.~\cite{DiazHGA23} introduce an overarching framework for \emph{tangential interpolation}~--~the interpolation of matrix-valued transfer functions along specified direction vectors~--~of dynamical systems with up to quadratic-bilinear dynamics and quadratic-bilinear outputs;
this general model class includes~\cref{eq:lqoSys} as a special case.
However, the placement of interpolation points, selection of tangent directions, and type of interpolation one should enforce to guarantee quality surrogates has yet to be investigated in the setting of~\cref{eq:lqoSys}.
We also highlight the recent work by Hillebrecht and Unger~\cite{HilU25}, which introduces a $\CH_\infty$ norm for \LQO{} systems and a related optimization-based \MOR{} algorithm.

The focus of this work is the $\CH_2$-optimal \MOR{} of \LQO{} systems.
Formally, given an order-$n$ \LQO{} system $\Sys$ as in~\cref{eq:lqoSys}, we aim to identify, for a fixed order of reduction $r\ll n$, a reduced-order \LQO{} system $\Sysred$ as in~\cref{eq:lqoSysRed} such that the $\CH_2$ error in approximating $\Sys$ with $\Sysred$ is minimized.
The $\CH_2$-optimal \MOR{} of \LQO{} systems has also been studied in~\cite{Reietal25,YanWJ25}. 
The Wilson (or Gramian-based) $\CH_2$-optimality framework from linear \MOR{}~\cite{Wil70,VanGPA08} is established for \LQO{} systems in~\cite{Reietal25}, and~\cite{YanWJ25} performs $\CH_2$-optimal \MOR{} using the Riemannian $\mathsf{BFGS}$ method.
In the \LTI{} setting, $\CH_2$-optimal reduced models are tangential interpolants of the full-order system; the optimal interpolation points are the mirror images of the reduced model poles.
This was first shown for single-input, single-output (\SISO{}) systems by Meier and Luenberger~\cite{MeiL67}, and later established for multiple-input, multiple-output (\MIMO{}) systems in~\cite{GugAB08,VanGPA08,BunKVW10}.
Similar results hold for other classes of weakly nonlinear systems; e.g., in the $\CH_2$-optimal \MOR{} of \emph{bilinear}~\cite{BenB12,Fla12,FlaG15} or \emph{quadratic-bilinear}~\cite{CaoEtal22,BenGG18} dynamical systems, optimal approximations satisfy so-called \emph{multipoint Volterra series interpolation} conditions.
Thus, it is natural to question whether there exist similar characterizations of $\CH_2$-optimal approximations to \LQO{} systems based on transfer function interpolation. In this paper, we provide a complete and affirmative answer to this question.

More specifically, we establish a novel interpolation-based optimality framework for the $\CH_2$-optimal model reduction of \LQO{} systems~\cref{eq:lqoSys}.
Our main contributions are as follows: After reviewing the requisite mathematical preliminaries in~\Cref{sec:bg}, we provide new formulae for calculating the Hardy $\CH_2$ norm and inner product of an \LQO{} system~\cref{eq:lqoSys} in~\Cref{thm:H2PoleRes}.
In~\Cref{sec:H2OptimalMOR}, we apply~\Cref{thm:H2PoleRes} to derive tangential interpolation-based first-order optimality conditions for the \LQO{} $\CH_2$-optimal model reduction problem. The optimality conditions are presented in~\Cref{thm:lqoH2OptInterpolationCons}; they amount to the Lagrange interpolation of the full-order linear- and quadratic-output transfer functions individually, as well as to the Lagrange and Hermite interpolation of their weighted sum.
We refer to the latter type of interpolation as \emph{mixed-multipoint} tangential interpolation.
In~\Cref{thm:enforceMixedInterp}, we show how to enforce the mixed-multipoint tangential interpolation conditions by Petrov-Galerkin projection~\cref{eq:projMor} using appropriately constructed bases $\BV$ and $\BW$.
To find the optimal interpolation data, an extension of the Iterative Rational Krylov Algorithm (\IRKA)~\cite{GugAB08} is proposed in~\Cref{sec:CompFramework} for the $\CH_2$-optimal \MOR{} of \LQO{} systems~\cref{eq:lqoSysRed}.
The proposed method, which we call \emph{linear quadratic-output \textsf{IRKA}} (\LQOIRKA{}), aims to compute \LQO{}-\ROM{}s that satisfy the interpolatory optimality conditions by iteratively corrected projection.
\Cref{sec:numericalResults} illustrates the effectiveness of \LQOIRKA{} on two model reduction benchmarks;~\Cref{sec:conclusions} concludes the paper.

%% file: background.tex
%%%%%%%%%%%%%%%%%%%%%%%%%%%%%%%%%%%%%%%%%%%%%%%%%%%%%%%%%%%%%%%%%%
\section{Mathematical background and preliminaries}
\label{sec:bg}
%%%%%%%%%%%%%%%%%%%%%%%%%%%%%%%%%%%%%%%%%%%%%%%%%%%%%%%%%%%%%%%%%%%

In this section, we establish the necessary mathematical preliminaries required for the forthcoming results.

%%%%%%%%%%%%%%%%%%%%%%%%%%%%%%%%%%%%%%%%%%%%%%%%%%%%%%%%%%%%%%%%%%%%
\subsection{Kronecker product and vectorization identities}
\label{ss:Kron}
%%%%%%%%%%%%%%%%%%%%%%%%%%%%%%%%%%%%%%%%%%%%%%%%%%%%%%%%%%%%%%%%%%%
First, for ease of reference in the forthcoming technical arguments, we recall from~\cite{MagN79,Bre78} some facts about the Kronecker product $\BX\otimes\BY\in\R^{n_1m_1\times n_2m_2}$ of $\BX\in\R^{n_1\times n_2}$, $\BY\in\R^{m_1\times m_2}$.

If the matrices $\BX_1, \BX_2,$ and $\BY_1,\BY_2$ are compatible in the sense that one can form the products $\BX_1\BX_2$ and $\BY_1\BY_2$, then 
\begin{equation}
    \label{eq:mixedProdProperty}
    \left(\BX_1\otimes\BY_1\right)\left(\BX_2\otimes\BY_2\right)=\left(\BX_1\BX_2\otimes\BY_1\BY_2\right).
\end{equation}
A well-known identity involving the vectorization operator $\vecm\colon\R^{n_1\times n_2}\to\R^{n_1n_2}$ and the Kronecker product that we will exploit is
\begin{equation}
\label{eq:vecKronId}
    \vecm(\BW\BX{\BY})=\left(\BY^{\trans}\otimes\BW\right)\vecm(\BX),
\end{equation}
for matrices $\BX,\BW$ and $\BY$ of compatible dimensions.
In general, the Kronecker product is not commutative in the sense that $(\BX\otimes\BY)\neq (\BY\otimes\BX)$. It is \emph{permutation} equivalent in the sense that $\BK_{n_1n_2}(\BX\otimes\BY)\BK_{m_1m_2}= (\BY\otimes\BX)$,
where $\BK_{n_1n_2}\in\R^{n_1n_2\times n_1n_2}$ and $\BK_{m_1m_2}\in\R^{m_1m_2\times m_1m_2}$ are the \emph{perfect shuffle} matrices defined in~\cite[Def.~3.1]{MagN79}.
From~\cite[Theorem~3.1]{MagN79},
one has for any $\BX\in\R^{n_1\times n_2}$ and $\Bv\in\R^{n_2}$:
\begin{equation}
    \label{eq:kronProperties}
    \BK_{n_1n_2}\left(\BX\otimes \Bv\right)=\left( \Bv\otimes \BX\right),~~~
    \BK_{n_1n_2}\vecm\left(\BX\right)=\vecm\left(\BX^{\trans}\right).
\end{equation} 
Consider the quadratic-output matrix $\BM$ in~\cref{eq:altQO}. For any $\BX\in\C^{n\times n}$ and $\Bz\in\Cn$:
\begin{equation*}
    \BM(\BX\otimes \Bz)=\BM\BK_{nn}(\Bz\otimes\BX)=\begin{bmatrix}
    \vecm{(\BM_1)}^{\trans}\BK_{nn}\\[1pt]
    \vdots\\[1pt]
    \vecm{(\BM_p)}^{\trans}\BK_{nn}\\[1pt]
\end{bmatrix}(\Bz\otimes \BX)=\begin{bmatrix}
    \vecm{(\BM_1)}^{\trans}\\[1pt]
    \vdots\\[1pt]
    \vecm{(\BM_p)}^{\trans}\\[1pt]
\end{bmatrix}(\Bz\otimes \BX),\\
\end{equation*}
by~\cref{eq:kronProperties} and because $\BM_k=\BM_k^{\trans}$ for each $k$.
This yields the symmetry relation
\begin{equation}
    \label{eq:qoMatSymmetry}
    \BM\left(\BX\otimes\Bz\right)=\BM\left(\Bz\otimes \BX\right).
\end{equation}

%%%%%%%%%%%%%%%%%%%%%%%%%%%%%%%%%%%%%%%%%%%%%%%%%%%%%%%%%%%%%%%%%%%
\subsection{Input-output representations of linear quadratic-output systems}
\label{ss:tf}
%%%%%%%%%%%%%%%%%%%%%%%%%%%%%%%%%%%%%%%%%%%%%%%%%%%%%%%%%%%%%%%%%%%
Multiple classes of weakly nonlinear dynamical systems can be understood via an infinite series of \emph{Volterra kernels}~\cite{Rug81}.
Because the nonlinearity in~\cref{eq:lqoSys} is restricted to the output equation, only \emph{two} kernels are required to fully describe the system's input-to-output response~\cite{BenGPD21}.
Solving for the state in~\cref{eq:lqoSys} and plugging it into the equation for $\By(t)$ reveals the input-to-output relationship for any time $t\geq 0$
\begin{equation}
    \label{eq:lqoConv}
    \By(t)=\int_0^t {\Bg_1(\tau)}\Bu(t-\tau)\,\ds\tau 
    + \int_0^t\int_0^t {\Bg_2(\tau_1,\tau_2)}\left(\Bu(t-\tau_1) \otimes \Bu(t-\tau_2)\right)\,\ds\tau_1@\ds\tau_2.
\end{equation}
The Volterra kernels $\Bg_1\colon[0,\infty)\to \Rpm$ and $\Bg_2\colon[0,\infty)\times [0,\infty) \to \R^{p\times m^2}$ that appear in~\cref{eq:lqoConv} 
are defined as
\begin{subequations}
\label{eq:lqoKernels}
\begin{align}
        \label{eq:linKernel}
        \Bg_1(t)&\defeq\BC e^{\BE^{-1}\BA t}\BE^{-1}\BB\\
        \label{eq:quadKernel}
        \mbox{and}~~\Bg_2(t_1,t_2)&\defeq\BM\left(e^{\BE^{-1}\BA t_1}\BE^{-1}\BB\otimes e^{\BE^{-1}\BA t_2}\BE^{-1}\BB\right).
\end{align}
\end{subequations}
By computing the univariate and bivariate Laplace transformations~\cite[Ch.~7.3.1]{AntBG20} of the Volterra kernels in~\cref{eq:linKernel} and~\cref{eq:quadKernel}, one obtains an alternative representation of the \LQO{} system~\cref{eq:lqoSys} in the form of \emph{two rational transfer functions}; see also~\cite[Section~3.2]{DiazHGA23}.
Explicitly, these are the complex matrix-valued functions $\BG_1\colon \C\to\Cpm$ and $\BG_2\colon \C\times \C\to\C^{p\times m^2}$ defined as
\begin{subequations}
\label{eq:lqoTfs}
\begin{align}
    \label{eq:lqoLinTf}
    \BG_1(s_1)&\defeq\BC (s_1\BE-\BA)^{-1}\BB\\
    \label{eq:lqoQuadTf}
    \mbox{and}~~\BG_2(s_1,s_2)&\defeq\BM\left((s_1\BE-\BA)^{-1}\BB\otimes (s_2\BE-\BA)^{-1}\BB\right).
\end{align}
\end{subequations}
The transfer functions $\BG_1$ and $\BG_2$ correspond to the linear- and quadratic-components of the \LQO{} system's input-to-output response.
One can view $\BG_1$ and $\BG_2$ as describing the frequency response of the coupled linear- and quadratic-output subsystems that comprise~\cref{eq:lqoSys}.
Note that $\BG_1$ is precisely the usual transfer function of the \LTI{} system obtained from~\cref{eq:lqoSys} by setting $\BM=\Bzero_{p\times n^2}$.

As a straightforward consequence of~\cref{eq:qoMatSymmetry}, it follows that the quadratic-output transfer function $\BG_2$ in~\cref{eq:lqoQuadTf} and its first partial derivatives are symmetric with respect to the interchange of their arguments and matrix-vector products.
These symmetry conditions will be used to simplify the interpolation-based optimality conditions that we derive in~\Cref{sec:H2OptimalMOR}.
Below, we use the notation 
\begin{equation*}
    \tfrac{\partial}{\partial s_i}\BG_2(s,z)=\tfrac{\partial}{\partial s_i}\BG_2(s_1,s_2)\vert_{(s_1,s_2)=(s,z)},~~i=1,2,
\end{equation*}
for the partial derivatives of $\BG_2$.
\begin{lemma}
    \label{lemma:qoTfSymmetry} 
    Let $\BG_2\colon \C\times \C\to\C^{p\times m^2}$ be defined as in~\cref{eq:lqoQuadTf}. 
    Then, for any $\BU\in\C^{m\times \ell}$ and $\Bv\in\Cm$:
    \begin{align}
        \label{eq:G2Symmetry}
        \BG_2(s,z)(\BU\otimes\Bv) &=\BG_2(z,s)(\Bv\otimes\BU),\\
        \label{eq:G2PartialSymmetry}
        \frac{\partial}{\partial s_1}\BG_2(s,z)(\BU\otimes \Bv)&=\frac{\partial}{\partial s_2}\BG_2(z,s)(\Bv\otimes\BU).
    \end{align}
\end{lemma}

\begin{proof}
    By applying~\cref{eq:mixedProdProperty} and~\cref{eq:qoMatSymmetry} simultaneously, it follows that
    \begin{align*}
        \BG_2(s,z)(\BU\otimes\Bv) &= \BM\left((s\BE-\BA)^{-1}\BB\otimes (z\BE-\BA)^{-1}\BB\right)(\BU\otimes\Bv)  \\
        &=\BM\left((z\BE-\BA)^{-1}\BB\otimes (s\BE-\BA)^{-1}\BB\right)(\Bv\otimes\BU)=\BG_2(z,s)(\Bv\otimes\BU),
    \end{align*}
    for any $\BU\in\C^{m\times \ell}$ and $\Bv\in\Cm$, proving~\cref{eq:G2Symmetry}. Equation \cref{eq:G2PartialSymmetry} follows analogously.
\end{proof}

%%%%%%%%%%%%%%%%%%%%%%%%%%%%%%%%%%%%%%%%%%%%%%%%%%%%%%%%%%%%%%%%%%%%%%%%%%%%%%%%%%
\subsection{The Hardy $\CH_2$ norm of a linear quadratic-output system}
%%%%%%%%%%%%%%%%%%%%%%%%%%%%%%%%%%%%%%%%%%%%%%%%%%%%%%%%%%%%%%%%%%%%%%%%%%%%%%%%%%
To quantify the model reduction error, we use the Hardy $\CH_2$ norm of an \LQO{} system~\cite{GosA19}. 
The definitions of the system $\CH_2$ norm and inner product that we present are derived from an underlying Hilbert space structure of the linear- and quadratic-output transfer functions; see, e.g.,~\cite[Sec.~2.1]{AntBG20}.
Specifically, $\BG_1$ belongs to the \emph{Hardy space} $\CH_2^{p\times m}(\C_+)$ of functions $\BH_1\colon\C\to\C^{p\times m}$ that are analytic in $\C_+$, where $\C_+$ denotes the open right complex half plane, and satisfy the square integrability constraint
\begin{equation*}
    \sup_{x\,>\,0}\int_{-\infty}^\infty \|\BH_1(x+\imunit y)\|_{\frob}^2@\ds y<\infty.
\end{equation*}
Likewise, $\BG_2$ belongs to the Hardy space $\CH_2^{p\times m^2}(\C_+\times\C_+)$ of functions $\BH_2\colon\C\times \C\to\C^{p\times m^2}$ that are analytic in $\C_+\times\C_+$ and satisfy the square integrability constraint
\begin{equation*}
    \sup_{x_1,\,x_2\,>\,0}\int_{-\infty}^\infty\int_{-\infty}^\infty \|\BH_2(x_1+\imunit y_1,x_2+\imunit y_2)\|_{\frob}^2@\ds y_1@\ds y_2<\infty. 
\end{equation*}
For the transfer functions $\BG_1$ and $\BG_2$, these suprema can be shown to be achieved in the limits as $x,x_1$, and $x_2$ approach zero by analytically extending $\BG_1$ and $\BG_2$.
The norms and inner products associated with the Hardy spaces are implicitly introduced next in Definition~\ref{def:H2normTf}.
From here on, we take $\BGonebar(s)$ and $\BGtwobar(s_1,s_2)$ to mean that complex conjugation is applied \emph{only} to the matrices in the transfer function and not the arguments $s,s_1$, and $s_2$, i.e.,
\begin{align}
\begin{split}
\label{eq:conjTf}
    \BGonebar(s)&=\overline{\BC}\big(s\overline{\BE}-\overline{\BA}\big)^{-1}\overline{\BB},\\
    \BGtwobar(s_1,s_2)&=\overline{\BM}\left(\big(s_1\overline{\BE}-\overline{\BA}\big)^{-1}\overline{\BB}\otimes\big(s_2\overline{\BE}-\overline{\BA}\big)^{-1}\overline{\BB}\right).
\end{split}
\end{align}
For dynamical systems~\cref{eq:lqoSys} with real-valued state-space realizations, it follows that $\BGonebar(s)=\BG_1(s)$ and $\BGtwobar(s_1,s_2)=\BG_2(s_1,s_2)$.

\begin{definition}
    \label{def:H2normTf}
    Let $\Sys$ and $\Sysred$ be asymptotically stable \LQO{} systems as in~\cref{eq:lqoSys} and~\cref{eq:lqoSysRed} with transfer functions $\BG_1,\BG_2$ and $\BGonered,\BGtwored$ defined according to~\cref{eq:lqoTfs}.
    The \emph{$\CH_2$ inner product} of $\Sys$ and $\Sysred$ is defined as the sum of the Hardy $\CH_2$ inner products of $\BG_1$ and $\BGonered$, and $\BG_2$ and $\BGtwored$, i.e.,
    \begin{align}
    \begin{split}
        \label{eq:H2ipTf}
        \left\langle\Sys, \Sysred\,\right\rangle_{\CH_2}& \defeq \frac{1}{2\pi}\int_{-\infty}^\infty \trace\left(\BGonebar(-\imunit\omega)@ \BGonered(\imunit\omega)^{\trans}\right) \ds\omega \\
        &~~~
        + \frac{1}{(2\pi)^2}\int_{-\infty}^\infty \int_{-\infty}^\infty \trace\left(\BGtwobar(-\imunit\omega_1,-\imunit\omega_2)@ \BGtwored(\imunit\omega_1,\imunit\omega_2)^{\trans} \right) \ds\omega_1@ \ds\omega_2\\
        &=\left\langle\BG_1,\BGonered\right\rangle_{\CH_2^{p\times m}(\C_+)}+\left\langle\BG_2,\BGtwored\right\rangle_{\CH_2^{p\times m^2}(\C_+\times\C_+)}.
    \end{split}
    \end{align}
    The \emph{$\CH_2$ norm} of $\Sys$ is defined as the sum of the Hardy $\CH_2$ norms of $\BG_1$ and $\BG_2$, i.e.,
    \begin{align}
    \begin{split}
        \label{eq:H2normTf}
        \|\Sys\|_{\CH_2}^2&\defeq  \frac{1}{2\pi}\int_{-\infty}^\infty \|\BG_1(\imunit \omega)\|_{\frob}^2 \,\ds\omega
        + \frac{1}{(2\pi)^2}\int_{-\infty}^\infty\int_{-\infty}^\infty \|\BG_2(\imunit\omega_1,\imunit\omega_2)\|_{\frob}^2\,\ds\omega_1@\ds\omega_2\\
        &=\|\BG_1\|_{\CH_2^{p\times m}(\C_+)}^2+\|\BG_2\|_{\CH_2^{p\times m^2}(\C_+\times\C_+)}^2.
    \end{split}
    \end{align}    
\end{definition}

We mention that \Cref{def:H2normTf} can be formulated in terms of the Volterra kernels in~\cref{eq:lqoKernels} or the state-space matrices of~\cref{eq:lqoSys} and the system's Gramians; the latter of these is more computationally tractable,
see~\cite[Definition~3.1]{BenGPD21} and~\cite[Theorem~2.1]{Reietal25}.
It is a direct consequence of Plancherel's relation in one- and two-variables~\cite{BocC49} that the frequency- and time-domain formulations of~\cref{eq:H2ipTf} and~\cref{eq:H2normTf} are equivalent. 

Our rationale for using the $\CH_2$ norm as a performance metric stems from the fact that the $\CH_2$ \emph{system error} controls the $\CL^{p}_\infty$ \emph{output error} in the time domain. 
For an ``admissible'' input $\Bu$, the $\CL^{p}_\infty$ error between the full- and reduced-order outputs of~\cref{eq:lqoSys} and~\cref{eq:lqoSysRed} is bounded above by the corresponding $\CH_2$ system error, i.e.,
\begin{equation}
    \label{eq:H2ErrorBound}
    \|\By-\Byr\|_{\CL^p_\infty}\leq \|\Sys-\Sysred\|_{\CH_2} \left(\|\Bu\|_{\CL^m_2}^2+\|\Bu\otimes \Bu\|_{\CL^{m^2}_2}^2\right)^{1/2},
\end{equation}
where $\|\By-\Byr\|_{\CL^p_\infty}\defeq \sup_{t\geq0}\|\By(t)-\Byr(t)\|_\infty$, and
\begin{align}
\begin{split}
\label{eq:L2norms}
    \|\Bu\|_{\CL^m_2}^2\defeq\int_0^{\infty}\|\Bu(\tau)\|_{2}^2@\ds\tau,~~
    \|\Bu\otimes\Bu\|_{\CL^{m^2}_2 }^2\defeq\int_0^{\infty}\hspace{-1mm}\int_0^{\infty}\|\Bu(\tau_1)\otimes\Bu(\tau_2)\|_{2}^2@\ds\tau_1@\ds\tau_2.
\end{split}
\end{align}
By ``admissible'' $\Bu$, we mean that the norms in~\cref{eq:L2norms} are finite.
We refer to~\cite[Theorem~3.4]{BenGPD21} or~\cite{Reietal25} for a derivation of~\cref{eq:H2ErrorBound}. 
Thus, if one wishes to ensure that the 
output error induced by~\cref{eq:lqoSysRed} is uniformly small over time $t\geq 0$ for any $\CL_2$ input, then the bound~\cref{eq:H2ErrorBound} suggests minimizing the corresponding $\CH_2$ model reduction error.

This motivates our study of the following $\CH_2$-optimal model reduction problem. Given an order-$n$ asymptotically stable \LQO{} system as in~\cref{eq:lqoSys}, we seek an asymptotically stable reduced model $\Sysred$ as in~\cref{eq:lqoSysRed} of a fixed approximation order $1\leq r< n$ such that the $\CH_2$ error in approximating $\Sys$ is minimized, i.e., $\Sysred$ solves
\begin{equation}
    \label{eq:H2OptimalMORProb}
      \|\Sys-\Sysred\|_{\CH_2}^2=\min_{\textrm{order}(\Syscheck)=r} \|\Sys-\Syscheck\|_{\CH_2}^2~~\mbox{such that}~\Syscheck~\mbox{is asymptotically stable.}
\end{equation} 
The squared $\CH_2$ error is only used for the ease of deriving first-order optimality conditions later on.
The $\CH_2$ minimization problem~\cref{eq:H2OptimalMORProb} is in general nonconvex, and global minimizers are hard to characterize. 
Thus, we adopt the more modest goal of identifying \ROM{}s~\cref{eq:lqoSysRed} that satisfy some first-order necessary conditions for local optimality. 
Here, we derive conditions based upon the tangential interpolation of the (univariate) linear- and (multivariate) quadratic-output transfer functions in~\cref{eq:lqoTfs}.

%%%%%%%%%%%%%%%%%%%%%%%%%%%%%%%%%%%%%%%%%%%%%%%%%
\subsection{A pole-residue formulation of the linear quadratic-output $\CH_2$ system norm} 
%%%%%%%%%%%%%%%%%%%%%%%%%%%%%%%%%%%%%%%%%%%%%%%%%
Before considering~\cref{eq:H2OptimalMORProb}, we first derive new expressions for computing the $\CH_2$ inner product~\cref{eq:H2ipTf} and norm \cref{eq:H2normTf} of an \LQO{} system~\cref{eq:lqoSys} in terms of the poles and residues of its transfer functions $\BG_1$ and $\BG_2$ in~\cref{eq:lqoTfs}.
These expressions will enable us to reformulate the $\CH_2$ minimization problem~\cref{eq:H2OptimalMORProb} as a multivariate rational approximation problem, and ultimately derive the interpolatory optimality conditions that are presented in~\Cref{sec:H2OptimalMOR}.

Consider an asymptotically stable reduced \LQO system $\Sysred$ as in~\cref{eq:lqoSysRed}. Henceforth and unless otherwise specified, we assume that $\Sysred$ has \emph{simple} poles $\lambdared_1,\ldots,\lambdared_r$.
Let $\BGonered$ and $\BGtwored$ be the transfer functions of $\Sysred$ defined according to~\cref{eq:lqoTfs}.
Because the poles of $\Sysred$ are simple, the pair $\BAr,\BEr$ is diagonalizable and satisfies
\begin{equation*}
    \BTred^{\trans}\BAr\BSred=\BDr~~\mbox{and}~~\BTred^{\trans}\BEr\BSred =\BI_r,~~\mbox{where}~~\BDr = \diag(\lambdared_1,\ldots,\lambdared_r),
\end{equation*}
the columns of $\BTred,\BSred\in\Crr$ carry the generalized left and right eigenvectors $\Btred_j$ and $\Bsred_j$ of $\BAr$, and $\BEr$, and $\BI_r$ is the $r\times r$ identity matrix.
It is straightforward to verify that the transfer functions $\BGonered$ and $\BGtwored$ are invariant with respect to the underlying state-space realization~\cref{eq:lqoSysRed} of $\Sysred$.
Thus, we assume without loss of generality that the realization of $\Sysred$ in~\cref{eq:lqoSysRed} is such that $\BEr=\BI_r$ and $\BAr=\BDr$.
Expanding $\BGonered$ and $\BGtwored$ in this representation, we obtain the \emph{pole-residue expansions}
\vspace{-.5mm}
\begin{equation}
\label{eq:poleResidue}
    \BGonered(s)=\sum_{j=1}^r\frac{\Bcred_j\Bbred_j^{~\trans}}{s-\lambdared_j}~~~\mbox{and}~~~\BGtwored(s_1,s_2)=\sum_{j=1}^r\sum_{k=1}^r\frac{\Bmred_{j,k}(\Bbred_j\otimes\Bbred_k)^{\trans}}{(s_1-\lambdared_j)(s_2-\lambdared_k)},
\end{equation}
\vspace{-.5mm}
where the vectors $\Bbred_j\in\Cm$, $\Bcred_j\in\Cp$, and $\Bmred_{j,k}\in\Cp$ defined by
\vspace{-.5mm}
\begin{align}   
    \label{eq:resComponents}
    \Bbred_j^{~\trans}\defeq \Btred_j^{~\trans}\BBr,~~\Bcred_j\defeq\BCr\Bsred_j,~\,\mbox{and}~\,\Bmred_{j,k}\defeq \BMr\left(\Bsred_j\otimes\Bsred_k\right)~~\mbox{for}~~j,k=1,\ldots,r,
\end{align}
\vspace{-.5mm}
are called the \emph{residue directions} of $\BGonered$ and $\BGtwored$.
As a direct consequence of~\cref{eq:qoMatSymmetry}, the left residue directions $\Bmred_{j,k}$ obey the symmetry condition
\vspace{-.5mm}
\begin{equation}
    \label{eq:qoResidueSymmetry}
    \Bmred_{j,k}=\BMr\left(\Bsred_j\otimes\Bsred_k\right)=\BMr\left(\Bsred_k\otimes\Bsred_j\right)=\Bmred_{k,j}~\,\mbox{for each}~\, j,k=1,\ldots,r.
\end{equation}
\vspace{-.5mm}
Similar pole-residue expansions to~\cref{eq:poleResidue} can be derived in the case of repeated poles, although these scenarios rarely appear in practice; see~\cite{VanGA10} for the linear case. 
The expansions in~\cref{eq:poleResidue} enable us to derive the following expressions.

\begin{theorem}
    \label{thm:H2PoleRes}
    Suppose that $\Sys$ and $\Sysred$ are asymptotically stable \LQO{} systems as in~\cref{eq:lqoSys} and~\cref{eq:lqoSysRed} having the transfer functions $\BG_1,\BG_2$, and $~\BGonered,\BGtwored$ defined according to~\cref{eq:lqoTfs}, and that $\Sysred$ has simple poles $\lambdared_1,\ldots,\lambdared_r$.
    Then, the $\CH_2$ inner product~\cref{eq:H2ipTf} of $\Sys$ and $\Sysred$ and the norm~\cref{eq:H2normTf} of $\Sysred$ are given by
    \begin{align}
    \begin{split}
    \label{eq:H2ipPoleRes}
        \left\langle\Sys,\Sysred\,\right\rangle_{\CH_2}&=\sum_{i=1}^r \Bcred_i^{\,\trans}\BGonebar(-{\lambdared}_i)\Bbred_i+\sum_{j=1}^r\sum_{k=1}^r \Bmred_{j,k}^{\,\trans}\BGtwobar(-{\lambdared}_j,-{\lambdared}_k)(\Bbred_j\otimes\Bbred_k)\\
        &=\langle\BG_1,@\BGonered\rangle_{\CH_2^{p\times m}(\C_+)}+\langle\BG_2,@\BGtwored\rangle_{\CH_2^{p\times m^2}(\C_+\times \C_+)}
    \end{split}\\
    \begin{split}
    \label{eq:H2normPoleRes}
        \mbox{and}~~\|\Sysred\|_{\CH_2}^2&=\sum_{i=1}^r \Bcred_i^{\,\trans}\BGoneRedBar(-{\lambdared}_i)\Bbred_i +\sum_{j=1}^r\sum_{k=1}^r \Bmred_{j,k}^{\,\trans}\BGtwoRedBar(-{\lambdared}_j,-{\lambdared}_k)(\Bbred_j\otimes\Bbred_k)\\
        &=\|\BGonered\|_{\CH_2^{p\times m}(\C_+)}^2+\|\BGtwored\|_{\CH_2^{p\times m^2}(\C_+\times \C_+)}^2.
    \end{split}
    \end{align}
\end{theorem}

\begin{proof}
    Derivations to prove~\Cref{thm:H2PoleRes} are conceptually intuitive yet technically intricate. 
    Therefore, we leave its presentation to~\Cref{app:proofPoleRes}.
\end{proof}

When applied to a pair of \LTI{} systems, which are a special case of~\cref{eq:lqoSys} with $\BM=\Bzero_{p\times n^2}$,~\Cref{thm:H2PoleRes} agrees with the usual expressions of the $\CH_2$ norm and inner product for \LTI{} systems; see~\cite[Lemma~2.1.4]{AntBG20} or~\cite[Lemma~3.5]{GugAB08}. 

%%%%%%%%%%%%%%%%%%%%%%%%%%%%%%%%%%%%%%%%%%%%%%%%%%%%%%%%%%%%%%%%%%%
\subsection{A review of $\CH_2$-optimal model reduction for linear time-invariant systems}
\label{sec:H2OptimalLTIMOR}
%%%%%%%%%%%%%%%%%%%%%%%%%%%%%%%%%%%%%%%%%%%%%%%%%%%%%%%%%%%%%%%%%%%

Finally, because we build upon ideas from linear $\CH_2$-optimal model reduction, and to draw comparisons later on, we briefly recall the interpolation-based $\CH_2$-optimality theory for \LTI{} systems developed in~\cite{MeiL67,GugAB08}.

Consider the \LTI{} system retrieved from~\cref{eq:lqoSys} by taking $\BM=\Bzero_{p\times n^2}$. 
In this case, the quadratic-output transfer function $\BG_2$ in~\cref{eq:lqoQuadTf} vanishes, and the system's frequency response is fully specified by $\BG_1$ in~\cref{eq:lqoLinTf}.
Then, if an asymptotically stable \LTI{}-\ROM{} with transfer function $\BGonered$ has simple poles and is $\CH_2$-optimal, it holds that
\begin{align}
    \begin{split}
    \label{eq:ltiH2OptimalInterpolationCons}
    \BG_1(-\lambdared_k)\Bbred_k=\BGonered(-\lambdared_k)\Bbred_k,&~~~\Bcred_k^{\trans}\BG_1(-\lambdared_k)=\Bcred_k^{\trans}\BGonered(-\lambdared_k),\\ 
    \mbox{and}~~\Bcred_k^{\trans}\frac{d}{ds}\BG_1(-\lambdared_k)\Bbred_k&=\Bcred_k^{\trans}\frac{d}{ds}\BGonered(-\lambdared_k)\Bbred_k,~~k=1,\ldots,r.
    \end{split}
\end{align}
In other words, the transfer function $\BGonered$ of an $\CH_2$-optimal reduced model is a tangential interpolant of $\BG_1$ at the \emph{mirror images of the reduced model poles}.
For \SISO{} systems, these were first derived by Meier and Luenberger~\cite{MeiL67} and later established for \MIMO{} systems in~\cite{GugAB08,VanGPA08,BunKVW10}.
The work~\cite{GugAB08} proposed the Iterative Rational Krylov Algorithm (\IRKA{}), a numerically efficient approach for computing locally $\CH_2$-optimal \ROM{}s. \IRKA{} iteratively constructs tangential interpolants by treating the poles of the previous reduced model iterate as fixed points.

%% file: h2opt.tex
%%%%%%%%%%%%%%%%%%%%%%%%%%%%%%%%%%%%%%%%%%%%%%%%%%%%%%%%%%%%%%%%%%%
\section{Optimal-$\CH_2$ approximation of linear quadratic-output systems by multivariate rational interpolation}
\label{sec:H2OptimalMOR}
%%%%%%%%%%%%%%%%%%%%%%%%%%%%%%%%%%%%%%%%%%%%%%%%%%%%%%%%%%%%%%%%%%%
In this section, we formally consider the $\CH_2$-optimal model reduction problem for \LQO{} systems stated in~\cref{eq:H2OptimalMORProb}.
The major theoretical result of this work in~\Cref{thm:lqoH2OptInterpolationCons} establishes first-order \emph{interpolatory} optimality conditions for the $\CH_2$ approximation of~\cref{eq:lqoSys}.
These provide a satisfying generalization of the interpolatory $\CH_2$-optimality conditions from linear model reduction~\cref{eq:ltiH2OptimalInterpolationCons}, thus establishing the analogous $\CH_2$-optimality framework for \LQO{} systems.

%%%%%%%%%%%%%%%%%%%%%%%%%%%%%%%%%%%%%%%%%%%%%%%%
\subsection{Mixed-multipoint tangential interpolation conditions for $\CH_2$ optimality}
%%%%%%%%%%%%%%%%%%%%%%%%%%%%%%%%%%%%%%%%%%%%%%%%

We are ready to state the principal theoretical result of the paper.

\begin{theorem}
    \label{thm:lqoH2OptInterpolationCons}
   Let $\Sys$ and $\Sysred$ be asymptotically stable \LQO{} systems as in~\cref{eq:lqoSys} and~\cref{eq:lqoSysRed} with the transfer functions $\BG_1,\BG_2$, and $~\BGonered,\BGtwored$ defined according to~\cref{eq:lqoTfs}. Also suppose that $\Sysred$ has simple poles $\lambdared_1,\ldots,\lambdared_r$. Let $\Bbred_j\in\Cm, \Bcred_j\in\Cp, \Bmred_{j,k}\in\Cp$ be the corresponding residue directions defined in~\cref{eq:resComponents}. 
    If $\Sysred$ minimizes the squared $\CH_2$ error in~\cref{eq:H2OptimalMORProb}, then $\BGonered$ and $\BGtwored$ satisfy the tangential interpolation conditions:
    \begin{subequations}
        \label{eq:H2OptConds}
        \begin{align}
            \label{eq:H1RightCon}
            \Bzero_p &=\left(\BG_1(-\lambdared_k)-\BGonered(-\lambdared_k)\right)\Bbred_k,\\
            \label{eq:H2RightCon}
            \Bzero_p&= \left(\BG_2(-\lambdared_j,-\lambdared_k)-\BGtwored(-\lambdared_j,-\lambdared_k)\right)(\Bbred_j\otimes\Bbred_k),\\
            \begin{split}
                \label{eq:H1H2MixedLeftCon}
                \Bzero_m^{\trans} &= \Bcred_k^{\,\trans}\left(\BG_1(-\lambdared_k)-\BGonered(-\lambdared_k)\right)\\
                &~~~~~~~~~+\sum_{\ell=1}^r\Bmred_{k,\ell}^{\,\trans} \left(\BG_2(-\lambdared_k,-\lambdared_{\ell})-\BGtwored(-\lambdared_k,-\lambdared_{\ell})\right)(\BI_m\otimes \Bbred_{\ell}) \\
                &~~~~~~~~~+\sum_{\ell=1}^r\Bmred_{\ell,k}^{\,\trans} \left(\BG_2(-\lambdared_{\ell},-\lambdared_k)-\BGtwored(-\lambdared_{\ell},-\lambdared_k)\right)(\Bbred_{\ell}\otimes\BI_m),
            \end{split}\\
            \begin{split}
            \label{eq:H1H2MixedHermiteCon}
                0 &= \Bcred_k^{\,\trans}\left(\frac{d}{ds}\BG_1(-\lambdared_k)-\frac{d}{ds}\BGonered(-\lambdared_k)\right)\Bbred_k\\
                &~~~~~~~~~+\sum_{{\ell=1}}^r\Bmred_{k,\ell}^{\,\trans} \left(\frac{\partial}{\partial s_1}\BG_2(-\lambdared_k,-\lambdared_{\ell})-\frac{\partial}{\partial s_1}\BGtwored(-\lambdared_k,-\lambdared_{\ell})\right)(\Bbred_k\otimes \Bbred_{\ell}) \\
                &~~~~~~~~~+\sum_{\ell=1}^r\Bmred_{\ell,k}^{\,\trans} \left(\frac{\partial}{\partial s_2}\BG_2(-\lambdared_{\ell},-\lambdared_k)-\frac{\partial}{\partial s_2}\BGtwored(-\lambdared_{\ell},-\lambdared_k)\right)( \Bbred_{\ell}\otimes\Bbred_k),
            \end{split}        \end{align}
        for all $j,k=1,\ldots,r$.
    \end{subequations}
\end{theorem}
\begin{proof}[Proof of \Cref{thm:lqoH2OptInterpolationCons}]
    Due to its length, we present the full proof of~\Cref{thm:lqoH2OptInterpolationCons} in~\Cref{app:proofH2opt}; we provide an overview here.
    Let $\Syshat$ be any order-$r$, asymptotically stable \LQO{} system defined according to~\cref{eq:lqoSysRed} that exists in a neighborhood of $\Sysred$ 
    such that $\Syshat$ is a locally sub-optimal $\CH_2$ approximation of $\Sys$. 
    Let $\BGonehat$ and $\BGtwohat$ be the transfer functions of $\Syshat$ according to~\cref{eq:lqoTfs}.
    The sub-optimality assumption along with elementary manipulations of the transfer function $\CH_2$ norms and inner products yields
    \begin{align}
        \nonumber
        \|\Sys-\Sysred\|_{\CH_2}^2\leq \|\Sys-\Syshat\|_{\CH_2}^2&=\|\BG_1-\BGonehat\|_{\CH_2^{p\times m}}^2+\|\BG_2-\BGtwohat\|_{\CH_2^{p\times m^2}}^2\\
        \begin{split}
        \label{eq:inequalityStar1}
        \Rightarrow~~~ 0&\leq2\real @\langle\BG_1-\BGonered,@\BGonered-\BGonehat\rangle_{\CH_2^{p\times m}} + \|\BGonered - \BGonehat\|_{\CH_2^{p\times m}}^2\\
        &~~~~~+ 2\real @\langle\BG_2-\BGtwored,@\BGtwored-\BGtwohat\rangle_{\CH_2^{p\times m^2}} + \|\BGtwored - \BGtwohat\|_{\CH_2^{p\times m^2}}^2.
        \end{split}
    \end{align}
    The skeleton of the argument that we use to derive each set of interpolation conditions in~\cref{eq:H2OptConds} is as follows: First, assume for the sake of contradiction that a single interpolation condition in one of~\cref{eq:H1RightCon}--\cref{eq:H1H2MixedHermiteCon} does not hold. Then, for an arbitrarily fixed $\varepsilon > 0$, take $\BGonehat$ and $\BGtwohat$ to differ from the $\CH_2$-optimal transfer functions $\BGonered$ and $\BGtwored$ by carefully selected $\varepsilon$-perturbations of the poles or residue directions (e.g., perturbing $\Bmred_{j,k}$ ultimately yields the $(j,k)$-th right-tangential Lagrange condition~\cref{eq:H2RightCon}).
    The formulae in~\Cref{thm:H2PoleRes} are then used to evaluate the norms and inner products in~\cref{eq:inequalityStar1}. 
    Finally, taking $\varepsilon>0$ to be sufficiently small yields a contradiction to the inequality in~\cref{eq:inequalityStar1}, and so the interpolation condition in question must hold. Repeating this for each of~\cref{eq:H1RightCon}--\cref{eq:H1H2MixedHermiteCon} and all $j,k=1,\ldots,r$ proves the full result.
    The full details are in~\Cref{app:proofH2opt}.
\end{proof}

\Cref{thm:lqoH2OptInterpolationCons} explicitly ties the $\CH_2$-optimal model reduction of \LQO{} systems~\cref{eq:lqoSys} with multivariate rational interpolation.
Specifically, it shows that any minimizer of the $\CH_2$ model error in~\cref{eq:H2OptimalMORProb} is necessarily a tangential interpolant of the full-order model. 
The interpolatory optimality conditions in~\cref{eq:H2OptConds} amount to: 
\begin{enumerate}
    \item The \emph{right-tangential Lagrange interpolation} of $\BG_1$ and $\BG_2$, individually;
    \item The \emph{left-tangential Lagrange interpolation} of the sum of $\BG_1$ and $\BG_2$ evaluated at all possible combinations of the optimal interpolation points;
    \item The \emph{bitangential Hermite interpolation} of the sum of $\BG_1$ and $\BG_2$ evaluated at all possible combinations of the optimal interpolation points.
\end{enumerate}
Henceforth, we refer to the conditions appearing in~\cref{eq:H1H2MixedLeftCon} and~\cref{eq:H1H2MixedHermiteCon} as \emph{mixed-multipoint} tangential interpolation conditions, given that they interpolate a linear combination (or mix) of $\BG_1$ and $\BG_2$ evaluated at multiple (and in fact, all possible) combinations of the optimal interpolation points.

How does the $\CH_2$-optimality framework prescribed by~\Cref{thm:lqoH2OptInterpolationCons} compare with analogous interpolation-based optimality frameworks in linear and nonlinear model reduction?
As with the $\CH_2$-optimal model reduction of \LTI{}~\cite{GugAB08,VanGPA08}, bilinear~\cite{FlaG15}, and quadratic-bilinear~\cite{CaoEtal22} systems, the optimal interpolation points from~\Cref{thm:lqoH2OptInterpolationCons} are the \emph{mirror images of the reduced model poles reflected across the imaginary axis}; the optimal tangential directions are the residue directions~\cref{eq:resComponents} associated with these poles. 
Moreover, when applied to a pair of \LTI{} systems with only linear outputs,
then the quadratic-output transfer functions $\BG_2$ and $\BGtwored$ vanish, and \Cref{thm:lqoH2OptInterpolationCons} yields the familiar interpolation-based first-order optimality conditions~\cref{eq:ltiH2OptimalInterpolationCons} from \LTI{}-\MOR{}.
Alternatively, the mixed-multipoint tangential conditions in~\cref{eq:H2OptConds} can be viewed as respecting the external Volterra series representation~\cref{eq:lqoConv} of the underlying system. This is referred to as \emph{multipoint Volterra series interpolation} in the $\CH_2$-optimal \MOR{} of bilinear and quadratic-bilinear systems; see~\cite{FlaG15,Fla12} and~\cite{CaoEtal22} for further details.

%%%%%%%%%%%%%%%%%%%%%%%%%%%%%%%%%%%%%%%%%%%%%%%%%%%%%%%%%%%
\subsection{Enforcing the necessary optimality conditions of~\Cref{thm:lqoH2OptInterpolationCons} by projection}\label{ss:enforcingInterp}
%%%%%%%%%%%%%%%%%%%%%%%%%%%%%%%%%%%%%%%%%%%%%%%%%%%%%%%%%%%

For the time being, suppose that the optimal interpolation data (the poles and residue directions of an \LQO{} system~\cref{eq:lqoSysRed} that minimizes the $\CH_2$ error in~\cref{eq:H2OptimalMORProb}) are given. Can the interpolation-based optimality conditions of~\Cref{thm:lqoH2OptInterpolationCons} be enforced by Petrov-Galerkin projection? 
From~\cite[Theorem~3.2]{Reietal25}, it is known that any $\CH_2$-optimal \ROM{} of the form~\cref{eq:lqoSysRed} is necessarily obtained via projection.
As an immediate consequence, the interpolatory $\CH_2$-optimal \ROM{}s characterized by~\Cref{thm:lqoH2OptInterpolationCons} are projection-based, as well. However, it is not \emph{a priori} clear how to enforce all of the $3r+r^2$ interpolation conditions in~\cref{eq:H2OptConds} simultaneously by an appropriate choice of model reduction bases $\BVr$ and $\BWr$. 
It is shown in~\cite[Cor.~1]{DiazHGA23} how to enforce the right-tangential Lagrange conditions~\cref{eq:H1RightCon} and~\cref{eq:H2RightCon}, but not the newly derived mixed-multipoint conditions~\cref{eq:H1H2MixedLeftCon} and~\cref{eq:H1H2MixedHermiteCon} that are necessary for optimality.
In the subsequent result, we prove how to enforce \emph{all} of the necessary interpolation conditions simultaneously by explicit construction of $\BV$ and $\BW$ in~\cref{eq:projMor}.

\begin{theorem}
\label{thm:enforceMixedInterp}  
    Let $\Sys$ and $\Sysred$ be \LQO{} systems as in~\cref{eq:lqoSys} and~\cref{eq:lqoSysRed} with the transfer functions $\BG_1,\BG_2$, and $~\BGonered,\BGtwored$ defined according to~\cref{eq:lqoTfs}.
    Consider interpolation points $\sigma_1,\ldots,\sigma_r\in\C$ such that $\sigma_k\BE-\BA$ and $\sigma_k\BEr-\BAr$ are invertible for all $k=1,\ldots,r$, right-tangential directions $\Br_1,\ldots,\Br_r\in\Cm$, and left-tangential directions $\Bl_1,\ldots,\Bl_r\in\Cp$ and $\Bq_{1,1}, \ldots, \Bq_{r,r}\in\Cp$ such that $\Bq_{j,k}=\Bq_{k,j}$ for all $j,k=1,\ldots,r$.
    Suppose that $\BV\in\Cnr$ and $\BW\in\Cnr$ have full rank and satisfy
    \begin{align}
    \label{eq:interpBasisVr}
       \Bv_k\defeq \left(\sigma_k\BE-\BA\right)^{-1}\BB\Br_k&\in\range\left(\BV\right),\\
    \label{eq:interpBasisWr}
        \Bw_k\defeq\left(\sigma_k\BE^{\trans}-\BA^{\trans}\right)^{-1}\left(2\sum_{\ell=1}^r\BN_\ell\,\Bq_{k,\ell} + \BC^{\trans}\Bl_k\right)&\in\range(\BW),
    \end{align}
    for all $k=1,\ldots,r$, where $\BN_{\ell}\defeq\begin{bmatrix}
                \BM_1\Bv_\ell & \cdots & \BM_p\Bv_\ell
            \end{bmatrix}\in\C^{n\times p}$.
    Then, if $\Sysred$ is computed by Petrov-Galerkin projection~\cref{eq:projMor} using $\BVr$ and $\BWr$ satisfying~\cref{eq:interpBasisVr} and~\cref{eq:interpBasisWr},
    its transfer functions $\BGonered$ and $\BGtwored$ satisfy the tangential interpolation conditions:
    \begin{subequations}
        \label{subeq:genInterpolationCons}
        \begin{align}
            \label{eq:linRightTangential}
            \Bzero_p &=\left(\BG_1(\sigma_k)-\BGonered(\sigma_k)\right)\Br_k,  \\
            \label{eq:quadRightTangential}
            \Bzero_p &= \left(\BG_2(\sigma_j,\sigma_k)-\BGtwored(\sigma_j,\sigma_k)\right)\left(\Br_j\otimes\Br_k\right),\\
            \begin{split}
                \label{eq:MixedLeftTangential}
                \Bzero_m^{\trans} &= \Bl_k^{\trans}\left(\BG_1(\sigma_k)-\BGonered(\sigma_k)\right)+\sum_{\ell=1}^r\Bq_{k,\ell}^{\trans} \left(\BG_2(\sigma_k,\sigma_\ell)-\BGtwored(\sigma_k,\sigma_\ell)\right)\left(\BI_m\otimes \Br_\ell\right) \\
                &~~~~~~~~~+\sum_{\ell=1}^r\Bq_{\ell,k}^{\trans} \left(\BG_2(\sigma_\ell,\sigma_k)-\BGtwored(\sigma_\ell,\sigma_k)\right)\left(\Br_\ell\otimes\BI_m\right),
            \end{split} \\
            \begin{split}
                \label{eq:MixedHermiteBiTangential}
                0 &= \Bl_k^{\trans}\left(\frac{d}{ds}\BG_1(\sigma_k)-\frac{d}{ds}\BGonered(\sigma_k)\right)\Br_k\\
                &~~~~~~~~~+\sum_{\ell=1}^r\Bq_{k,\ell}^{\trans} \left(\frac{\partial}{\partial s_1}\BG_2(\sigma_k,\sigma_\ell)-\frac{\partial}{\partial s_1}\BGtwored(\sigma_k,\sigma_\ell)\right)\left(\Br_k\otimes \Br_\ell\right) \\
                &~~~~~~~~~+\sum_{\ell=1}^r\Bq_{\ell,k}^{\trans} \left(\frac{\partial}{\partial s_2}\BG_2(\sigma_\ell,\sigma_k)-\frac{\partial}{\partial s_2}\BGtwored(\sigma_\ell,\sigma_k)\right)\left(\Br_\ell\otimes\Br_k\right),
            \end{split}
        \end{align}
        \end{subequations}
        for all $j,k=1,\ldots,r$.
\end{theorem}

\begin{proof}[Proof of \Cref{thm:enforceMixedInterp}]
\allowdisplaybreaks
    The construction of $\BV$ alone yields the conditions~\cref{eq:linRightTangential} and \cref{eq:quadRightTangential}. A proof of this fact can be found in~\cite[Cor.~1]{DiazHGA23} and it is thereby omitted here.
    Define $\Bvarphi(s)\defeq s\BE-\BA$ and $\wtBvarphi(s)\defeq s\BEr-\BAr$.
    To prove the remaining conditions~\cref{eq:MixedLeftTangential} and~\cref{eq:MixedHermiteBiTangential}, we first derive two intermediate identities \cref{eq:interpId1} and~\cref{eq:interpId2}.
    By the construction of $\BV\in\Cnr$ in~\cref{eq:interpBasisVr} and the assumption that $\BVr$ is full rank, there exists $\wtBv_k\in\Cr$ so that $\BVr\wtBv_k=\Bv_k=\Bvarphi(\sigma_k)^{-1}\BB\Br_k$ and
    \begin{align}
        \nonumber
        \wtBvarphi(\sigma_k)\wtBv_k
        &=\BWr^{\trans}\Bvarphi(\sigma_k)\BVr\wtBv_k=\BW^{\trans}\Bvarphi(\sigma_k)\Bvarphi(\sigma_k)^{-1}\BB\Br_k=\BBr\Br_k\\
        \label{eq:interpId1}
        \Rightarrow~~\wtBv_k&=\wtBvarphi(\sigma_k)^{-1}\BBr\Br_k.
    \end{align}
    Similarly, by the construction of $\BW\in\Cnr$ in~\cref{eq:interpBasisWr} and the assumption that $\BWr$ is full rank, there exists $\wtBw_k\in\Cr$ so that $\BW\wtBw_k=\Bw_k$ and 
    \begin{align*}
        \wtBw_k^{\trans}\wtBvarphi(\sigma_k)&=\wtBw_k^{\trans}\BW^{\trans}\Bvarphi(\sigma_k)\BV=\Bl_k^{\trans}\BC\Bvarphi(\sigma_k)^{-1}\Bvarphi(\sigma_k)\BV+ 2\sum_{\ell=1}^r\Bq_{k,\ell}^{\trans}\BN_{\ell}^{\trans}\Bvarphi(\sigma_k)^{-1}\Bvarphi(\sigma_k)\BV\\
        \Rightarrow~~\wtBw_k^{\trans}&=\Bl_k^{\trans}\BCr\wtBvarphi(\sigma_k)^{-1}+2\sum_{\ell=1}^r\Bq_{k,\ell}^{\trans}\BN_{\ell}^{\trans}\BV\wtBvarphi(\sigma_k)^{-1} .
    \end{align*}
    To simplify this expression further, we note that, for each $\ell=1,\ldots,r$ and $j=1,\ldots,p$
    \begin{align*}
        \Bv_{\ell}^{\trans}\BM_j\BV\wtBvarphi(\sigma_k)^{-1}&=\wtBv_{\ell}^{\trans}\BV^{\trans} \BM_j\BV\wtBvarphi(\sigma_k)^{-1}=\Br_{\ell}^{\trans}\BBr^{\trans}\wtBvarphi(\sigma_{\ell})^{-\trans}\BMr_j\wtBvarphi(\sigma_k)^{-1}\\
        &=\vecm\left(\BMr_j\right)^{\trans}\left(\wtBvarphi(\sigma_k)^{-1}\otimes\wtBvarphi(\sigma_{\ell})^{-1}\BBr\Br_{\ell}\right)~\,\mbox{by~\cref{eq:interpId1} and~\cref{eq:vecKronId}.}
    \end{align*}
    Thus, by definition of $\BN_\ell$ and $\BMr$ according to~\cref{eq:altQO}, it follows that
    \begin{align*}
        \BN_{\ell}^{\trans}\BV\wtBvarphi(\sigma_k)^{-1}=\begin{bmatrix}
            \Bv_{\ell}^{\trans}\BM_1\BV\wtBvarphi(\sigma_k)^{-1}\\
            \vdots\\
            \Bv_{\ell}^{\trans}\BM_p\BV\wtBvarphi(\sigma_k)^{-1}
        \end{bmatrix}=
        \BMr
        \left(\wtBvarphi(\sigma_k)^{-1}\otimes\wtBvarphi(\sigma_{\ell})^{-1}\BBr\Br_{\ell}\right).
    \end{align*}
    Plugging the above expression into the one for $\wtBw_k$, we ultimately have that
    \begin{equation}
    \label{eq:interpId2}
        \wtBw_k^{\trans} = \Bl_k^{\trans}\BCr\wtBvarphi(\sigma_k)^{-1} + 2\sum_{\ell=1}^r\Bq_{k,\ell}^{\trans}\,\BMr\left(\wtBvarphi(\sigma_k)^{-1}\otimes\wtBvarphi(\sigma_\ell)^{-1}\BBr\Br_\ell\right).
    \end{equation}
    By construction of~\cref{eq:interpBasisWr}, an analogous expression holds for $\Bw_{k}^{\trans}$ (without the tildes).
    We are now prepared to prove that~\cref{eq:MixedLeftTangential} holds under the stated assumptions. 
    By \Cref{lemma:qoTfSymmetry} and the assumption that $\Bq_{\ell,k}=\Bq_{k,\ell}$ for all $\ell,k$, we have that $$\Bq_{k,\ell}^{\trans}\BGtwored(\sigma_k,\sigma_\ell)\left(\BI_m\otimes\Br_\ell\right)=\Bq_{\ell,k}^{\trans}\BGtwored(\sigma_\ell,\sigma_k)\left(\Br_\ell\otimes\BI_m\right)~~\mbox{for all}~~\ell,k=1,\ldots,r.$$
    A similar equality holds for $\BG_2$.
    Then, the conditions in~\cref{eq:MixedLeftTangential} simplify to
    \begin{equation}
    \label{eq:simplifiedMixedLeftTangential}
        \Bzero_m^{\trans} = \Bl_k^{\trans}\left(\BG_1(\sigma_k)-\BGonered(\sigma_k)\right)+2\sum_{\ell=1}^r\Bq_{k,\ell}^{\trans} \left(\BG_2(\sigma_k,\sigma_\ell)-\BGtwored(\sigma_k,\sigma_\ell)\right)\left(\BI_m\otimes \Br_\ell\right).
    \end{equation}
    It thus suffices to prove~\cref{eq:simplifiedMixedLeftTangential} to show~\cref{eq:MixedLeftTangential}.
    Applying~\cref{eq:mixedProdProperty} twice, it follows that 
    \begin{align*}
        \BGtwored(\sigma_k,\sigma_\ell)\left(\BI_m\otimes\Br_\ell\right)&=\BMr\left(\wtBvarphi(\sigma_k)^{-1}\BBr\otimes\wtBvarphi(\sigma_{\ell})^{-1}\BBr\right)\left(\BI_m\otimes\Br_\ell\right)\\
        &=\BMr\left(\wtBvarphi(\sigma_k)^{-1}\otimes\wtBvarphi(\sigma_{\ell})^{-1}\BBr\Br_\ell\right)\BBr.
    \end{align*}
    An analogous equality holds for $\BG_2(\sigma_k,\sigma_\ell)\left(\BI_m\otimes\Br_\ell\right)$. Substituting for these equalities directly in the right-hand side of~\cref{eq:simplifiedMixedLeftTangential} yields
    \begin{align*}
        &\Bl_k^{\trans}\left(\BG_1(\sigma_k)-\BGonered(\sigma_k)\right)+2\sum_{\ell=1}^r\Bq_{k,\ell}^{\trans} \left(\BG_2(\sigma_k,\sigma_\ell)-\BGtwored(\sigma_k,\sigma_\ell)\right)\left(\BI_m\otimes \Br_\ell\right)\\
        &=\Bl_k^{\trans}\left(\BC\Bvarphi(\sigma_k)^{-1}\BB-\BCr\wtBvarphi(\sigma_k)^{-1}\BBr\right)
        +2\sum_{\ell=1}^r\Bq_{k,\ell}^{\trans}\bigg(\BM\left(\Bvarphi(\sigma_k)^{-1}\otimes\Bvarphi(\sigma_{\ell})^{-1}\BB\Br_\ell\right)\BB\\
        &~~~~~~~~~~~~-\BMr\left(\wtBvarphi(\sigma_k)^{-1}\otimes\wtBvarphi(\sigma_{\ell})^{-1}\BBr\Br_\ell\right)\BBr\bigg)\\
        &=\left(\Bl_k^{\trans}\BC\Bvarphi(\sigma_k)^{-1} + 2\sum_{\ell=1}^r\Bq_{k,\ell}^{\trans}\BM\left(\Bvarphi(\sigma_k)^{-1}\otimes\Bvarphi(\sigma_\ell)^{-1}\BB\Br_\ell\right)\right)\BB\\
        &~~~~~~~~~~~~-\left(\Bl_k^{\trans}\BCr\wtBvarphi(\sigma_k)^{-1}+2\sum_{\ell=1}^r\Bq_{k,\ell}^{\trans}\,\BMr\left(\wtBvarphi(\sigma_k)^{-1}\otimes\wtBvarphi(\sigma_\ell)^{-1}\BBr\Br_\ell\right) \right)\BBr\\
        &=\Bw_k^{\trans}\BB-\wtBw_k^{\trans}\BBr=\wtBw_k^{\trans}\BW^{\trans}\BB-\wtBw_k^{\trans}\BBr=\Bzero_m^{\trans},
    \end{align*}
    where the final equality follows from~\cref{eq:interpId2} and the fact that $\Bw_k=\BW\wtBw_k$.
    
    Similarly, we show~\cref{eq:MixedHermiteBiTangential}.
    By~\Cref{lemma:qoTfSymmetry}, we have that
    \begin{equation*}
        \Bq_{k,\ell}^{\trans}\tfrac{\partial}{\partial s_1}\BGtwored(\sigma_k,\sigma_\ell)(\Br_k\otimes\Br_\ell)=\Bq_{\ell,k}^{\trans}\tfrac{\partial}{\partial s_2}\BGtwored(\sigma_\ell,\sigma_k)(\Br_\ell\otimes\Br_k)~~\mbox{for all}~~\ell,k=1,\ldots,r
    \end{equation*}
    and likewise for $\BG_2$. Thus, the conditions in~\cref{eq:MixedHermiteBiTangential} simplify to
    \begin{align}
    \begin{split}
    \label{eq:simplifiedMixedHermiteBitangential}
        0 &= \Bl_k^{\trans}\left(\frac{d}{ds}\BG_1(\sigma_k)-\frac{d}{ds}\BGonered(\sigma_k)\right)\Br_k\\
        &~~~~~~~~~+2\sum_{\ell=1}^r\Bq_{k,\ell}^{\trans} \left(\frac{\partial}{\partial s_1}\BG_2(\sigma_k,\sigma_\ell)-\frac{\partial}{\partial s_1}\BGtwored(\sigma_k,\sigma_\ell)\right)\left(\Br_k\otimes \Br_\ell\right).
    \end{split}
    \end{align}
    It suffices to prove~\cref{eq:simplifiedMixedHermiteBitangential} in place of~\cref{eq:MixedHermiteBiTangential}.
    The right-hand side of~\cref{eq:simplifiedMixedHermiteBitangential} can be rewritten as 
    \begin{align*}
        &-\Bl_k^{\trans}\left(\BC\Bvarphi(\sigma_k)^{-1}\BE\Bvarphi(\sigma_k)^{-1}\BB-\BCr\wtBvarphi(\sigma_k)^{-1}\BEr\wtBvarphi(\sigma_k)^{-1}\BBr\right)\Br_k\\
        &~~~~~~~~~-2\sum_{\ell=1}^r\Bq_{k,\ell}^{\trans}\bigg(\BM\left(\Bvarphi(\sigma_k)^{-1}\otimes\Bvarphi(\sigma_\ell)^{-1}\BB\Br_\ell\right)\BE
        \Bvarphi(\sigma_k)^{-1}\BB\\
        &~~~~~~~~~-\BMr\left(\wtBvarphi(\sigma_k)^{-1}\otimes\wtBvarphi(\sigma_\ell)^{-1}\BBr\Br_\ell\right)\BEr
        \wtBvarphi(\sigma_k)^{-1}\BBr\bigg)\Br_k\\
        &=\bigg(\Bl_k^{\trans}\BCr\wtBvarphi(\sigma_k)^{-1}+2\sum_{\ell=1}^r\Bq_{k,\ell}^{\trans}\,\BMr\left(\wtBvarphi(\sigma_k)^{-1}\otimes\wtBvarphi(\sigma_\ell)^{-1}\BBr\Br_{\ell}\bigg) \right)\BEr\wtBvarphi(\sigma_k)^{-1}\BBr\Br_k\\
        &~~~~~~~~~-\bigg(\Bl_k^{\trans}\BC\Bvarphi(\sigma_k)^{-1}+2\sum_{\ell=1}^r\Bq_{k,\ell}^{\trans}\BM\left(\Bvarphi(\sigma_k)^{-1}\otimes\Bvarphi(\sigma_\ell)^{-1}\BB\Br_\ell\right)\bigg){\BE}
        \Bvarphi(\sigma_k)^{-1}\BB\Br_k\\
        &= \wtBw_k^{\trans}\BEr\wtBv_k - \wtBw_k^{\trans}\BW^{\trans}\BE\BV\wtBv_k= 0.
    \end{align*}
    The last line follows by construction of $\Bv_k$ and $\Bw_k$, and by applying~\cref{eq:interpId2} and~\cref{eq:interpId1}. This completes the proof.
\end{proof} 

We call the bases $\BV$ and $\BW$ that satisfy the hypotheses of~\Cref{thm:enforceMixedInterp} \emph{interpolatory} model reduction bases.
In a vacuum,~\Cref{thm:enforceMixedInterp} offers a new strategy for the interpolatory model reduction of \LQO{} systems by imposing the mixed-multipoint tangential interpolation conditions in~\cref{eq:MixedLeftTangential} and~\cref{eq:MixedHermiteBiTangential}.
We note that, for the choice of $\Bq_{j,k}=\Bmred_{j,k}$, the symmetry hypothesis imposed on the left-tangential directions $\Bq_{j,k}$ by~\Cref{thm:enforceMixedInterp} is trivially satisfied due to~\cref{eq:qoResidueSymmetry}.
Thus, with regard to $\CH_2$-optimal model reduction, if we choose the interpolation data in~\Cref{thm:enforceMixedInterp} to be $\sigma_k=-\lambdared_k$, $\Br_k=\Bbred_k$, $\Bl_k=\Bcred_k$, and $\Bq_{j,k}=\Bmred_{j,k}$ (the poles and residue directions of a system that minimizes the $\CH_2$ model error) then the first-order optimality conditions from~\Cref{thm:lqoH2OptInterpolationCons} will be satisfied by the \ROM{}.
Of course, this assumes having access to the optimal reduced model. We resolve this circular causality issue in the next section. 

\begin{remark}
    One could modify \Cref{thm:enforceMixedInterp} to use general \emph{left} and \emph{right} interpolation points $\mu_k$ and $\sigma_k$ for which $\sigma_k \neq \mu_k$.
    In this case,~\eqref{eq:MixedLeftTangential} will involve $\mu_k$ and the mixed Hermite bitangential condition~\eqref{eq:MixedHermiteBiTangential} will no longer hold because its derivation requires that $\sigma_k = \mu_k$. This is analogous to the \LTI{} case; see~\cite[Thm.~3.3.1]{AntBG20}.
    We do not provide specific details because our proposed $\CH_2$-optimality framework requires choosing the same interpolation points in the construction of $\BV$ and $\BW$.
\end{remark}

%% file: lqoirka.tex
%%%%%%%%%%%%%%%%%%%%%%%%%%%%%%%%%%%%%%%%%%%%%%%%%%%%%%%%%%%%%%%%%%%%%%%%%%%%%%%%
\section{A computational framework for interpolatory optimal-$\CH_2$ approximation of linear quadratic-output systems}
\label{sec:CompFramework}
%%%%%%%%%%%%%%%%%%%%%%%%%%%%%%%%%%%%%%%%%%%%%%%%%%%%%%%%%%%%%%%%%%%%%%%%%%%%%%%%

As illustrated by~\Cref{thm:lqoH2OptInterpolationCons}, the optimal selection of interpolation points and tangential directions requires \emph{a priori} knowledge of an $\CH_2$-optimal reduced model, which is impractical. In this section, we introduce an algorithm for automatically determining the optimal interpolation data and enforcing the corresponding $\CH_2$-optimality conditions~\cref{eq:H2OptConds} in an iterative fashion.
We then discuss various practical aspects of the algorithm.

%%%%%%%%%%%%%%%%%%%%%%%%%%%%%%%%%%%%%%%%%%%%%%%%%%%%%%%%%%%%%%%%%%%%%%%%%%%%%%%%
\subsection{The iterative rational Krylov algorithm for optimal-$\CH_2$ approximation of linear quadratic-output systems}
\label{sec:lqoIrka}
%%%%%%%%%%%%%%%%%%%%%%%%%%%%%%%%%%%%%%%%%%%%%%%%%%%%%%%%%%%%%%%%%%%%%%%%%%%%%%%%

Because the optimal interpolation data depend explicitly on the unknown $\CH_2$-optimal reduced model, the optimality conditions in~\cref{eq:H2OptConds} cannot be enforced in a single projection step using the $\BV$ and $\BW$ in~\Cref{thm:enforceMixedInterp}. 
Instead, an iterative procedure is required to enforce these optimality conditions. This situation is conceptually similar to the purely linear $\CH_2$-optimal \MOR{} problem, and suggests a natural extension of the \emph{iterative rational Krylov algorithm} (\IRKA{}) from~\cite{GugAB08} to the $\CH_2$-optimal \MOR{} of \LQO{} systems. 

\begin{algorithm2e}[t!]
  \SetAlgoHangIndent{1pt}
  \DontPrintSemicolon
  \caption{Linear quadratic-output iterative rational Krylov algorithm (\LQOIRKA{}).}
  \label{alg:lqoIrka}

  \KwIn{$\BE,\BA,\BB,\BC,\BM_1,\ldots,\BM_p$ from~\cref{eq:lqoSys}, order $r$ ($1\leq r<n$,) tolerance $\tau>0$, max number of iterations $M\geq 1$, initial interpolation data $\sigma_1,\ldots,\sigma_r\in\C$,  $\Br_1,\ldots,\Br_r\in\Cm$, $\Bl_1,\ldots,\Bl_r\in\Cp$, $\Bq_{1,1},\ldots,\Bq_{r,r}\in\Cp$ closed under complex conjugation such that $\sigma_k\BE-\BA$ is invertible and $\Bq_{j,k}=\Bq_{k,j}$ for all $j,k=1,\ldots,r$.}
  \vspace{1mm}
  \KwOut{$\BEr,\BAr,\BBr,\BCr,\BMr_1,\ldots,\BMr_p$~--~state-space matrices of~\cref{eq:lqoSysRed}.}
  \vspace{1mm}
  Iteration count $i = 0$.\;
  \While{max change in $(\lambdared_k)>\tau$ and $i \leq M$}{
    Compute interpolatory model reduction bases $\BVr,\BWr\in\Rnr$ according to~\Cref{lemma:realValuedBases} such that, for each $k=1,\ldots,r$
    \begin{align*}
        \Bv_k= \left(\sigma_k\BE-\BA\right)^{-1}\BB\Br_k&\in\range\left(\BV\right),\\
        \left(\sigma_k\BE^{\trans}-\BA^{\trans}\right)^{-1}\left(2\sum_{\ell=1}^r\BN_{\ell}\,\Bq_{k,\ell} + \BC^{\trans}\Bl_k\right)&\in\range(\BW).
    \end{align*}
    \vspace{-4mm}\;
    Orthonormalize bases $\BVr \gets \text{orth}(\BVr)$ and $\BWr \gets \text{orth}(\BWr).$
    \;
    Compute reduced-order matrices by Petrov-Galerkin projection:
    \begin{equation*}
        \begin{aligned}
        &\BEr\gets\BW^{\trans}\BE\BV,~~&\BAr\gets&~\BW^{\trans}\BA\BV, ~~&\BBr\gets\BW^{\trans}\BB,\\
        &\BCr\gets\BC\BV,~~&\BMr_k\gets&~\BV^{\trans}\BMk\BV,~~&k=1,\ldots,p. 
        \end{aligned}
    \end{equation*}
    \vspace{-4mm}\;
    Compute $\lambdared_k\in\C$ and $\Bbred_k\in\Cm$, $\Bcred_k\in\Cp$, $\Bmred_{j,k}\in\Cp$ according to~\cref{eq:resComponents} from the eigendecomposition of $s\BEr-\BAr$; update the interpolation data
    \begin{equation*}
        \sigma_k\gets -\lambdared_k,~~~\Br_k\gets\Bbred_k,~~~\Bl_k\gets\Bcred_k,~~~\Bq_{j,k}\gets\Bmred_{j,k}.
    \end{equation*}
    \vspace{-4mm}\;
    Set $i\gets i + 1$.
  }
\end{algorithm2e}

The resulting computational procedure, which we present in~\Cref{alg:lqoIrka} and call the \emph{linear quadratic-output iterative rational Krylov algorithm} (\LQOIRKA{}), performs iteratively corrected interpolation using the model reduction bases in~\Cref{thm:enforceMixedInterp}.
Specifically, at each step, the interpolation points and tangential directions are taken from the poles and residue directions of the previous reduced model iterate; the $3r+r^2$ tangential interpolation conditions in~\cref{subeq:genInterpolationCons} are then enforced by Petrov-Galerkin projection using these data.
The algorithm repeats until the (relative) largest magnitude change in the reduced model poles between consecutive iterates falls below a user-specified tolerance. Thus, the interpolation-based $\CH_2$-optimality conditions in~\cref{eq:H2OptConds} will be satisfied up to this tolerance if~\Cref{alg:lqoIrka} converges.
Because the construction of $\BVr$ and $\BWr$ in~\cref{eq:interpBasisVr} and~\cref{eq:interpBasisWr} requires only shifted linear solves and sparse matrix calculations involving the full-order matrix operators, the proposed method is suitable for large-scale problems.

%%%%%%%%%%%%%%%%%%%%%%%%%%%%%%%%%%%%%%%%%%%%%%%%%%%%%%%%%%%%%%%%%%%%%%%%%%%%%%%%
\subsection{Practical refinements of~\Cref{alg:lqoIrka}}
%%%%%%%%%%%%%%%%%%%%%%%%%%%%%%%%%%%%%%%%%%%%%%%%%%%%%%%%%%%%%%%%%%%%%%%%%%%%%%%%
Briefly, we discuss some practical implementation details of~\Cref{alg:lqoIrka}.

%%%%%%%%%%%%%%%%%%%%%%%%%%%%%%%%%%%%%%%%%%%%%%%%%%%%%%%%%%%%%%%%%%%%%%%%%%%%%%%%
\subsubsection{Real-valued reduced models from complex-valued interpolation data}
%%%%%%%%%%%%%%%%%%%%%%%%%%%%%%%%%%%%%%%%%%%%%%%%%%%%%%%%%%%%%%%%%%%%%%%%%%%%%%%%
A primitive construction of the interpolatory bases $\BVr$ and $\BWr$ given by~\Cref{thm:enforceMixedInterp} will likely produce complex-valued \ROM{}s when complex-valued interpolation data are used, as in~\Cref{alg:lqoIrka}.
To circumvent this issue, we propose an alternative, real-valued construction of $\BVr$ and $\BWr$ for use in~\Cref{alg:lqoIrka}.
(In the subsequent result, we use \MATLAB{} notation to index the columns of a matrix.)

\begin{lemma}
    \label{lemma:realValuedBases}
    Assume that we have the following interpolation data that satisfy the hypotheses of~\Cref{thm:enforceMixedInterp}: distinct interpolation points $\sigma_1,\ldots,\sigma_r\in\C$, right-tangential directions $\Br_1,\ldots,\Br_r\in\Cm$, and left-tangential directions $\Bl_1,\ldots,\Bl_r\in\Cp$ and $\Bq_{1,1}, \ldots, \Bq_{r,r}\in\Cp$.
    Suppose that the interpolation points are arranged into complex conjugate pairs so that $\overline{\sigma}_k=\sigma_{k+1}$ or $\sigma_k$ is real-valued, and the corresponding tangential directions are arranged as follows:
    \begin{align}
    \begin{split}
    \label{eq:complexDataOrg}
        \overline{\Br}_k &= \begin{cases}
            \Br_{k+1} & \mbox{if}~~\overline{\sigma}_k=\sigma_{k+1}\\
            \Br_{k} & \mbox{else},
        \end{cases}~~~~
        \overline{\Bl}_k=\begin{cases}
            \Bl_{k+1} & \mbox{if}~~\overline{\sigma}_k=\sigma_{k+1}\\
            \Bl_{k} & \mbox{else},
        \end{cases}\\
        \overline{\Bq}_{j,k}&=\begin{cases}
            \Bq_{j+1,k+1} & \mbox{if}~~\overline{\sigma}_j=\sigma_{j+1},~~~\overline{\sigma}_k=\sigma_{k+1}\\
            \Bq_{j+1,k} & \mbox{if}~~\overline{\sigma}_j=\sigma_{j+1} ,~\,~\imag\left({\sigma_k}\right)= 0\\
            \Bq_{j,k+1} & \mbox{if}~~\imag\left({\sigma_j}\right)=0 ,~\overline{\sigma}_k=\sigma_{k+1}\\
            \Bq_{j,k} & \mbox{else},
        \end{cases}
    \end{split}
    \end{align}
    for every other $j,k.$
    Let $\Bv_k\in\Cn$ and $\Bw_k\in\Cn$ be defined as in~\cref{eq:interpBasisVr} and~\cref{eq:interpBasisWr}.
    Suppose that the matrices $\BV\in\Cnr$ and $\BW\in\Cnr$ are constructed as
    \begin{align}
        \begin{split}
            \label{eq:RVinterpBasisVr}
            \BV(\colon,k)&=\Bv_k,~~~~~~~~~~~~~~~~~~~~~~\,\mbox{if}~\imag\left(\sigma_k\right)=0,\\
            \BV(\colon,k\colon k+1)&=\begin{bmatrix}\real\left(\Bv_k\right) & \imag\left(\Bv_k\right)\end{bmatrix}~~\mbox{else},
        \end{split} \\
        \begin{split}
            \label{eq:RVinterpBasisWr}
           \BW(\colon,k)&=\Bw_k~~~~~~~~~~~~~~~~~~~~~~~\,\mbox{if}~\imag\left(\sigma_k\right)=0,\\
            \BW(\colon,k\colon k+1)&=\begin{bmatrix}\real\left(\Bw_k\right) & \imag\left(\Bw_k\right)\end{bmatrix}~\mbox{else},
        \end{split}
    \end{align}
    {for every other $k$.}
    Then, $\BV$ and $\BW$ are real valued, and it holds that
    \begin{equation*}
        \range\left(\BV\right)=\range\left(\BV_{\operatorname{p}}\right)~~\mbox{and}~~  \range\left(\BW\right)=\range\left(\BW_{\operatorname{p}}\right),   
    \end{equation*}
    where     
    $\BVp=\begin{bmatrix}
        \Bv_1 & \cdots & \Bv_r
    \end{bmatrix}\in\Cnr$ and $\BWp=\begin{bmatrix}
        \Bw_1 & \cdots & \Bw_r
    \end{bmatrix}\in\Cnr$.
\end{lemma}

\begin{proof}[Proof of~\Cref{lemma:realValuedBases}]
The proof is an intuitive extension of the known result for linear systems. One only needs to show that the columns of $\BWp$ and $\BVp$ are closed under complex conjugation.
Because of its similarity to the linear case, we omit it for brevity. The full details can be found in~\cite[Lemma~6.12]{Rei25}.
\end{proof}

\Cref{lemma:realValuedBases} shows how to compute real-valued interpolatory reduced models that satisfy the conditions in~\cref{subeq:genInterpolationCons} from a real-valued full-order model~\cref{eq:lqoSys} when complex interpolation data are used.
The organizational structure imposed upon the interpolation data in~\Cref{lemma:realValuedBases} is meant to mimic that of the interpolation data computed during~\Cref{alg:lqoIrka}, as well as the optimal interpolation data.
Indeed, for a reduced model~\cref{eq:lqoSysRed}, the eigenvalues $\lambdared_k\in\C$ and eigenvectors $\Btred_k,\Bsred_k\in\Cr$ computed from the generalized eigendecomposition of $\BEr$ and $\BAr$ are closed under complex conjugation, since these matrices are real valued. Thus, the eigenvalues and eigenvectors can be organized into conjugate eigenpairs according to~\cref{eq:complexDataOrg}.
One can verify directly that the residue directions~\cref{eq:resComponents} used for the interpolatory projections throughout~\Cref{alg:lqoIrka} obey the organizational scheme laid out in~\cref{eq:complexDataOrg}.

%%%%%%%%%%%%%%%%%%%%%%%%%%%%%%%%%%%%%%%%%%%%%%%%%%%%%%%%%%%%%%%%%%%%%%%%%%%%%%%%
\subsubsection{Convergence monitoring, unstable intermediate models, and initialization strategies}\label{sss:algoDetails}
%%%%%%%%%%%%%%%%%%%%%%%%%%%%%%%%%%%%%%%%%%%%%%%%%%%%%%%%%%%%%%%%%%%%%%%%%%%%%%%%
The iteration in~\Cref{alg:lqoIrka} repeats until either the iteration count exceeds a maximum number of allowed steps $M\geq 1$, or the largest magnitude (relative) change in the reduced model poles between consecutive iterates falls below a user-specified tolerance $\tau >0$.
$M$ is a budget parameter set by the user. For choosing $\tau$, our numerical experiments suggest that $\tau\sim O(10^{-4})$ is sufficient; see our discussion in~\Cref{sec:numericalResults}.
Although there are many possibilities for monitoring convergence, we choose to use the change in the poles because this guarantees that the optimality conditions in~\cref{eq:H2OptConds} will be satisfied if the iteration converges. (In fact, this quantity is typically used to monitor convergence in the traditional \IRKA{} iteration~\cite{GugAB08}.)
Moreover, this criterion is numerically efficient since the poles (and residues) of the current model iterate need to be computed regardless, to update the interpolation data for the next step.
Because \LQOIRKA{} aims to solve the $\CH_2$ minimization problem~\cref{eq:H2OptimalMORProb}, a natural alternative is to monitor the $\CH_2$ error throughout the iteration, and terminate once the change in the relative $\CH_2$ error falls below $\tau$. 
However, this would require one to pre-compute all eigenvalues and eigenvectors of the full-order problem to apply~\cref{eq:H2normPoleRes} to the error system, or solve a large-scale Lyapunov equation en route to computing the $\CH_2$ error using the formulae in~\cite{Reietal25,PrzDGB24}. 

As with the original \IRKA{} iteration, asymptotic stability is not guaranteed by~\Cref{alg:lqoIrka} but is typically maintained in practice. 
If an unstable intermediate model does appear, one can simply reflect the unstable pole across the imaginary axis to avoid interpolation at this point, and ensure the interpolatory first-order necessary conditions are satisfied upon convergence.
In our experiments, we have never observed that \LQOIRKA{} converges to an unstable reduced model given a stable initialization.

The initialization of~\Cref{alg:lqoIrka} corresponds to an appropriate selection of complex interpolation points and tangential directions, and will affect the quality of the final result model. 
Because the optimal interpolation points are the mirror images of the reduced model poles, and one would expect these to lie in the numerical range of $\BE^{-1}\BA$, choosing $r$ interpolation points in this region is usually an effective strategy.
We investigate \LQOIRKA{} under different initialization strategies in~\Cref{sec:numericalResults}.

%%%%%%%%%%%%%%%%%%%%%%%%%%%%%%
\begin{remark}
\label{remark:inexactSolves}
    In this section, we have implicitly assumed that \emph{direct} methods are used to solve the linear systems required to compute the interpolatory bases $\BV$ and $\BW$. For the \LTI{} case, Beattie et al.~\cite{BeaGW12} investigated the impact of (inexact) iterative solves on the resulting interpolatory \ROM{}s, and showed that employing a Petrov-Galerkin framework for the inexact solves yields a rational interpolant of a nearby full-order system, thus establishing a backward stability framework for interpolatory \MOR{}.
    It will be an interesting research direction to establish whether such a backward error result holds for the bases in~\Cref{thm:enforceMixedInterp} and interpolatory \LQO-\MOR{}.
\end{remark}
%%%%%%%%%%%%%%%%%%%%%%%%%%%%%%

\begin{remark}
The benchmarks that we consider in~\Cref{sec:numericalResults} have at most $m = 2$ inputs or $p = 2$ outputs.
For the linear \IRKA{}, the convergence behavior depends on $p$ and $m$; see, e.g.,~\cite[Ch.~5.2]{AntBG20}.
In our experiments, we have observed that the convergence of \LQOIRKA{} may slow considerably if $m$ or $p$ is large. The degree to which convergence slows depends on a variety of influences, such as the magnitude of $m$ and $p$ relative to $n$.
In the interest of space, we do not investigate here any examples with large $m$ or $p$, and leave such a rigorous convergence analysis for future work.
\end{remark}

%% file: numerics.tex
%%%%%%%%%%%%%%%%%%%%%%%%%%%%%%%%%%%%%%%%%%%%%%%%%%%%%%%%%%%%%%%%%%%%%%%%%%%%%%%%
\section{Numerical results}
\label{sec:numericalResults}
%%%%%%%%%%%%%%%%%%%%%%%%%%%%%%%%%%%%%%%%%%%%%%%%%%%%%%%%%%%%%%%%%%%%%%%%%%%%%%%%
We test the proposed~\Cref{alg:lqoIrka} on  two benchmark problems from the \MOR{} literature. 
All experiments were performed on a MacBook Air with 8 gigabytes of RAM and an Apple M2 processor running macOS Sequoia version 15.2 with MATLAB 23.2.0.2515942 (R2023b) Update 7.
The source codes for recreating the numerical experiments and results are available at~\cite{supRei26}.

%%%%%%%%%%%%%%%%%%%%%%%%%%%%%%%%%%%%%%%%%%%%%%%%%%%%%%%%%%%%%%%%%%%%%%%%%%%%%%%%
\subsection{Experimental setup}
\label{ss:setup}
%%%%%%%%%%%%%%%%%%%%%%%%%%%%%%%%%%%%%%%%%%%%%%%%%%%%%%%%%%%%%%%%%%%%%%%%%%%%%%%%
For \LQOIRKA{}, two choices for the initial interpolation data are tested to assess the iteration's robustness to different initializations:
% \vspace{.5\baselineskip}
\begin{description} 
    \item[\eigs{}] uses the (mirrored) poles and residue directions of an initial reduced model computed by Galerkin projection $\BV=\BW$, where $\BV\in\Rnr$ is the orthonormalized basis of the $r$-dimensional invariant subspace of $\BE^{-1}\BA$
    corresponding to the eigenvalues with smallest magnitude, which are obtained using \MATLAB{}'s \eigs{} command with a tolerance of $10^{-10}$ and the `$\mathsf{smallestabs}$' input option.
    \item[\imagsf{}] takes the initial interpolation points to be $r$ points of the form $\sigma_k=\imunit z_k$, where $z_k$ are $r/2$ logarithmically or linearly spaced points in an interval; these points are closed under complex conjugation. The tangential-directions are chosen to be the leading canonical basis vectors of dimension $r$.
\end{description}
% \vspace{.5\baselineskip}
Per our discussion in~\Cref{sss:algoDetails}, the change in the reduced model poles is used to monitor convergence in \LQOIRKA{}.
We compare \LQOIRKA{} in~\Cref{alg:lqoIrka} against two other \MOR{} strategies for computing \ROM{}s of the benchmark problems.
% \vspace{.5\baselineskip}
\begin{description}
    \item[\LQOBT{}] is the balanced truncation \MOR{} algorithm for \LQO{} systems from~\cite{BenGPD21};
    \item[\interpOneStep{}] computes a (one-step) interpolatory reduced model using $\BV\in\Rnr$ and $\BW\in\Rnr$ as in \Cref{thm:enforceMixedInterp} (\Cref{lemma:realValuedBases}) with non-optimal interpolation data. For these experiments, the data are chosen according to \eigs{} and \imagInit{}.
    We refer to \interpOneStep{} with these selection strategies as \interpOneStepEigs{} and \interpOneStepImag{}.
    In either case, \interpOneStep{} produces a reduced model that satisfies all the interpolation conditions of~\Cref{thm:enforceMixedInterp}, but for non-optimal interpolation data. Note that these two choices correspond to the initial interpolation data we use for  \LQOIRKA{}, thus to the first step of \LQOIRKA{}.
\end{description}
% \vspace{.5\baselineskip}
Our rationale for including \interpOneStep{} is to investigate the quality of \ROM{}s that satisfy the mixed-multipoint conditions of~\Cref{thm:enforceMixedInterp} using \emph{non-optimal} interpolation data, since these conditions have not previously appeared in the \LQO-\MOR{} literature.

We test the performance of the computed \LQO-\ROM{}s in recovering the (time-domain) output $\By$ of the full-order model being approximated for particular choices of inputs. 
Because the benchmarks we consider have either a single output or multiple linear- and quadratic-outputs that we analyze separately,
we write $y=\By$ below.
The time-domain simulations are implemented using \MATLAB{}'s $\mathsf{ode15s}$ using a fixed step size.
To visibly compare the performance of the reduced models, we plot the full- and reduced-order outputs, as well as their pointwise relative error given by
\begin{equation}
    \label{eq:pointwiseOutputError}
    \relerr(t_i)\defeq\frac{|y(t_i)-\yr(t_i)|}{|y(t_i)|},~~t_i\in[t_{{\min}},t_{{\max}}],
\end{equation}
where $t_i\in[t_{\min},t_{\max}]$ are the $N$ (equidistant) time steps in the simulation.
To assess the worst-case performance of the \ROM{}s over the simulation window, we use an approximation of the relative $\CL_\infty$ error:
\begin{equation}
    \label{eq:relLinftyerror}
    \relerr_{\CL_\infty}\defeq\max_{t_i}\frac{|y(t_i)-\yr(t_i)|}{|y(t_i)|}.
\end{equation}
To assess the average performance of the reduced models over the simulation window, we use an approximation of the relative $\CL_2$ error:
\begin{equation}
    \label{eq:relL2error}
    \relerr_{\CL_2}\defeq\left(\frac{\sum_{i=1}^N |y(t_i)-\yr(t_i)|^2}{\sum_{i=1}^N|y(t_i)|^2}\right)^{1/2}.
\end{equation}
We also score the reduced model performance using the relative $\CH_2$ system error defined according to~\Cref{def:H2normTf}:
\begin{equation}
    \label{eq:relH2error}
    \relerr_{\CH_2}\defeq \frac{\|\Sys-\Sysred\|_{\CH_2}}{\|\Sys\|_{\CH_2}}.
\end{equation}

%%%%%%%%%%%%%%%%%%%%%%%%%%%%%%%%%%%%%%%%%%%%%%%%%%%%
\subsection{1D advection-diffusion equation with a quadratic cost}
\label{ss:advecdiff_example}
%%%%%%%%%%%%%%%%%%%%%%%%%%%%%%%%%%%%%%%%%%%%%%%%%%%%
We first consider a
1D advection-diffusion equation from~\cite[Section~4.1]{DiazHGA23}.
The governing equations are written as
\begin{align}
\begin{split}
\label{eq:advec}
    \frac{\partial}{\partial t} v(t,x)-\alpha \frac{\partial^2}{\partial x^2} v(t,x) + \beta \frac{\partial}{\partial x} v(t,x)&=0,\\
    v(t,0)=u_0(t), \ \ \alpha \frac{\partial}{\partial x}v(t,1)=u_1(t), \ \ 
    v(0,x)&=0,
\end{split}
\end{align}
for $x\in(0,1)$ and $t\in(0,T)$ and inputs $u_0,u_1\in\CL_2(0,T)$. The diffusion and advection coefficients are chosen as $\alpha=1$ and $\beta=1$.
The output that we consider is
\begin{align}
\label{eq:quadCost}
    \frac{1}{2}\int_0^1|v(t,x)-1|^2@\ds x.
\end{align}
Such an observable output may arise from, e.g., the objective cost function in an optimal control problem.
Discretizing the equations in~\cref{eq:advec} using $n+1$ equidistant spatial points yields an order-$n$ state-space model of the form~\cref{eq:lqoSys} with $m=2$ inputs $u_0$, $u_1$, and $p=1$ output $y$. Let $\Bx(t)\in\Rn$ denote the spatial discretization of $v(t,x)$,  $h=1/n$, and $\mathbf{1}_n\in\Rn$ be the $n$-dimensional vector consisting of all ones.
Then, the discretization provides an approximation to the quadratic cost function~\cref{eq:quadCost} 
\begin{equation*}
    \frac{h}{2}\|\Bx(t)-\mathbf{1}\|_2^2=-h@\mathbf{1}_n^{\trans}\Bx(t)+ \frac{h}{2}\vecm\left(\BI_n\right)^{\trans} \left(\Bx(t)\otimes\Bx(t)\right)+ \frac{h}{2}\|\mathbf{1}_n\|_2^2= y(t) + \frac{h}{2}\|\mathbf{1}_n\|_2^2.
\end{equation*}
To fit the framework of~\cref{eq:lqoSys}, we consider the single output of the discretized system to be given by $y(t)=\BC\Bx(t) +\BM\left(\Bx(t)\otimes\Bx(t)\right)$ for $\BC=-h@\mathbf{1}_n^{\trans}\in\R^{1\times n}$ and $\BM=\frac{h}{2}\vecm\left(\BI_n\right)^{\trans} \in\R^{1\times n^2}$, where $\BI_n$ is the $n\times n$ identity matrix. The approximation to the cost~\cref{eq:quadCost} is recovered from the output $y(t)$ via $\frac{h}{2}\|\Bx(t)-\mathbf{1}_n\|_2^2=y(t) + \frac{h}{2}\|\mathbf{1}_n\|_2^2.$

To obtain an \LQO{} system in state-space form~\cref{eq:lqoSys} from~\cref{eq:advec}, an upwind finite-difference discretization of~\cref{eq:advec} is performed using $n+1=3\,001$ spatial grid points.
For this example, $\BE=\BI_n$ by construction.

%%%%%%%%%%%%%%%%%%%%%%%%%%%%%%%%%%%%%%%%%%%%%%%%%%%%%%%%%%%%%%%%%%%%%%%%%%%%%%%%
\subsubsection{Discussion of the results}
\label{sss:advecDiscussion}
%%%%%%%%%%%%%%%%%%%%%%%%%%%%%%%%%%%%%%%%%%%%%%%%%%%%%%%%%%%%%%%%%%%%%%%%%%%%%%%%
Before comparing the performance of the reduced models, we investigate the convergence behavior of \LQOIRKA{} using the initialization strategies \eigs{} and \imagsf{}. This experiment also serves to examine the robustness of \LQOIRKA{} with respect to the different initialization strategies.
We set $r=30$, the tolerance to $\tau=10^{-10}$, and the maximum number of allowed iterations to $M=200.$ Note that this convergence tolerance is smaller in magnitude than one would typically use in practice; we choose this to investigate the long-term convergence behavior of the iteration.
Although the change in the reduced model poles is used to monitor the convergence of \LQOIRKA{}, we also compute the relative $\CH_2$ error~\cref{eq:relH2error} throughout to investigate how this quantity evolves throughout.

%%%%%%%%%%%%%%%%%%%%%%%%%%%%%%%%%%%%%%%%%%%%%%%%%%%%%%%%%%%%%%%%
\begin{figure}[t!]
    \centering
    \begin{subfigure}[b]{.49\linewidth}
        \raggedleft
  \tikzexternalenable%
  \tikzsetnextfilename{advecdiff_conv_errors}%
  \input{graphics/advecdiff_conv_errors.tikz}%
  \tikzexternaldisable%

    \caption{
    Relative $\CH_2$ errors of the intermediate order $r=30$ \ROM{}s of the advection diffusion example computed during the first $50$ \LQOIRKA{} iterations.
    }
    \label{fig:advecdiff_conv_errors}
    \end{subfigure}%
    % \\
    \hfill%
    \begin{subfigure}[b]{.49\linewidth}
        \raggedleft
  \tikzexternalenable%
  \tikzsetnextfilename{advecdiff_conv_poles}%
  \input{graphics/advecdiff_conv_poles.tikz}%
  \tikzexternaldisable%

    \caption{
    Max relative change in the poles of the intermediate order $r=30$ \ROM{}s of the advection diffusion example computed during the first $50$ \LQOIRKA{} iterations.
    }
    \label{fig:advecdiff_conv_poles}
    \end{subfigure}%

\vspace{.5\baselineskip}
\begin{center} 
  \tikzexternalenable%
  \tikzsetnextfilename{conv_legend}%
  \input{graphics/conv_legend.tikz}%
  \tikzexternaldisable%

\end{center}
\caption{Relative $\CH_2$ errors and max change in the poles of the intermediate order $r=30$ \ROM{}s of the advection-diffusion example computed during \LQOIRKAeigs{} and \LQOIRKAimag{} for the first $50$ iterations.}
\label{fig:advecdiff_conv}
\end{figure}
%%%%%%%%%%%%%%%%%%%%%%%%%%%%%%%%%%%%%%%%%%%%%%%%%%%%%%%%%%%%%%%%%

We plot the change in the relative $\CH_2$ errors and the max relative change in the \ROM{} poles throughout \LQOIRKAeigs{} and \LQOIRKAimag{} in~\Cref{fig:advecdiff_conv_errors} and~\Cref{fig:advecdiff_conv_poles}, respectively.
Both iterations converged within the maximally allowed number of steps prescribed by $M$. 
\LQOIRKAeigs{} and \LQOIRKAimag{} exhibit very similar convergence behavior; for each iteration, the relative $\CH_2$ error drops several orders of magnitude, and \LQOIRKAeigs{} and \LQOIRKAimag{} seem to identify the \emph{same} local minimum within the first fifteen iterations. 
On the other hand, \Cref{fig:advecdiff_conv_poles} shows that the reduced model poles continue to change well beyond this, despite the iteration having seemingly identified a local minimum.
Thus, while computing the relative $\CH_2$ error at every step is more computationally expensive, it is also a better indicator of convergence of the method. 
In aggregate, these observations suggest that a relatively larger tolerance, e.g., $\tau=10^{-4}$, is sufficient for \LQOIRKA{} to identify a local minimum of the $\CH_2$ error.

% 
%%%%%%%%%%%%%%%%%%%%%%%%%%%%%%%%%%%%%%%%%%%%%%%%%%%%%%%%%%%%%%%%%
\begin{figure*}[t!]
    \centering
    \begin{subfigure}[b]{.49\linewidth}
        \raggedleft
  \tikzexternalenable%
  \tikzsetnextfilename{advecdiff_outputs_sin}%
  \input{graphics/advecdiff_outputs_sin.tikz}%
  \tikzexternaldisable%
\\
  \tikzexternalenable%
  \tikzsetnextfilename{advecdiff_outputerrors_sin}%
  \input{graphics/advecdiff_outputerrors_sin.tikz}%
  \tikzexternaldisable%
\\
    \caption{
    Output responses and pointwise relative errors~\cref{eq:pointwiseOutputError} of the full- and reduced-order models for $u_0(t)=0$ and $u_1(t)=u_{\sinc}(t)$.
    }
    \label{fig:r30advecdiff_output_sin}
    \end{subfigure}%
    % \\
    \hfill%
    \begin{subfigure}[b]{.49\linewidth}
        \raggedleft
  \tikzexternalenable%
  \tikzsetnextfilename{advecdiff_outputs_exp}%
  \input{graphics/advecdiff_outputs_exp.tikz}%
  \tikzexternaldisable%
\\
  \tikzexternalenable%
  \tikzsetnextfilename{advecdiff_outputerrors_exp}%
  \input{graphics/advecdiff_outputerrors_exp.tikz}%
  \tikzexternaldisable%
\\
    \caption{
    Output responses and pointwise relative errors~\cref{eq:pointwiseOutputError} of the full- and reduced-order models for $u_0(t)=0$ and $u_1(t)=u_{\exp}(t)$.
    }
    \label{fig:r30advecdiff_output_exp}
    \end{subfigure}%

  \vspace{.5\baselineskip}
  \tikzexternalenable%
  \tikzsetnextfilename{outputs_legend}%
  \input{graphics/outputs_legend.tikz}%
  \tikzexternaldisable%

    \label{fig:advecDiff_OutputErrors}
  \caption{
  Output responses and pointwise relative errors~\cref{eq:pointwiseOutputError} of the full-order and order $r=30$ reduced models of the advection-diffusion example driven by $u_1(t)=u_{\sinc}(t)$  and $u_1(t)=u_{\exp}(t)$ in~\cref{eq:inputs}.
  }
\end{figure*}
%%%%%%%%%%%%%%%%%%%%%%%%%%%%%%%%%%%%%%%%%%%%%%%%%%%%%%%%%%%%%%%%%

%%%%%%%%%%%%%%%%%%%%%%%%%%%%%%%%%%%%%%%%%%%%%%%%%%%%%%%%%%%%%%%%%
\begin{table*}[t!]
  \centering
  \resizebox{\linewidth}{!}{
  \begin{tabular}{llllllll}
    \hline\noalign{\smallskip}
      & \multicolumn{1}{c}{\LQOIRKAeigs{}}
      & \multicolumn{1}{c}{\LQOIRKAimag{}}
      & \multicolumn{1}{c}{\LQOBT{}}
      & \multicolumn{1}{c}{\interpOneStepEigs{}}
      & \multicolumn{1}{c}{\interpOneStepImag{}}\\
      \noalign{\smallskip}\hline\noalign{\smallskip}
      $\relerr_{\CL_\infty}$ ($u_{\sinc}$) 
      & $\boldsymbol{6.4135\texttt{e-}5}$
      & $6.4236\texttt{e-}5$
      & $2.4916\texttt{e-}4$
      & $5.5440\texttt{e-}2$
      & $2.5442\texttt{e}0$\\
        $\relerr_{\CL_\infty}$ ($u_{\exp}$) 
        & $\boldsymbol{5.9078\texttt{e-}6}$
      & $5.9037\texttt{e-}6$
      & ${1.7226\texttt{e-}4}$
      & ${1.6854\texttt{e-}2}$
      & ${1.8087\texttt{e}0}$\\
        $\relerr_{\CL_2}$ ~($u_{\sinc}$) 
      & $\boldsymbol{3.5477\texttt{e-}6}$
      & $3.5604\texttt{e-}6$
      & $4.2695\texttt{e-}5$
      & $9.4232\texttt{e-}3$
      & $4.7825\texttt{e-}1$\\
        $\relerr_{\CL_2}$ ~($u_{\exp}$) 
      & $5.4997\texttt{e-}7$
      & $\boldsymbol{5.4889\texttt{e-}7}$
      & $6.4736\texttt{e-}5$
      & $4.5368\texttt{e-}3$
      & $2.0120\texttt{e-}1$\\
        $\relerr_{\CH_2}$ 
      & $4.6226\texttt{e-}7$
      & $\boldsymbol{4.5150\texttt{e-}7}$
      & $7.7204\texttt{e-}7$
      & $9.9902\texttt{e-}1$
      & $9.5336\texttt{e-}1$\\
      \noalign{\smallskip}\hline\noalign{\smallskip}
  \end{tabular}
  }
\caption{Relative errors~\cref{eq:relLinftyerror}~--~\cref{eq:relH2error} for the order $r=30$ \ROM{}s of the advection-diffusion example. The smallest error for each metric is highlighted in \textbf{boldface}.}
  \label{tab:relErrors}
\end{table*} 
%%%%%%%%%%%%%%%%%%%%%%%%%%%%%%%%%%%%%%%%%%%%%%%%%%%%%%%%%%%%%%%%%

Next, we test the performance of the reduced models via time-domain simulations of their outputs.
Five order $r=30$ reduced models of the order $n=3\,000$ full-order model are computed using \LQOIRKAeigs{}, \LQOIRKAimag{}, \LQOBT{}, \interpOneStepEigs{}, and \interpOneStepImag{} according to~\Cref{ss:setup}.
Based on the previous convergence analysis, the convergence tolerance is set to $\tau
= 10^{-4} $ and the maximum number of allowed iterations is $M = 100$ for the \LQOIRKA{} iterations.
As before, each iteration converged within $M$ steps of the iteration where the change in the reduced model poles is used to monitor convergence.
Time-domain simulations are performed using two different pairs of input signals; in either case, we enforce the Dirichlet boundary condition of $u_0(t)=v(t,0)=0$.
The two different input signals used for $u_1$ are:
\begin{equation}
\label{eq:inputs}
    u_{\sinc}(t)\defeq5\frac{\sin(\pi t)}{\pi t}~~\mbox{and}~~u_{\exp}(t)\defeq e^{-t/5}\sin(4\pi t)
\end{equation}
for $t\in[0,10]$. 
The magnitudes of the full- and reduced-order outputs in response to $u_1=u_{\sinc}$ and $u_1=u_{\exp}$, along with the associated relative pointwise errors, are plotted in~\Cref{fig:r30advecdiff_output_sin} and~\Cref{fig:r30advecdiff_output_exp}, respectively. 
The relative $\CL_\infty$, $\CL_2$, and $\CH_2$ error measures in~\cref{eq:relLinftyerror},~\cref{eq:relL2error} and~\cref{eq:relH2error} induced by the reduced models are reported in~\Cref{tab:relErrors}. 
We observe that the \LQOIRKA{} and \LQOBT{} reduced models all produce high-fidelity approximations to the full-order output for both choices of $u_1$.
While \interpOneStepEigs{} offers a reasonable approximation, \interpOneStepImag{} misses the output entirely in both cases.
Overall, the \LQOIRKA{} reduced models produce approximations that are a few orders of magnitude better than those produced by the \LQOBT{} and \interpOneStep{} reduced models.
For the relative errors in~\Cref{tab:relErrors}, the \LQOIRKA{} reduced models record the smallest values in each measure.

%%%%%%%%%%%%%%%%%%%%%%%%%%%%%%%%%%%%%%%%%%%%%%%%%%%%%%%%%%%%%%%%%
\begin{table}[t!]
  \centering
  \resizebox{\linewidth}{!}{
  \begin{tabular}{llllll}
    \hline\noalign{\smallskip}
      & \multicolumn{1}{c}{\LQOIRKAeigs{}}
      & \multicolumn{1}{c}{\LQOIRKAimag{}}
      & \multicolumn{1}{c}{\LQOBT{}}
      & \multicolumn{1}{c}{\interpOneStepEigs{}}
      & \multicolumn{1}{c}{\interpOneStepImag{}}
      \\
      \noalign{\smallskip}
      \hline\noalign{\smallskip}
      Run time (s) 
      & $0.79$\,s
      & $1.00$\,s
      & $35.42$\,s
      & $0.09$\,s
      & $0.06$\,s\\
     Iteration count & $34$ & $44$ & N/A & N/A & N/A \\
      \noalign{\smallskip}\hline\noalign{\smallskip}
  \end{tabular}
}
\caption{Run times and iteration counts for computing the order $r=30$ reduced models of the advection diffusion example. The convergence parameters used for \LQOIRKA{} are $\tau=10^{-4}$ and $M=100$.}
\label{table:runtimes}
\end{table} 
%%%%%%%%%%%%%%%%%%%%%%%%%%%%%%%%%%%%%%%%%%%%%%%%%%%%%%%%%%%%%%%%%
The timings required for computing the reduced models are reported in Table~\ref{table:runtimes}.
As expected, the non-iterative (interpolation-based) methods \interpOneStepEigs{} and \interpOneStepImag{} are very fast since they only require solving $2r$ sparse linear systems. 
The \LQOIRKA{} reduced models are computed an order of magnitude faster than the \LQOBT{} reduced model, and roughly $1$ second slower than the one-shot interpolatory reduced models.
Most of the time spent by \LQOBT{} is in solving the two full-order Lyapunov equations that are necessary for the method.

%%%%%%%%%%%%%%%%%%%%%%%%%%%%%%%%%%%%%%%%%%%%%%%%%%%%%%%%%%%%%%%%%
\begin{figure}[t]
\centering
\raggedleft
  \tikzexternalenable%
  \tikzsetnextfilename{advecdiff_H2errors}%
  \input{graphics/advecdiff_H2errors.tikz}%
  \tikzexternaldisable%

\vspace{.5\baselineskip}
\begin{center}
  \tikzexternalenable%
  \tikzsetnextfilename{H2errors_legend}%
  \input{graphics/H2errors_legend.tikz}%
  \tikzexternaldisable%

\end{center}
\caption{Relative $\CH_2$ errors~\cref{eq:relH2error} due to the hierarchy of reduced models for orders $r = 2, 4, \ldots, 30$ of the advection-diffusion example. }
\label{fig:relH2errors}
\end{figure}
%%%%%%%%%%%%%%%%%%%%%%%%%%%%%%%%%%%%%%%%%%%%%%%%%%%%%%%%%%%%%%%%%

As a final experiment, we compute hierarchies of \ROM{}s for $r = 2, 4, \ldots, 30$ using \LQOIRKAeigs{}, \LQOIRKAimag{}, and \LQOBT{}.
We compute the relative $\CH_2$ errors due to these approximations and plot them with respect to the increasing order $r$ in~\Cref{fig:relH2errors}.
The same experiment was performed for \interpOneStep{}, and the computed reduced models all produced large relative $\CH_2$ errors. (We do not report these results here.)
For the \LQOIRKA{} and \LQOBT{} reduced models, the relative $\CH_2$ error steadily decreases as the approximation order increases. 
The \LQOIRKA{} reduced models exhibit the smallest relative $\CH_2$ error for each order, although only marginally so for orders $r\geq 10$.
\Cref{fig:relH2errors} also indicates the different initialization strategies \LQOIRKAeigs{} and \LQOIRKAimag{} converge to the same local minimum for each order of reduction. 

%%%%%%%%%%%%%%%%%%%%%%%%%%%%%%%%%%%%%%%%%%%%%%%%%%%%
\subsection{Stochastic Galerkin model of a mass-spring-damper configuration}
\label{ss:galerkin}
%%%%%%%%%%%%%%%%%%%%%%%%%%%%%%%%%%%%%%%%%%%%%%%%%%%%
The second benchmark we consider is the stochastic Galerkin model of a mass-spring-damper configuration from~\cite[Sec.~5.1]{Pul18}.
The system is derived from a parametric, coupled mass-spring-damper system comprising $4$ masses connected by $6$ springs and $4$ dampers, with $14$ physical parameters.
Using the stochastic Galerkin approach developed in~\cite{Pul18}, the parameters are modeled as independent random variables,
leading to a $n=5\,440$ dimensional (non-parametric) single-input, multiple-output (\SIMO{}) \LTI{} system.
The single input is the excitation of the bottom spring.
The first output $\expec_1\colon[0,\infty)\to\R$ is an approximation of the expected value of the position of the mass at the top of the configuration at time $t\geq 0$; the remaining outputs collectively produce an approximation $\var_1\colon[0,\infty)\to\R$ to the variance of this mass's position.
For the full details of the modeling approach, we refer the reader to~\cite{Pul18}.

Following~\cite[Sec.~4.3]{PulA19}, we treat the stochastic Galerkin model instead as a \SIMO{} \LQO{} system with $p=2$ outputs.
The first output remains linear and is $\expec_1(t)$; the second output is quadratic, and models $\var_1(t)$ \emph{directly}, instead of via multiple linear outputs.
To fit~\cref{eq:lqoSys}, the matrices $\BC\in\R^{2\times n}$ and $\BM\in\R^{2\times n^2}$ are chosen so that 
\begin{equation*}
\By(t)=\BC\Bx(t)+\BM\left(\Bx(t)\otimes\Bx(t)\right)=\begin{bmatrix} \expec_1(t)\\[1ex]\Bzero_{n}^{\trans}
    \end{bmatrix}+\begin{bmatrix}\Bzero_{n^2}^{\trans}\\[1ex]\var_1(t)
    \end{bmatrix},~~~t\geq 0.
\end{equation*}
For the precise construction of $\BC$ and $\BM$, we refer to the code package~\cite{supRei26}.
 
%%%%%%%%%%%%%%%%%%%%%%%%%%%%%%%%%%%%%%%%%%%%%%%%%%%%
\subsubsection{Discussion of the results}
\label{sss:galerkinDiscussion}
%%%%%%%%%%%%%%%%%%%%%%%%%%%%%%%%%%%%%%%%%%%%%%%%%%%%

%%%%%%%%%%%%%%%%%%%%%%%%%%%%%%%%%%%%%%%%%%%%%%%%%%%%%%%%%%%%%%%%%
\begin{figure*}[t!]
    \centering
    \begin{subfigure}[b]{.49\linewidth}
        \raggedleft
  \tikzexternalenable%
  \tikzsetnextfilename{galerkin_lo}%
  \input{graphics/galerkin_lo.tikz}%
  \tikzexternaldisable%
\\
  \tikzexternalenable%
  \tikzsetnextfilename{galerkin_lo_errors}%
  \input{graphics/galerkin_lo_errors.tikz}%
  \tikzexternaldisable%
\\
    \caption{
   Linear output responses and pointwise relative errors~\cref{eq:pointwiseOutputError} of the full- and reduced-order models for $u(t)=u_{\sin}(t)$.
    }
    \label{fig:r100galerkin_lo}
    \end{subfigure}%
    % \\
    \hfill%
    \begin{subfigure}[b]{.49\linewidth}
        \raggedleft
  \tikzexternalenable%
  \tikzsetnextfilename{galerkin_qo}%
  \input{graphics/galerkin_qo.tikz}%
  \tikzexternaldisable%
\\
  \tikzexternalenable%
  \tikzsetnextfilename{galerkin_qo_errors}%
  \input{graphics/galerkin_qo_errors.tikz}%
  \tikzexternaldisable%
\\
    \caption{
    Quadratic output responses and pointwise relative errors~\cref{eq:pointwiseOutputError} of the full- and reduced-order models for $u(t)=u_{\sin}(t)$.
    }
    \label{fig:r100galerkin_qo}
    \end{subfigure}%

  \vspace{.5\baselineskip}
  \tikzexternalenable%
  \tikzsetnextfilename{galerkin_outputs_legend}%
  \input{graphics/galerkin_outputs_legend.tikz}%
  \tikzexternaldisable%

    \label{fig:galerkin_OutputErrors}
  \caption{
  Linear- and quadratic-output ($\expec_1$ and $\var_1$) responses and pointwise relative errors~\cref{eq:pointwiseOutputError} of the full-order and order $r=100$ reduced models of the stochastic Galerkin example driven by the input $u(t)=u_{\sin}(t)$.
  }
\end{figure*}
%%%%%%%%%%%%%%%%%%%%%%%%%%%%%%%%%%%%%%%%%%%%%%%%%%%%%%%%%%%%%%%%%

%%%%%%%%%%%%%%%%%%%%%%%%%%%%%%%%%%%%%%%%%%%%%%%%%%%%%%%%%%%%%%%%%
\begin{table*}[t!]
  \centering
    \resizebox{.6\linewidth}{!}{
    \begin{tabular}{lccc} %
      \hline\noalign{\smallskip}
        & \multicolumn{1}{c}{\LQOIRKAeigs{}}
        & \multicolumn{1}{c}{\LQOIRKAimag{}}
        & \multicolumn{1}{c}{\LQOBT{}}\\
      \noalign{\smallskip}\hline\noalign{\smallskip}
      $\relerr_{\CL_\infty}$ ($\expec_{1}$) 
        & $\boldsymbol{9.3188\texttt{e-}3}$
        & $2.3274\texttt{e-}2$
        & $1.1137\texttt{e-}2$\\
      $\relerr_{\CL_\infty}$ ($\var_{1}$) 
        & $5.3579\texttt{e-}3$
        & $\boldsymbol{4.2190\texttt{e-}3}$
        & $2.7215\texttt{e-}2$\\
      $\relerr_{\CL_2}$ ~($\expec_{1}$) 
        & $2.0612\texttt{e-}5$
        & $1.7362\texttt{e-}5$
        & $\boldsymbol{9.3407\texttt{e-}6}$\\
      $\relerr_{\CL_2}$~~($\var_{1}$) 
        & $2.9571\texttt{e-}5$
        & $\boldsymbol{1.6877\texttt{e-}5}$
        & $1.7295\texttt{e-}5$\\
      $\relerr_{\CH_2}$
        & $2.0819\texttt{e-}4$
        & $\boldsymbol{1.8725\texttt{e-}4}$
        & $2.9025\texttt{e-}4$\\
      \noalign{\smallskip}\hline\noalign{\smallskip}
    \end{tabular}%
    }%
    
  \caption{Relative errors~\cref{eq:relLinftyerror}~--~\cref{eq:relH2error} for the order $r=100$ \ROM{}s of the stochastic Galerkin example. The smallest error for each metric is highlighted in \textbf{boldface}.}
  \label{tab:relErrors_gal}
\end{table*}
%%%%%%%%%%%%%%%%%%%%%%%%%%%%%%%%%%%%%%%%%%%%%%%%%%%%%%%%%%%%%%%%%

We again test the performance of the computed \ROM{}s in reproducing the full-order $\expec_1$ and $\var_1$.
Three order $r=100$ \ROM{}s are computed using \LQOIRKAeigs{}, \LQOIRKAimag{}, and \LQOBT{}; the \interpOneStep{} approach is not included because it produced reduced models that were not asymptotically stable.
For the \LQOIRKA{} iterations, we set $\tau=10^{-4}$ and $M=100$; both iterations converged within these tolerances.
The \imagsf{} initialization exhibited similar convergence behavior to what is depicted in~\Cref{fig:advecdiff_conv}, although \LQOIRKAimag{} converged faster for this example, taking only $20$ iterations before the reduced model poles stopped changing within the prescribed $\tau$. 
On the other hand, for the \eigs{} initialization, the poles of the \LQOIRKAeigs{} iterates oscillated for the first $50$ iterations, before converging steadily after $57$ iterations.
For this example, the \LQOIRKAimag{}-\ROM{} took roughly the same amount of time to compute as the \LQOBT{}-\ROM{}, while the \LQOIRKAeigs{}-\ROM{} took the longest to compute.
In the interest of space, we do not plot the full convergence behavior or include the timings here; these results are available in the accompanying code package~\cite{supRei26}.
A single time-domain simulation is performed using the input $u_{\sin}(t)\defeq\sin(.2t)$ for $t\in[0,500].$
The reconstructed outputs $\expec_1$ and $\var_1$ under this forcing as well as the associated pointwise relative errors~\cref{eq:pointwiseOutputError} are plotted in
Figures~\ref{fig:r100galerkin_lo} and~\cref{fig:r100galerkin_qo}.
The error measures~\cref{eq:relLinftyerror}~--~\cref{eq:relH2error} are reported in~\Cref{tab:relErrors_gal}. 
We observe that the \LQOIRKA{}- and \LQOBT{}-\ROM{}s all produce quality approximations to $\expec_1$ and $\var_1$ for the input $u_{\sin}$.
Depending on the output and error measure, one \ROM{} may perform marginally better, e.g., \LQOIRKAimag{} reproduces $\var_{1}$ more accurately in both metrics, whereas \LQOIRKAeigs{} and \LQOBT{} reproduce $\expec_{1}$ more accurately in the relative $\CL_\infty$ and $\CL_2$ metrics, respectively.
For this example, the relative $\CH_2$ error in~\Cref{tab:relErrors_gal} suggests that \LQOIRKAeigs{} and \LQOIRKAimag{} identify distinct, but nearby, local minima of the $\CH_2$ error.

%% file: graphics/advecdiff_conv_errors.tikz
\begin{tikzpicture}[font = \plotfontsize]

  \begin{semilogyaxis}[%
    width  = .73\linewidth,
    height = .1\textheight,
    scale only axis,
    xmin = 1,
    xmax = 50,
    ymin = 1e-8,
    ymax = 1e3,
    xminorticks = true,
    yminorticks = true,
    xlabel = {iteration count $i$},
    ylabel = {relative $\CH_2$ error},
    ylabel style   = {yshift = -.3em},
    scaled x ticks = false,
    x tick label style = {/pgf/number format/1000 sep={\,}},
    y tick label style = {/pgf/number format/1000 sep={\,}},
    cycle list name    = convplotlist,
  ]

  \pgfplotsset{cycle list shift = 1}

\pgfplotstableread{graphics/data/AdvecDiff3000_IRKA_eigsInit_r30conv.dat}\tableINPUTdiag
  
    \foreach \y in {1}{
      \addplot+ table[x index = 0, y index = \y] {\tableINPUTdiag};
    }

\pgfplotstableread{graphics/data/AdvecDiff3000_IRKA_imagInit_r30conv.dat}\tableINPUT
  
    \foreach \y in {1}{
      \addplot+ table[x index = 0, y index = \y] {\tableINPUT};
    }

  \end{semilogyaxis}
\end{tikzpicture}

%% file: graphics/advecdiff_conv_poles.tikz
\begin{tikzpicture}[font = \plotfontsize]

  \begin{semilogyaxis}[%
    width  = .73\linewidth,
    height = .1\textheight,
    scale only axis,
    xmin = 1,
    xmax = 50,
    ymin = 1e-5,
    ymax = 1e7,
    xminorticks = true,
    yminorticks = true,
    xlabel = {iteration count $i$},
    ylabel = {change in poles},
    ylabel style   = {yshift = -.3em},
    scaled x ticks = false,
    x tick label style = {/pgf/number format/1000 sep={\,}},
    y tick label style = {/pgf/number format/1000 sep={\,}},
    cycle list name    = convplotlist,
  ]

  \pgfplotsset{cycle list shift = 1}

\pgfplotstableread{graphics/data/AdvecDiff3000_IRKA_eigsInit_r30conv.dat}\tableINPUTdiag
  
    \foreach \y in {2}{
      \addplot+ table[x index = 0, y index = \y] {\tableINPUTdiag};
    }

\pgfplotstableread{graphics/data/AdvecDiff3000_IRKA_imagInit_r30conv.dat}\tableINPUT
  
    \foreach \y in {2}{
      \addplot+ table[x index = 0, y index = \y] {\tableINPUT};
    }

  \end{semilogyaxis}
\end{tikzpicture}

%% file: graphics/conv_legend.tikz
\begin{tikzpicture}
  \begin{axis}[%
    hide axis,
    width  = 1mm,
    height = 1mm,
    scale only axis,
    xmin = 0,
    xmax = 1,
    ymin = 0,
    ymax = 1,
    legend columns = 3, 
    legend style   = {
      at     = {(0,0)},
      anchor = center,
      /tikz/every even column/.append style = {column sep = 0.2cm}},
    legend cell align  = {left},
    clip mode          = individual,
    cycle list name    = convplotlist]

    \pgfplotsset{cycle list shift = 1}
    \foreach \y in {1, 2, ..., 3}{
      \addplot+ coordinates{ (0, 0) };
    }
    \addlegendentry{\LQOIRKAeigs{}}
    \addlegendentry{\LQOIRKAimag{}}
  \end{axis}
\end{tikzpicture}

%% file: graphics/advecdiff_outputs_sin.tikz
\begin{tikzpicture}[font = \plotfontsize]
  \pgfplotstableread{graphics/data/AdvecDiff3000_sinc_r30_Outputs.dat}\tableINPUT
  
  \begin{axis}[%
    width  = .73\linewidth,
    height = .1\textheight,
    scale only axis,
    xmin = 0,
    xmax = 10,
    ymin = 0,
    ymax = 3.5,
    xminorticks = true,
    yminorticks = true,
    xlabel = {time $t$ (s)},
    ylabel = {$y(t)$},
    ylabel style   = {yshift = -.3em},
    scaled x ticks = false,
    x tick label style = {/pgf/number format/1000 sep={\,}},
    y tick label style = {/pgf/number format/1000 sep={\,}},
    cycle list name    = plotlist
  ]

    \foreach \y in {1, 2, 3, 4, 5, 6}
    {
      \addplot+ table[x index = 0, y index = \y] {\tableINPUT};
    }
  \end{axis}
\end{tikzpicture}

%% file: graphics/advecdiff_outputerrors_sin.tikz
\begin{tikzpicture}[font = \plotfontsize]
    \pgfplotstableread{graphics/data/AdvecDiff3000_sinc_r30_OutputErrors.dat}\tableINPUT
  
  \begin{semilogyaxis}[%
    width  = .73\linewidth,
    height = .1\textheight,
    scale only axis,
    xmin = 0,
    xmax = 10,
    ymin = 1e-9,
    ymax = .5*1e2,
    xminorticks = true,
    yminorticks = true,
    xlabel = {time $t$ (s)},
    ylabel = {$\relerr(t)$},
    ylabel style   = {yshift = -.3em},
    scaled x ticks = false,
    x tick label style = {/pgf/number format/1000 sep={\,}},
    y tick label style = {/pgf/number format/1000 sep={\,}},
    cycle list name    = plotlist
  ]

  \pgfplotsset{cycle list shift = 1}
  
    \foreach \y in {1, 2, 3, 4, 5}
    {
      \addplot+ table[x index = 0, y index = \y] {\tableINPUT};
    }
  \end{semilogyaxis}
\end{tikzpicture}

%% file: graphics/advecdiff_outputs_exp.tikz
\begin{tikzpicture}[font = \plotfontsize]
  \pgfplotstableread{graphics/data/AdvecDiff3000_exponential_r30_Outputs.dat}\tableINPUT
  
  \begin{axis}[%
    width  = .73\linewidth,
    height = .1\textheight,
    scale only axis,
    xmin = 0,
    xmax = 10,
    ymin = .1,
    ymax = .9,
    xminorticks = true,
    yminorticks = true,
    xlabel = {time $t$ (s)},
    ylabel = {$y(t)$},
    ylabel style   = {yshift = -.3em},
    scaled x ticks = false,
    x tick label style = {/pgf/number format/1000 sep={\,}},
    y tick label style = {/pgf/number format/1000 sep={\,}},
    cycle list name    = plotlist
  ]
  
    \foreach \y in {1, 2, 3, 4, 5, 6}
    {
      \addplot+ table[x index = 0, y index = \y] {\tableINPUT};
    }
  \end{axis}
\end{tikzpicture}

%% file: graphics/advecdiff_outputerrors_exp.tikz
\begin{tikzpicture}[font = \plotfontsize]
    \pgfplotstableread{graphics/data/AdvecDiff3000_exponential_r30_OutputErrors.dat}\tableINPUT
  
  \begin{semilogyaxis}[%
    width  = .73\linewidth,
    height = .1\textheight,
    scale only axis,
    xmin = 0,
    xmax = 10,
    ymin = 1e-10,
    ymax = 1e1,
    xminorticks = true,
    yminorticks = true,
    xlabel = {time $t$ (s)},
    ylabel = {$\relerr(t)$},
    ylabel style   = {yshift = -.3em},
    scaled x ticks = false,
    x tick label style = {/pgf/number format/1000 sep={\,}},
    y tick label style = {/pgf/number format/1000 sep={\,}},
    cycle list name    = plotlist
  ]

  \pgfplotsset{cycle list shift = 1}

    \foreach \y in {1, 2, 3, 4, 5}
    {
      \addplot+ table[x index = 0, y index = \y] {\tableINPUT};
    }
  \end{semilogyaxis}
\end{tikzpicture}

%% file: graphics/outputs_legend.tikz
\begin{tikzpicture}
  \begin{axis}[%
    hide axis,
    width  = 1mm,
    height = 1mm,
    scale only axis,
    xmin = 0,
    xmax = 1,
    ymin = 0,
    ymax = 1,
    legend columns = 3, 
    legend style   = {
      at     = {(0,0)},
      anchor = center,
      /tikz/every even column/.append style = {column sep = 0.2cm}},
    legend cell align  = {left},
    clip mode          = individual,
    cycle list name    = plotlist]

    \foreach \y in {1,2,3,4,5,6}
    {
      \addplot+ coordinates{ (0, 0) };
    }
    \addlegendentry{\ensuremath{\mathsf{FOM}}}
    \addlegendentry{\LQOIRKAeigs{}}
    \addlegendentry{\LQOIRKAimag{}}
    \addlegendentry{\LQOBT{}}
    \addlegendentry{\interpOneStepEigs{}}
    \addlegendentry{\interpOneStepImag{}}

  \end{axis}
\end{tikzpicture}

%% file: graphics/advecdiff_H2errors.tikz
\begin{tikzpicture}[font = \plotfontsize]
  \pgfplotstableread{graphics/data/AdvecDiff3000_H2errors.dat}\tableINPUT
  
  \begin{semilogyaxis}[%
    width  = .85\linewidth,
    height = .1\textheight,
    scale only axis,
    xmin = 2,
    xmax = 30,
    ymin = 1e-7,
    ymax = 1e1,
    xminorticks = true,
    yminorticks = true,
    xlabel = {order $r$},
    ylabel = {relative $\CH_2$ error},
    ylabel style   = {yshift = -.3em},
    scaled x ticks = false,
    x tick label style = {/pgf/number format/1000 sep={\,}},
    y tick label style = {/pgf/number format/1000 sep={\,}},
    cycle list name    = errorplotlist
  ]
    \foreach \y in {1, 2, 3}
    {
      \addplot+ table[x index = 0, y index = \y] {\tableINPUT};
    }
  \end{semilogyaxis}
\end{tikzpicture}

%% file: graphics/H2errors_legend.tikz
\begin{tikzpicture}
  \begin{axis}[%
    hide axis,
    width  = 1mm,
    height = 1mm,
    scale only axis,
    xmin = 0,
    xmax = 1,
    ymin = 0,
    ymax = 1,
    legend columns = 3, 
    legend style   = {
      at     = {(0,0)},
      anchor = center,
      /tikz/every even column/.append style = {column sep = 0.2cm}},
    legend cell align  = {left},
    clip mode          = individual,
    cycle list name    = errorplotlist]

    \foreach \y in {1, 2, 3}{
      \addplot+ coordinates{ (0, 0) };
    }
    
    \addlegendentry{\LQOIRKAeigs{}}
    \addlegendentry{\LQOIRKAimag{}}
    \addlegendentry{\LQOBT{}}
  \end{axis}
\end{tikzpicture}

%% file: graphics/galerkin_lo.tikz
\begin{tikzpicture}[font = \plotfontsize]
  \pgfplotstableread{graphics/data/Galerkin5440_r100_LinOutputs.dat}\tableINPUT
  
  \begin{axis}[%
    width  = .73\linewidth,
    height = .1\textheight,
    scale only axis,
    xmin = 0,
    xmax = 500,
    ymin = -1,
    ymax = 1,
    xminorticks = true,
    yminorticks = true,
    xlabel = {time $t$ (s)},
    ylabel = {$\expec_1(t)$},
    ylabel style   = {yshift = -.3em},
    scaled x ticks = false,
    x tick label style = {/pgf/number format/1000 sep={\,}},
    y tick label style = {/pgf/number format/1000 sep={\,}},
    cycle list name    = plotlist
  ]

    \foreach \y in {1, 2, 3, 4}
    {
      \addplot+ table[x index = 0, y index = \y] {\tableINPUT};
    }
  \end{axis}
\end{tikzpicture}

%% file: graphics/galerkin_lo_errors.tikz
\begin{tikzpicture}[font = \plotfontsize]
    \pgfplotstableread{graphics/data/Galerkin5440_sin_r100_LinOutputErrors.dat}\tableINPUT
  
  \begin{semilogyaxis}[%
    width  = .73\linewidth,
    height = .1\textheight,
    scale only axis,
    xmin = 0,
    xmax = 500,
    ymin = 1e-10,
    ymax = 0,       
    xminorticks = true,
    yminorticks = true,
    xlabel = {time $t$ (s)},
    ylabel = {$\relerr(t)$},
    ylabel style   = {yshift = -.3em},
    scaled x ticks = false,
    x tick label style = {/pgf/number format/1000 sep={\,}},
    y tick label style = {/pgf/number format/1000 sep={\,}},
    cycle list name    = plotlist
  ]

  \pgfplotsset{cycle list shift = 1}
  
    \foreach \y in {1, 2, 3}
    {
      \addplot+ table[x index = 0, y index = \y] {\tableINPUT};
    }
  \end{semilogyaxis}
\end{tikzpicture}

%% file: graphics/galerkin_qo.tikz
\begin{tikzpicture}[font = \plotfontsize]
  \pgfplotstableread{graphics/data/Galerkin5440_r100_QuadOutputs.dat}\tableINPUT
  
  \begin{axis}[%
    width  = .73\linewidth,
    height = .1\textheight,
    scale only axis,
    xmin = 0,
    xmax = 500,
    ymin = 0,
    ymax = 4,
    xminorticks = true,
    yminorticks = true,
    xlabel = {time $t$ (s)},
    ylabel = {$\var_1(t)$},
    ylabel style   = {yshift = -.3em},
    scaled x ticks = false,
    x tick label style = {/pgf/number format/1000 sep={\,}},
    y tick label style = {/pgf/number format/1000 sep={\,}},
    cycle list name    = plotlist
  ]

    \foreach \y in {1, 2, 3, 4}
    {
      \addplot+ table[x index = 0, y index = \y] {\tableINPUT};
    }
  \end{axis}
\end{tikzpicture}

%% file: graphics/galerkin_qo_errors.tikz
\begin{tikzpicture}[font = \plotfontsize]
    \pgfplotstableread{graphics/data/Galerkin5440_sin_r100_QuadOutputErrors.dat}\tableINPUT
  
  \begin{semilogyaxis}[%
    width  = .73\linewidth,
    height = .1\textheight,
    scale only axis,
    xmin = 0,
    xmax = 500,
    ymin = 1e-10,
    ymax = 1e-3,
    xminorticks = true,
    yminorticks = true,
    xlabel = {time $t$ (s)},
    ylabel = {$\relerr(t)$},
    ylabel style   = {yshift = -.3em},
    scaled x ticks = false,
    x tick label style = {/pgf/number format/1000 sep={\,}},
    y tick label style = {/pgf/number format/1000 sep={\,}},
    cycle list name    = plotlist
  ]

  \pgfplotsset{cycle list shift = 1}
  
    \foreach \y in {1, 2, 3}
    {
      \addplot+ table[x index = 0, y index = \y] {\tableINPUT};
    }
  \end{semilogyaxis}
\end{tikzpicture}

%% file: graphics/galerkin_outputs_legend.tikz
\begin{tikzpicture}
  \begin{axis}[%
    hide axis,
    width  = 1mm,
    height = 1mm,
    scale only axis,
    xmin = 0,
    xmax = 1,
    ymin = 0,
    ymax = 1,
    legend columns = 4, 
    legend style   = {
      at     = {(0,0)},
      anchor = center,
      /tikz/every even column/.append style = {column sep = 0.2cm}},
    legend cell align  = {left},
    clip mode          = individual,
    cycle list name    = plotlist]

    \foreach \y in {1,2,3,4}
    {
      \addplot+ coordinates{ (0, 0) };
    }
    \addlegendentry{\ensuremath{\mathsf{FOM}}}
    \addlegendentry{\LQOIRKAeigs{}}
    \addlegendentry{\LQOIRKAimag{}}
    \addlegendentry{\LQOBT{}}

  \end{axis}
\end{tikzpicture}

%% file: conc.tex
\section{Conclusion}
\label{sec:conclusions}
We have presented a novel $\CH_2$-optimality framework for the approximation of \LQO{} systems~\cref{eq:lqoSys} based on multivariate rational interpolation. 
In~\Cref{thm:lqoH2OptInterpolationCons}, we derive first-order optimality conditions; these amount to the mixed-multipoint tangential interpolation of the linear- and quadratic-output transfer functions~\cref{eq:lqoTfs}, and generalize the analogous interpolatory $\CH_2$-optimality framework for the approximation of \LTI{} systems. 
We additionally show how to enforce the derived optimality conditions simultaneously by Petrov-Galerkin projection in~\Cref{thm:enforceMixedInterp}.
Finally, an iterative rational Krylov algorithm for \LQO{} systems (\LQOIRKA{}) is proposed in~\Cref{alg:lqoIrka}. Numerical examples illustrate the effectiveness of the proposed approach and its potential for treating large-scale problems.
Based on this work, there are several open research directions to consider. Namely, the effect of inexact solves used to compute the bases in~\Cref{thm:enforceMixedInterp} on the interpolatory \LQO{}-\ROM{}s, a more detailed convergence analysis for systems with many inputs and outputs, and the development of structured optimality frameworks are all promising directions.

%% file: appendices.tex
\appendix
%%%%%%%%%%%%%%%%%%%%%%%%%%%%%%%%%%%%%%%%
\section{Proof of~\Cref{thm:H2PoleRes}}
\label{app:proofPoleRes}
%%%%%%%%%%%%%%%%%%%%%%%%%%%%%%%%%%%%%%%%
\allowdisplaybreaks
We first prove a lemma that is used to bound certain quantities that appear in the proof of~\Cref{thm:H2PoleRes}.
\begin{lemma}
    \label{lemma:resolventBound}
    Let $\BG_2$ be the quadratic-output transfer function~\cref{eq:lqoQuadTf} of an \LQO{} system.
    Consider $z_1,z_2\in\C$ such that $z_1\BE-\BA$ and $z_2\BE-\BA$ are invertible. If $|z_i|>\|\BE^{-1}\BA\|_{2}$ for either $i=1$ or $i=2$,
    then $\BG_2(z_1,z_2)$ satisfies the bound
    \begin{equation}
    \label{eq:resolventBound}
        \|\BG_2(z_1,z_2)\|_{\frob}\leq \frac{C_i}{|z_i|-\|\BE^{-1}\BA\|_{2}}
        ~\,\mbox{for that}~\,i,
    \end{equation}
    where $C_i<\infty$ is a constant that does not depend on $z_i$.
\end{lemma}
\begin{proof}
    Assume $|z_1|>\|\BE^{-1}\BA\|_2$ without loss of generality.
    The bound~\cref{eq:resolventBound} follows by expanding $(z_1\BE-\BA)^{-1}$ in $\BG_2(z_1,z_2)$ as a Neumann series. Specifically, because $\|\BE^{-1}\BA\|_{2}/|z_1| < 1$, it follows that
    \begin{equation*}
        \left(z_1\BE-\BA\right)^{-1}=\frac{1}{z_1}\left(\BI-\frac{1}{z_1}\BE^{-1}\BA\right)^{-1}\BE^{-1}=\frac{1}{z_1}\sum_{k=0}^\infty\left(\frac{\BE^{-1}\BA}{z_1}\right)^k\BE^{-1}.
    \end{equation*}
    Taking the norm of both sides and evaluating the geometric series produces the bound
    \begin{align*}
        \|\left(z_1\BE-\BA\right)^{-1}\|_{2}\leq \frac{\|\BE^{-1}\|_2}{|z_1|}\sum_{k=0}^\infty\left(\frac{\|\BE^{-1}\BA\|_2}{|z_1|}\right)^k=\frac{\|\BE^{-1}\|_2}{|z_1|-\|\BE^{-1}\BA\|_2}.
    \end{align*}
    Using the fact that $\|\left(z_1\BE-\BA\right)^{-1}\|_{\frob}\leq\sqrt{n}\|\left(z_1\BE-\BA\right)^{-1}\|_2$, the bound~\cref{eq:resolventBound} follows:
    \begin{align*}
    \|\BG_2(z_1,z_2)\|_{\frob}&\leq\|\BM\|_{\frob}\|\left(z_1\BE-\BA\right)^{-1}\|_{\frob}\|\left(z_2\BE-\BA\right)^{-1}\|_{\frob}\|\BB\|_{\frob}^2\\
        &\leq\frac{\sqrt{n}\|\BE^{-1}\|_2\|\BM\|_{\frob}\|\left(z_2\BE-\BA\right)^{-1}\|_{\frob}\|\BB\|_{\frob}^2}{|z_1|-\|\BE^{-1}\BA\|_{2}}=\frac{C_1}{|z_1|-\|\BE^{-1}\BA\|_{2}},
    \end{align*}
    where $C_1$ does not depend on $z_1$, as claimed.
\end{proof}

We now proceed with the proof of~\Cref{thm:H2PoleRes}.
To begin, note that the first terms in~\cref{eq:H2ipTf} and~\cref{eq:H2ipPoleRes} are equal:
    \begin{equation*}
        \frac{1}{2\pi}\int_{-\infty}^\infty \trace\left(\BGonebar(-\imunit\omega)@ \BGonered(\imunit\omega)^{\trans}\right) \ds\omega=\sum_{i=1}^r \Bcred_i^{\trans}\BGonebar(-\lambdared_i)\Bbred_i.
    \end{equation*}
    This follows from classical results for calculating the Hardy $\CH_2$ inner product of two \LTI{} systems; see, e.g.,~\cite[Lemma~3.5]{GugAB08},~\cite[Lemma~2.1.4]{AntBG20}. 
    Then, to prove~\cref{eq:H2ipPoleRes} it suffices to prove the remaining equality
    \begin{equation}
    \label{eq:claim}
    \begin{aligned}
         \frac{1}{(2\pi)^2}\int_{-\infty}^\infty \int_{-\infty}^\infty \trace&\left(\BGtwobar(-\imunit\omega_1,-\imunit\omega_2)@ \BGtwored(\imunit\omega_1,\imunit\omega_2)^{\trans} \right) \ds\omega_1@ \ds\omega_2\\
         &=\sum_{j=1}^r\sum_{k=1}^r \Bmred_{j,k}^{\trans}\BGtwobar(-\lambdared_j,-\lambdared_k)(\Bbred_j\otimes\Bbred_k).
    \end{aligned}
    \end{equation}
    Let $R_1,R_2>0$ be arbitrarily chosen,
    and define the contours $\contourRi\subset \C$ as
    \begin{equation*}
        \contourRi\defeq [-\imunit R_i,\imunit R_i]\cup\{z=R_ie^{\imunit\theta}~\vert~ \pi/2\leq \theta\leq 3\pi/2\},~~i=1,2.
    \end{equation*}
    We may take $R_1$ and $R_2$ to be sufficiently large such that each contour encircles the poles of the reduced model. 
    Let $z\in\imunit\R$ be arbitrarily fixed and consider
    \begin{align*}
        \int_{\contourRone}\trace\left(\BGtwobar(-\zeta_1,-z)@ \BGtwored(\zeta_1,z)^{\trans} \right) \ds\zeta_1&=\imunit\int_{-R_1}^{R_1}\trace\left(\BGtwobar(-\imunit\omega_1,-z)@ \BGtwored(\imunit\omega_1,z)^{\trans} \right) \ds\omega_1\\
        +\,\imunit\int_{\pi/2}^{3\pi/2}\trace&\left(\BGtwobar(-R_1e^{\imunit\theta},-z)@ \BGtwored(R_1e^{\imunit\theta},z)^{\trans} \right) R_1e^{\imunit\theta} \, \ds\theta.
    \end{align*}
    Choose $R_1$ large enough so that $R_1>\|\overline{\BE}^{-1}\overline{\BA}\|_2, \|\BEr^{-1}\BAr\|_2$.
    Applying the Cauchy-Schwarz inequality in the Frobenius inner product and subsequently~\Cref{lemma:resolventBound} yields
    \begin{align*}
        \left|\trace\left(\BGtwobar(-R_1e^{\imunit\theta},-z)@ \BGtwored(R_1e^{\imunit\theta},z)^{\trans} \right)\right|
        &\leq \frac{C_1C_2}{(R_1-\|\overline{\BE}^{-1}\overline{\BA}\|_2)(R_1-\|\BEr^{-1}\BAr\|_2)}
    \end{align*}
    for constants $C_1,C_2<\infty$ that do not change with $R_1$; see the proof of~\Cref{lemma:resolventBound}.
    Using standard integral estimates, see, e.g.,~\cite[Ch.~IV, Eq.~(1.7)]{Gam03}, it follows that
    \begin{align*}
        \bigg|\int_{\pi/2}^{3\pi/2}\trace\left(\BGtwobar(-R_1e^{\imunit\theta},\right.&\left.-z)\BGtwored(R_1e^{\imunit\theta},z)^{\trans} \right) R_1e^{\imunit\theta} \ds\theta\bigg|\\
        &\leq \frac{\pi C_1 C_2 R_1}{(R_1-\|\overline{\BE}^{-1}\overline{\BA}\|_2)(R_1-\|\BEr^{-1}\BAr\|_2)}\longrightarrow 0~\,\mbox{as}~\,R_{1}\longrightarrow\infty.
    \end{align*}
    Because $R_1$ is arbitrary, we may take the limit as $R_1\to\infty$ to obtain the equality
    \begin{equation}
    \label{eq:contourInt1}
        \lim_{R_1\to\infty}\int_{\contourRone}\hspace{-4mm}\trace\left(\BGtwobar(-\zeta_1,-z) \BGtwored(\zeta_1,z)^{\trans} \right) \ds\zeta_1=\imunit\int_{-\infty}^\infty\hspace{-1mm}\trace\left(\BGtwobar(-\imunit\omega_1,-z)\BGtwored(\imunit\omega_1,z)^{\trans} \right)\ds\omega_1.
    \end{equation}
    Note that $\trace\left(\BGtwobar(-s_1,-z)@ \BGtwored(s_1,z)^{\trans} \right)$ is a scalar complex-valued function of the variable $s_1$ with poles at $-\mu_1,-\mu_2,\ldots,-\mu_n\in\C_+$ and \emph{simple poles} $\lambdared_1,\lambdared_2,\ldots,\lambdared_r\in\C_-$, where $\mu_i$ denotes the $i$-th eigenvalue of $\BE^{-1}\BA$. By the Residue Theorem~\cite[Ch.~VII]{Gam03}, we have that
    \begin{align*}
        \frac{1}{2\pi}\int_{-\infty}^\infty\trace&\left(\BGtwobar(-\imunit\omega_1,-z)@ \BGtwored(\imunit\omega_1,z)^{\trans} \right) \ds\omega_1\\
        &=\lim_{R_1\to\infty}\frac{1}{2\pi\imunit}\int_{\contourRone}\trace\left(\BGtwobar(-\zeta_1,-z)@ \BGtwored(\zeta_1,z)^{\trans} \right) \ds\zeta_1\\
        &=\sum_{j=1}^r\Res\left[\trace\left(\BGtwobar(-s_1,-z)@ \BGtwored(s_1,z)^{\trans}\right),s_1=\lambdared_j \right].
    \end{align*}
    Under the assumption that the poles $\lambdared_j$ are simple, for any fixed $z\in\imunit \R$ the residue of $\trace\left(\BGtwobar(-s_1,-z)@ \BGtwored(s_1,z)^{\trans}\right)$ at $s_1=\lambdared_j$ is given by
    \begin{align*}
        \Res&\left[\trace\left(\BGtwobar(-s_1,-z)@ \BGtwored(s_1,z)^{\trans}\right),s_1=\lambdared_j \right]\\
        &=\lim_{s_1\to\lambdared_j}(s_1-\lambdared_j)\trace\left(\BGtwobar(-s_1,-z)@ \BGtwored(s_1,z)^{\trans}\right)\\
        &= \trace\left(\BGtwobar(-\lambdared_j,-z)@ \lim_{s_1\to\lambdared_j}(s_1-\lambdared_j)\BGtwored(s_1,z)^{\trans}\right).
    \end{align*}
    Again, because the poles of $\Sysred$ are simple, $\BGtwored$ admits the pole-residue expansion in~\cref{eq:poleResidue}. Substituting in directly for~\cref{eq:poleResidue} yields:
    \begin{align*}
        \lim_{s_1\to\lambdared_j}(s_1-\lambdared_j)\BGtwored(s_1,z)^{\trans}
        &=\sum_{k=1}^r\frac{(\Bbred_j\otimes\Bbred_k)\Bmred_{j,k}^{\trans}}{z-\lambdared_k},~~\mbox{and thus}\\
        \Res\left[\trace\left(\BGtwobar(-s_1,-z)@ \BGtwored(s_1,z)^{\trans}\right),s_1=\lambdared_j \right]
        &= \trace\left(\BGtwobar(-\lambdared_j,-z)@ \sum_{k=1}^r\frac{(\Bbred_j\otimes\Bbred_k)\Bmred_{j,k}^{\trans}}{z-\lambdared_k}\right)
    \end{align*}
    for each $j=1,\ldots,r$.
    Substituting this into the previously computed contour integral~\cref{eq:contourInt1}, at last we have that
    \begin{align*}
        \frac{1}{2\pi}\int_{-\infty}^\infty\trace&\left(\BGtwobar(-\imunit\omega_1,-z)@ \BGtwored(\imunit\omega_1,z)^{\trans} \right) \ds\omega_1\\ &= \sum_{j=1}^r  \trace\left(\BGtwobar(-\lambdared_j,-z)@ \sum_{k=1}^r\frac{(\Bbred_j\otimes\Bbred_k)\Bmred_{j,k}^{\trans}}{z-\lambdared_k}\right)\\
        &=\sum_{j=1}^r \sum_{k=1}^r \Bmred_{j,k}^{\trans}\BGtwobar(-\lambdared_j,-z)(\Bbred_j\otimes\Bbred_k) \frac{1}{z-\lambdared_k},
    \end{align*}
    where the final equality follows from the fact that the trace operator $\trace\left(\cdot\right)$ is invariant under cyclic permutations, and that the trace of a scalar is just that scalar.
    Returning to the desired equality in~\cref{eq:claim}, our calculations up to this point yield
    \begin{align}
    \begin{split}
    \label{eq:simpClaim}
         &\frac{1}{(2\pi)^2}\int_{-\infty}^\infty \int_{-\infty}^\infty \trace\left(\BGtwobar(-\imunit\omega_1,-\imunit\omega_2)@ \BGtwored(\imunit\omega_1,\imunit\omega_2)^{\trans} \right) @\ds\omega_1@ \ds\omega_2\\
         &~~~~~~=\sum_{j=1}^r\sum_{k=1}^r\frac{1}{2\pi}\int_{-\infty}^\infty   \Bmred_{j,k}^{\trans}\BGtwobar(-\lambdared_j,-\imunit\omega_2)(\Bbred_j\otimes\Bbred_k)\frac{1}{\imunit\omega_2-\lambdared_k} @\ds\omega_2.
     \end{split}
    \end{align}
    Note that  
    \begin{align*}
        \int_{\contourRtwo} \Bmred_{j,k}^{\trans}\BGtwobar(-\lambdared_j,-\zeta_2)&(\Bbred_j\otimes\Bbred_k)\frac{1}{\zeta_2-\lambdared_k} @\ds\zeta_2\\&=\imunit\int_{-R_2}^{R_2}  \Bmred_{j,k}^{\trans}\BGtwobar(-\lambdared_j,-\imunit\omega_2)(\Bbred_j\otimes\Bbred_k)\frac{1}{\imunit\omega_2-\lambdared_k} @\ds\omega_2\\
        +\imunit\int_{\pi/2}^{3\pi/2}  \Bmred_{j,k}^{\trans}&\BGtwobar(-\lambdared_j,-R_2e^{\imunit\theta})(\Bbred_j\otimes\Bbred_k)\frac{1}{R_2e^{\imunit\theta}-\lambdared_k} R_2e^{\imunit\theta}@\ds\theta.
    \end{align*}
    We may choose $R_2>0$ large enough so that ~\Cref{lemma:resolventBound} applies, and thus we have the inequality $\|\BGtwobar(-\lambdared_j,-R_2e^{\imunit\theta})\|_{\frob}\leq C_{2}/(R_2-\|\overline{\BE}^{-1}\overline{\BA}\|_2)$ for all $\pi/2\leq\theta\leq3\pi/2$. Using a similar argument as before, one follows that
    \begin{equation*}
        \left|\int_{\pi/2}^{3\pi/2}  \Bmred_{j,k}^{\trans}\BGtwobar(-\lambdared_j,-R_2e^{\imunit\theta})(\Bbred_j\otimes\Bbred_k)\frac{1}{R_2e^{\imunit\theta}-\lambdared_k} R_2e^{\imunit\theta}@\ds\theta\right|
        \longrightarrow 0~~\mbox{as}~~R_2\longrightarrow\infty.
    \end{equation*}
    In the limit as $R_2\to\infty$, it follows that
    \begin{align*}
       & \lim_{R_2\to\infty} \int_{\contourRtwo} \Bmred_{j,k}^{\trans}\BGtwobar(-\lambdared_j,-\zeta_2)(\Bbred_j\otimes\Bbred_k)\frac{1}{\zeta_2-\lambdared_k} @\ds\zeta_2\\
       &~~~~~~= \imunit \int_{-\infty}^\infty   \Bmred_{j,k}^{\trans}\BGtwobar(-\lambdared_j,-\imunit\omega_2)(\Bbred_j\otimes\Bbred_k)\frac{1}{\imunit\omega_2-\lambdared_k} @\ds\omega_2.
    \end{align*}
    At this point, each integral appearing within the nested sum in the simplified expression~\cref{eq:simpClaim}
    can be evaluated by a straightforward application of the Residue Theorem.
    For each $j,k=1,\ldots,r$, the integrand $\Bmred_{j,k}^{\trans}\BGtwobar(-\lambdared_j,-z)(\Bbred_j\otimes\Bbred_k)/(z-\lambdared_k)$ is a scalar complex-valued rational function with poles at $-\mu_1,\ldots,-\mu_n\in\C_+$ (the eigenvalues of $\BE^{-1}\BA$) and $\lambdared_k\in\C_-$. Thus, for each $j,k=1,\ldots,r$, we have
    \begin{align*}
        &\frac{1}{2\pi}\int_{-\infty}^\infty   \Bmred_{j,k}^{\trans}\BGtwobar(-\lambdared_j,-\imunit\omega_2)(\Bbred_j\otimes\Bbred_k)\frac{1}{\imunit\omega_2-\lambdared_k} \ds\omega_2\\
        &~~~~~~=\lim_{R_2\to\infty}\frac{1}{2\pi\imunit} \int_{\contourRtwo} \Bmred_{j,k}^{\trans}\BGtwobar(-\lambdared_j,-\zeta_2)(\Bbred_j\otimes\Bbred_k)\frac{1}{\zeta_2-\lambdared_k} @\ds\zeta_2\\
        &~~~~~~=\Res\left[\Bmred_{j,k}^{\trans}\BGtwobar(-\lambdared_j,-s_2)(\Bbred_j\otimes\Bbred_k)\frac{1}{s_2-\lambdared_k}, @s_2=\lambdared_k\right]\\
        &~~~~~~=\Bmred_{j,k}^{\trans}\BGtwobar(-\lambdared_j,-\lambdared_k)(\Bbred_j\otimes\Bbred_k).
    \end{align*}
    Finally, plugging this into~\cref{eq:simpClaim} yields
        \begin{align*}
         &\frac{1}{(2\pi)^2}\int_{-\infty}^\infty \int_{-\infty}^\infty \trace\left(\BGtwobar(-\imunit\omega_1,-\imunit\omega_2)@ \BGtwored(\imunit\omega_1,\imunit\omega_2)^{\trans} \right) @\ds\omega_1@\ds\omega_2\\
            &~~~~~~=\sum_{j=1}^r\sum_{k=1}^r\Bmred_{j,k}^{\trans}\BGtwobar(-\lambdared_j,-\lambdared_k)(\Bbred_j\otimes\Bbred_k),
    \end{align*}
    which is the desired formula~\cref{eq:claim}, and thus the inner product formula in~\cref{eq:H2ipPoleRes}.
    The formula for the $\CH_2$ norm in~\cref{eq:H2normPoleRes} then follows directly by applying~\cref{eq:H2ipPoleRes} for $\Sys=\Sysred$.

%%%%%%%%%%%%%%%%%%%%%%%%%%%%%%%%%%%%%%%%
\section{Proof of~\Cref{thm:lqoH2OptInterpolationCons}}
\label{app:proofH2opt}
%%%%%%%%%%%%%%%%%%%%%%%%%%%%%%%%%%%%%%%%
\allowdisplaybreaks
    Recall from the sketch of the proof of~\Cref{thm:lqoH2OptInterpolationCons} that $\Syshat$ is any order-$r$, asymptotically stable \LQO{} system defined according to~\cref{eq:lqoSysRed} such that $\Syshat$ exists in a local neighborhood about $\Sysred$ and is not a locally-optimal $\CH_2$ approximation of $\Sys$. This leads to the inequality
    \begin{align}
        \begin{split}
        \label{eq:inequalityStar}
        \Rightarrow~~~ 0&\leq2\real @\langle\BG_1-\BGonered,@\BGonered-\BGonehat\rangle_{\CH_2^{p\times m}} + \|\BGonered - \BGonehat\|_{\CH_2^{p\times m}}^2\\
        &~~~~~+ 2\real @\langle\BG_2-\BGtwored,@\BGtwored-\BGtwohat\rangle_{\CH_2^{p\times m^2}} + \|\BGtwored - \BGtwohat\|_{\CH_2^{p\times m^2}}^2.
        \end{split}
    \end{align}
    Henceforth, we drop the matrix dimensions when invoking the Hardy space norms and inner products of the transfer functions~\cref{eq:lqoTfs} since they will be clear from context.
    Take $\varepsilon>0$ to be arbitrarily specified, and $\bxi$ to be an arbitrary unit vector in $\Cp$ or $\Cm$, depending on the setting.
    We will prove each set of interpolation conditions in~\cref{eq:H2OptConds} by choosing $\BGonehat$ and $\BGtwohat$ to differ from the $\CH_2$-optimal transfer functions $\BGonered$ and $\BGtwored$ by carefully chosen $\varepsilon$-perturbations of the poles and residue directions~\cref{eq:resComponents} of the optimal transfer functions. 
    Because the state-space matrices in~\cref{eq:lqoSys} and~\cref{eq:lqoSysRed} are assumed real, we take for granted that $\BGonebar(s)=\BG_1(s)$ and $\BGtwobar(s_1,s_2)=\BG_2(s_1,s_2)$ for all $s,s_1,s_2\in\C$ (and likewise for the transfer functions of~\cref{eq:lqoSysRed}) when invoking~\Cref{thm:H2PoleRes}, where $\BGonebar(s)$ and $\BGtwobar(s_1,s_2)$ are defined according to~\cref{eq:conjTf}.

    Because the conditions in~\cref{eq:H1RightCon} relate to the purely linear output, their derivation follows similarly to that of~\cite[Thm.~5.1.1]{AntBG20} for deriving the linear $\CH_2$-optimality conditions.
    For the sake of contradiction assume that the $(j,k)$-th interpolation condition in~\cref{eq:H2RightCon} does not hold.
    Define $\Syshat$ to be the system obtained by perturbing the $(j,k)$-th residue direction $\Bmred_{j,k}$ of $\BGtwored$ by $-\varepsilon e^{\imunit\theta}\bxi$ for $\theta\in\C$ that is to-be-defined. 
    In other words, the transfer functions of $\Syshat$ are defined as
    \begin{equation}
    \label{eq:pertM_tfs}
        \BGonehat(s)=\BGonered(s)~~\mbox{and}~~\BGtwored(s_1,s_2)-\BGtwohat(s_1,s_2)=\varepsilon e^{\imunit\theta}\frac{\bxi(\Bbred_j\otimes\Bbred_k)^{\trans}}{(s_1-\lambdared_j)(s_2-\lambdared_k)},
    \end{equation}
    where we choose $\theta\in\C$ to be
    \begin{equation*}
        \theta\defeq \pi-\arg\underbrace{\left(\bxi^{\trans}\left(\BG_2(-\lambdared_j,-\lambdared_k)-\BGtwored(-\lambdared_j,-\lambdared_k)\right)(\Bbred_j\otimes\Bbred_k)\right)}_{\defeq \,z}=\pi-\arg(z).
    \end{equation*}
    Note that $\theta$ is well-defined because we have assumed that the $(j,k)$-th condition~\cref{eq:H2RightCon} and $\bxi\in\Cp$ are nonzero. 
    Applying the formulae~\cref{eq:H2ipPoleRes} and~\cref{eq:H2normPoleRes} to the quantities in~\cref{eq:inequalityStar} for $\BGonehat$ and $\BGtwohat$ in~\cref{eq:pertM_tfs} as well as using the identity $z=|z|e^{\imunit\arg\left(z\right)}$ yields
    \begin{align*}
        \langle\BG_2-\BGtwored @,\BGtwored-\BGtwohat\rangle_{\CH_2}&=\varepsilon  e^{\imunit (\pi-\arg(z))} \bxi^{\trans}\left({\BG}_2(-\lambdared_j,-\lambdared_k)-{\widetilde{\BG}}_2(-\lambdared_j,-\lambdared_k)\right)(\Bbred_j\otimes\Bbred_k)\\
        &=-\varepsilon \left|\bxi^{\trans}\left(\BG_2(-\lambdared_j,-\lambdared_k)-\BGtwored(-\lambdared_j,-\lambdared_k)\right)(\Bbred_j\otimes\Bbred_k)\right|\neq 0,\\
        \mbox{and}~~\|\BGtwored-\BGtwohat\|_{\CH_2}^2&=\varepsilon^2|e^{\imunit\theta}|^2\bxi^{\trans}\left(\BGtwoRedBar(-\lambdared_j,-\lambdared_k)-\BGtwoCheckBar(-\lambdared_j,-\lambdared_k)\right)(\Bbred_j\otimes\Bbred_k)\\
        &=\varepsilon^2\frac{\|(\Bbred_j\otimes\Bbred_k)\|_2^2}{4\real(\lambdared_j)\real(\lambdared_k)}=O(\varepsilon^2).
    \end{align*}
    Clearly, $\langle\BG_1-\BGonered @,\BGonered-\BGonehat\rangle_{\CH_2}=\|\BG_1-\BGonered\|_{\CH_2}^2=0$ and $\|\BGtwored-\BGtwohat\|_{\CH_2}^2\geq 0$.
    Substituting the above calculations into~\cref{eq:inequalityStar}, we obtain
    \begin{equation*}
        0\leq-\varepsilon\left|\bxi^{\trans}\left(\BG_2(-\lambdared_j,-\lambdared_k)-\BGtwored(-\lambdared_j,-\lambdared_k)\right)(\Bbred_j\otimes\Bbred_k)\right|+ O\left(\varepsilon^2\right).
    \end{equation*}
    Since $\varepsilon>0$ is arbitrary, we may take $\varepsilon$ to be sufficiently small such that the negative $O\left(\varepsilon\right)$ term dominates the $O\left(\varepsilon^2\right)$ term, which is a contradiction. 
    However, we assumed initially that the $(j,k)$-th condition in~\cref{eq:H2RightCon} does not hold. Therefore, we must conclude by contradiction that it does.
    Repeating this argument for all $j,k$ gives
    \begin{equation*}
        \left(\BG_2(-\lambdared_j,-\lambdared_k)-\BGtwored(-\lambdared_j,-\lambdared_k)\right)(\Bbred_j\otimes\Bbred_k)=\Bzero_p~\,\mbox{for each}~\,j,k=1,\ldots,r,
    \end{equation*}
    which are precisely the right tangential conditions in~\cref{eq:H2RightCon}.

    Next, assume that the $k$-th condition in~\cref{eq:H1H2MixedLeftCon} does not hold.
    We obtain $\Syshat$ by applying the perturbation $-\varepsilon e^{\imunit\theta}\bxi$ to $\Bbred_k$ in~\cref{eq:poleResidue}, where $\theta$ is to-be-redefined (but using the same notation as before) and $\bxi\in\Cm$.
    The transfer functions of $\Syshat$ are
   \begin{align}
   \begin{split}
   \label{eq:pertB_Tfs}
        \BGonered(s)-\BGonehat(s)&=\varepsilon e^{\imunit 
        \theta}\frac{\Bcred_k\bxi^{\trans}}{s-\lambdared_k}~~\mbox{and}\\
        \BGtwored(s_1,s_2)-\BGtwohat(s_1,s_2)&= 
        \varepsilon e^{\imunit\theta}\sum_{\ell=1}^r\frac{\Bmred_{\ell,k}(\Bbred_\ell\otimes \bxi)^{\trans}}{(s_1-\lambdared_\ell)(s_2-\lambdared_k)} \\
        +\varepsilon e^{\imunit\theta}\sum_{\ell=1}^r&\frac{\Bmred_{k,\ell}( \bxi\otimes \Bbred_\ell)^{\trans}}{(s_1-\lambdared_k)(s_2-\lambdared_\ell)}-\varepsilon^2 e^{2\imunit\theta}\frac{\Bmred_{k,k}\left(\bxi\otimes\bxi\right)^{\trans}}{(s_1-\lambdared_k)(s_2-\lambdared_k)}.
    \end{split}
    \end{align}
    Implicitly, we have used the fact that the Kronecker product is bilinear~\cite{Bre78} in simplifying the expression for $\BGtwored-\BGtwohat$.
    We redefine $\theta\in\C$ as
    \begin{align}
        \theta\defeq\pi-\arg&\left[\Bcred_k^{\trans}\left(\BG_1(-\lambdared_k)-\BGonered(-\lambdared_k)\right)\bxi\right. \nonumber \\
        &~~~~~~~+\sum_{\ell=1}^r\Bmred_{k,\ell}^{\trans} \left(\BG_2(-\lambdared_k,-\lambdared_\ell)-\BGtwored(-\lambdared_k,-\lambdared_\ell)\right)(\BI_m\otimes \Bbred_\ell)\bxi \label{eq:deftheta}\\
        &~~~~~~~\left.+\sum _{\ell=1}^r\Bmred_{\ell,k}^{\trans} \left(\BG_2(-\lambdared_\ell,-\lambdared_k)-\BGtwored(-\lambdared_\ell,-\lambdared_k)\right)(\Bbred_\ell     \otimes \BI_m)\bxi\right], \nonumber
    \end{align}
    which is well-defined, since the quantity in the argument is nonzero.
    As before, we apply the formulae in~\Cref{thm:H2PoleRes} to compute the relevant terms in~\cref{eq:inequalityStar}. First, by~\cref{eq:H2ipPoleRes}, the inner products in~\cref{eq:inequalityStar} for $\BGonehat$ and $\BGtwohat$ in~\cref{eq:pertB_Tfs} are
    \begin{subequations}
    \begin{align}
    \begin{split}
    \label{eq:pertB_Ips}
        \langle \BG_1-\BGonered,&@\BGonered-\BGonehat \rangle_{\CH_2}= \varepsilon e^{\imunit\theta}\,\Bcred_k^{\trans}\left(\BG_1(-\lambdared_k)-\BGonered(-\lambdared_k)\right)\bxi\\
        \langle \BG_2-\BGtwored,&@\BGtwored-\BGtwohat \rangle_{\CH_2} = \\
        \varepsilon e^{\imunit\theta}\bigg[&\sum_{\ell=1}^r\Bmred_{\ell,k}^{\trans} \left(\BG_2(-\lambdared_\ell,-\lambdared_k)-\BGtwored(-\lambdared_\ell,-\lambdared_k)\right)(\Bbred_\ell\otimes \BI_m)\bxi\\
        +&\sum_{\ell=1}^r\Bmred_{k,\ell}^{\trans} 
        \left(\BG_2(-\lambdared_k,-\lambdared_\ell)-\BGtwored(-\lambdared_k,-\lambdared_\ell)\right)( \BI_m\otimes\Bbred_\ell)\bxi\bigg]\\
        -&\varepsilon^2 e^{2\imunit \theta} \Bmred_{k,k}^{\trans}\left(\BG_2(-\lambdared_k,-\lambdared_k)-\BGtwored(-\lambdared_k,-\lambdared_k)\right)(\bxi\otimes\bxi).
    \end{split}
    \end{align}
    In the latter inner product, we have used the fact that $(\Bbred_i\otimes \bxi)=(\Bbred_i\otimes \BI_m)\bxi$ and $( \bxi\otimes\Bbred_j)=(\BI_m\otimes\Bbred_j)\bxi$, which follows straightforwardly from the definition of the Kronecker product.
    By~\cref{eq:H2normPoleRes}, the norms of $\BGonered-\BGonehat$ and $\BGtwored-\BGtwohat$ are
    \begin{align}
        \label{eq:pertB_G1norm}
        \|\BGonered-\BGonehat\|_{\CH_2}^2&=\varepsilon^2 \frac{\|\Bcred_k\|_2^2}{-2\real(\lambdared_k)}=O\left(\varepsilon^2\right),\\
        \nonumber
        \|\BGtwored-\BGtwohat\|_{\CH_2}^2&=\varepsilon |e^{\imunit \theta}|\sum_{i=1}^r \Bmred_{i,k}^{\trans}\left(\BGtwoRedBar(-\lambdared_i,-\lambdared_k)-\BGtwoCheckBar(-\lambdared_i,-\lambdared_k)\right)(\Bbred_i\otimes\BI_m)\bxi\\
        \nonumber
        +&\varepsilon |e^{\imunit \theta}|\sum_{j=1}^r\Bmred_{k,j}^{\trans}\left(\BGtwoRedBar(-\lambdared_k,-\lambdared_j)-\BGtwoCheckBar(-\lambdared_k,-\lambdared_j)\right)(\BI_m\otimes \Bbred_j)\bxi\\
        \nonumber
        -&\varepsilon^2|e^{2\imunit \theta}|\Bmred_{k,k}^{\trans}\left(\BGtwoRedBar(-\lambdared_k,-\lambdared_k)-\BGtwoCheckBar(-\lambdared_k,-\lambdared_k)\right)(\bxi\otimes\bxi).
    \end{align}
    Substituting directly for the pole residue form of the error function $\BGtwored-\BGtwohat$ in~\cref{eq:pertB_Tfs} allows us to realize the norm of $\BGtwored-\BGtwohat$ as an $O\left(\varepsilon^2\right)$ term, i.e.,
    \begin{align}
    \begin{split}
    \label{eq:pertB_G2norm}
        \|\BGtwored &-\BGtwohat\|_{\CH_2}^2
        =\varepsilon^2 \sum_{i=1}^r \Bmred_{i,k}^{\trans}\bigg[\sum_{{\ell}=1}^r \frac{\overline{\Bmred}_{{\ell},k}(\overline{\Bbred}_{\ell}\otimes \bxi)^{\trans}}{(-\lambdared_i-\overline{\lambda}_{\ell})(-2\real(\lambdared_k))}\\
        +\sum_{{\ell}=1}^r &\frac{\overline{\Bmred}_{k,{\ell}}( \bxi\otimes \overline{\Bbred}_{\ell})^{\trans}}{(-\lambdared_i-\overline{\lambda}_k)(-\lambdared_k-\overline{\lambda}_{\ell})}\bigg](\overline{\Bbred}_i\otimes\BI_m)\bxi 
        +\varepsilon^2\sum_{j=1}^r\Bmred_{k,j}^{\trans}\bigg[\sum_{{\ell}=1}^r \frac{\overline{\Bmred}_{{\ell},k}(\overline{\Bbred}_{\ell}\otimes \bxi)^{\trans}}{(-\lambdared_k-\overline{\lambda}_{\ell})(-\lambdared_j-\overline{\lambda}_k)}\\
        &\hphantom{-\BGtwohat\|_{\CH_2}^2}+\sum_{{\ell}=1}^r \frac{\overline{\Bmred}_{{k,\ell}}( \bxi\otimes \overline{\Bbred}_{\ell})^{\trans}}{(-2\real(\lambdared_k))(-\lambdared_k-\overline{\lambda}_{\ell})}\bigg](\BI_m\otimes \Bbred_j)\bxi+O\left(\varepsilon^4\right)=O\left(\varepsilon^2\right).
    \end{split}
    \end{align}
    Then, substituting our calculations~\cref{eq:pertB_Ips}~--~\cref{eq:pertB_G2norm} into~\cref{eq:inequalityStar} and using the definition of $\theta$ in~\cref{eq:deftheta} yields
    \begin{align*}
        0&\leq -\varepsilon \left|\Bcred_k^{\trans}\left(\BG_1(-\lambdared_k)-\BGonered(-\lambdared_k)\right)\bxi\vphantom{\sum_{i=1}^r}\right.\\
        &~~~~~~~~~~+\left.\sum_{\ell=1}^r\Bmred_{k,\ell}^{\trans} \left(\BG_2(-\lambdared_k,-\lambdared_\ell)-\BGtwored(-\lambdared_k,-\lambdared_\ell)\right)(\BI_m\otimes\Bbred_\ell)\bxi\right.\\
        &~~~~~~~~~~+\left.\sum_{\ell=1}^r\Bmred_{\ell,k}^{\trans} \left(\BG_2(-\lambdared_\ell,-\lambdared_k)-\BGtwored(-\lambdared_\ell,-\lambdared_k)\right)(\Bbred_\ell\otimes\BI_m)\bxi\right| +{O(\varepsilon^2)}.
    \end{align*}
    For sufficiently small $\varepsilon > 0$, this yields a contradiction. 
    Because $\bxi\neq \Bzero_{m}$, we conclude
    \begin{align*}
        &\Bcred_k^{\trans}\left(\BG_1(-\lambdared_k)-\BGonered(-\lambdared_k)\right)+\sum_{\ell=1}^r\Bmred_{k,\ell}^{\trans} \left(\BG_2(-\lambdared_k,-\lambdared_\ell)-\BGtwored(-\lambdared_k,-\lambdared_\ell)\right)(\BI_m\otimes\Bbred_\ell) \\
        &~~~~~~+\sum_{\ell=1}^r\Bmred_{\ell,k}^{\trans} \left(\BG_2(-\lambdared_\ell,-\lambdared_k)-\BGtwored(-\lambdared_\ell,-\lambdared_k)\right)( \Bbred_\ell\otimes\BI_m)=\Bzero_m^{\trans}~~\mbox{for}~k=1,\ldots,r,
    \end{align*}
    by repeating this argument for all $k$, thereby proving~\cref{eq:H1H2MixedLeftCon}.
    \end{subequations}

    Finally, we prove the bi-tangential Hermite condition in~\cref{eq:H1H2MixedHermiteCon}. 
    As before, we assume that the $k$-th condition in~\cref{eq:H1H2MixedHermiteCon} does not hold. Redefine $\theta\in\C$ as
    \begin{align}
    \begin{split}
    \label{eq:pertPole_theta}
        \theta\defeq-\arg&\left[\Bcred_k^{\trans}\left(\frac{d}{ds}\BG_1(-\lambdared_k)-\frac{d}{ds}\BGonered(-\lambdared_k)\right)\Bbred_k\right.\\
        &\left.+\sum_{\ell=1}^r\Bmred_{k,\ell}^{\trans} \left(\frac{\partial}{\partial s_1}\BG_2(-\lambdared_k,-\lambdared_\ell)-\frac{\partial}{\partial s_1}\BGtwored(-\lambdared_k,-\lambdared_\ell)\right)(\Bbred_k\otimes \Bbred_\ell)\right.\\
        &\left.+\sum _{\ell=1}^r\Bmred_{\ell,k}^{\trans} \left(\frac{\partial}{\partial s_2}\BG_2(-\lambdared_\ell,-\lambdared_k)-\frac{\partial}{\partial s_2}\BGtwored(-\lambdared_\ell,-\lambdared_k)\right)( \Bbred_\ell\otimes\Bbred_k)\right].
    \end{split}
    \end{align}
    Take $\varepsilon >0$ to be small enough so that $\eta_k\defeq\lambdared_k+\varepsilon e^{\imunit\theta}$
    does not coincide with any of the remaining poles of $\Sysred$ and $\real\left(\eta_k\right)<0.$
    We obtain $\Syshat$ by replacing the $k$-th pole $\lambdared_k$ of $\Sysred$ with $\eta_k$ defined above. 
    Then, the transfer functions of $\Syshat$ are 
    \begin{align}
        \begin{split}
        \label{eq:pertPole_Tfs}
            \BGonered(s)-\BGonehat(s)=&~\Bcred_k\Bbred_k^{\trans}\left(\frac{1}{s-\lambdared_k} -\frac{1}{s-\eta_k}\right)\\
            \mbox{and}~~\BGtwored(s_1,s_2)-\BGtwohat(s_1,s_2)=& \sum_{\ell\neq k}^r\frac{\Bmred_{\ell,k}(\Bbred_\ell\otimes\Bbred_k)^{\trans}}{s_1-\lambdared_\ell}\left(\frac{1}{s_2-\lambdared_k}-\frac{1}{s_2-\eta_k}\right)\\
            +&\sum_{\ell\neq k}^r\left(\frac{1}{s_1-\lambdared_k}-\frac{1}{s_1-\eta_k}\right)\frac{\Bmred_{k,\ell}(\Bbred_k\otimes\Bbred_\ell)^{\trans}}{s_2-\lambdared_\ell}\\
            +~\Bmred_{k,k}(\Bbred_k\otimes\Bbred_k)^{\trans}&\left(\frac{1}{(s_1-\lambdared_k)(s_2-\lambdared_k)}-\frac{1}{(s_1-\eta_k)(s_2-\eta_k)}\right).
        \end{split}
    \end{align}
    We observe that the error function $\BGonered-\BGonehat$ in~\cref{eq:pertPole_Tfs} has two poles $\lambdared_k$ and $\eta_k$ corresponding to the residues $\Bcred_k\Bbred_k^{\trans}$ and $-\Bcred_k\Bbred_k^{\trans}$. Thus, applying~\cref{eq:H2ipPoleRes} yields
    \begin{equation*}
        \langle\BG_1-\BGonered,\BGonered-\BGonehat\rangle_{\CH_2}=\Bcred_k^{\trans}\bigg(\underbrace{{\BG_1(-\lambdared_k)-\BGonered(-\lambdared_k)\bigg)\Bbred_k}}_{=\,\Bzero_p~\,\textrm{by~\cref{eq:H1RightCon}}} -\Bcred_k^{\trans}\bigg(\BG_1(-\eta_k)-\BGonered(-\eta_k)\bigg)\Bbred_k.
    \end{equation*}
    To resolve this further, we recognize that $\BG_1(s)$ and $\BGonered(s)$ are both analytic at $s=-\lambdared_k$, and thus admit power series representations about this point. Expanding each about $s=-\lambdared_k$ and evaluating at $s=-\eta_k$ gives
    \begin{subequations}
    \begin{align}
        \nonumber
        \langle\BG_1-\BGonered,\BGonered-\BGonehat\rangle_{\CH_2}&=-\Bcred_k^{\trans}\bigg(\BG_1(-\eta_k)-\BGonered(-\eta_k)\bigg)\Bbred_k\\
        \nonumber
        &=-\Bcred_k^{\trans}\bigg[\bigg(\BG_1(-\lambdared_k)+(\underbrace{-\eta_k-\lambdared_k}_{=-\varepsilon e^{\imunit \theta}})\frac{d}{ds}\BG_1(-\lambdared_k)+O\left(\varepsilon^2\right)\bigg)\\
        \nonumber
        &\hphantom{-\Bcred_k^{\trans}\bigg[}-\bigg(\BGonered(-\lambdared_k)+(\underbrace{-\eta_k-\lambdared_k}_{=-\varepsilon e^{\imunit \theta}})\frac{d}{ds}\BGonered(-\lambdared_k)+O\left(\varepsilon^2\right)\bigg)\bigg]\Bbred_k\\
        \label{eq:pertPole_G1ip}
        &=-\varepsilon e^{\imunit\theta}\Bcred_k^{\trans}\left(\frac{d}{ds}\BGonered(-\lambdared_k)-\frac{d}{ds}\BG_1(-\lambdared_k)\right)\Bbred_k+O\left(\varepsilon^2\right),
    \end{align}
    since $(\BG_1(-\lambdared_k)-\BGonered(-\lambdared_k))\Bbred_k=\Bzero_p$ by~\cref{eq:H1RightCon}. Accounting for all the pole-residue pairs of $\BGtwored-\BGtwohat$ in~\cref{eq:pertPole_Tfs}, applying~\cref{eq:H2ipPoleRes} yields
    \begin{align}
    \begin{split}
    \label{eq:pertPole_G2ip}
        \langle\BG_2-\BGtwored, \BGtwored-\BGtwohat\rangle_{\CH_2}&={\sum_{i\neq k}^r\Bmred_{i,k}^{\trans}\bigg(\underbrace{\BG_2(-\lambdared_i,-\lambdared_k)-\BGtwored(-\lambdared_i,-\lambdared_k)\bigg)(\Bbred_i\otimes\Bbred_k)}_{=\Bzero_p~\textrm{by \cref{eq:H2RightCon}}}}\\
        &-\sum_{i\neq k}^r\Bmred_{i,k}^{\trans}\bigg(\BG_2(-\lambdared_i,-\eta_k)-\BGtwored(-\lambdared_i,-\eta_k)\bigg)(\Bbred_i\otimes\Bbred_k)\\
        &+{\sum_{j\neq k}^r\Bmred_{k,j}^{\trans}\bigg(\underbrace{\BG_2(-\lambdared_k,-\lambdared_j)-\BGtwored(-\lambdared_k,-\lambdared_j)\bigg)(\Bbred_k\otimes\Bbred_j)}_{=\Bzero_p~\textrm{by \cref{eq:H2RightCon}}}}\\\
        &-\sum_{j\neq k}^r\Bmred_{k,j}^{\trans}\bigg(\BG_2(-\eta_k,-\lambdared_j)-\BGtwored(-\eta_k,-\lambdared_j)\bigg)(\Bbred_k\otimes\Bbred_j)\\
        &\hphantom{\sum_{j\neq k}^r}+{\Bmred_{k,k}^{\trans}\bigg(\underbrace{\BG_2(-\lambdared_k,-\lambdared_k)-\BGtwored(-\lambdared_k,-\lambdared_k)\bigg)(\Bbred_k\otimes\Bbred_k)}_{=\Bzero_p~\textrm{by \cref{eq:H2RightCon}}}}\\\
        &\hphantom{\sum_{j\neq k}^r}-\Bmred_{k,k}^{\trans}\bigg(\BG_2(-\eta_k,-\eta_k)-\BGtwored(-\eta_k,-\eta_k)\bigg)(\Bbred_k\otimes\Bbred_k).
    \end{split}
    \end{align}
    Both $\BG_2(s_1,s_2)$ and $\BGtwored(s_1,s_2)$ are analytic at $s=-\lambdared_k$ in each separate argument, and thus admit power series expansions about this point.
    Expanding $\BG_2(-\lambdared_i,s_2)-\BGtwored(-\lambdared_i,s_2)$ in $s_2$ about $-\lambdared_k$ and evaluating at $s_2=\eta_k$ for each $i\neq k$ gives
    \allowdisplaybreaks[0]
    \begin{align*}
        \Bmred_{i,k}^{\trans}&\left(\BG_2(-\lambdared_i,-\eta_k)-\BGtwored(-\lambdared_i,-\eta_k)\right)(\Bbred_i\otimes\Bbred_k) \\
        =&~\Bmred_{i,k}^{\trans}\bigg(\BG_2(-\lambdared_i,-\lambdared_k)+(\underbrace{-\eta_k-\lambdared_k}_{=-\varepsilon e^{\imunit \theta}})\frac{\partial}{\partial s_2}\BG_2(-\lambdared_i,-\lambdared_k) + O\left(\varepsilon^2\right)\bigg)(\Bbred_i\otimes\Bbred_k)\\
        -&~\Bmred_{i,k}^{\trans}\bigg(\BGtwored(-\lambdared_i,-\lambdared_k)+(\underbrace{-\eta_k-\lambdared_k}_{=-\varepsilon e^{\imunit \theta}})\frac{\partial}{\partial s_2}\BGtwored(-\lambdared_i,-\lambdared_k) + O\left(\varepsilon^2\right)\bigg)(\Bbred_i\otimes\Bbred_k)\\
        =&~\varepsilon e^{\imunit\theta}\Bmred_{i,k}^{\trans}\left(\frac{\partial}{\partial s_2}\BGtwored(-\lambdared_i,-\lambdared_k) -\frac{\partial}{\partial s_2}\BG_2(-\lambdared_i,-\lambdared_k) \right)(\Bbred_i\otimes\Bbred_k) + O\left(\varepsilon^2\right),
    \end{align*}
    \allowdisplaybreaks
    because $(\BG_2(-\lambdared_i,-\lambdared_k)-\BGtwored(-\lambdared_i,-\lambdared_k))(\Bbred_i\otimes\Bbred_k)=\Bzero_p$
    by~\cref{eq:H2RightCon}.
    Similarly, expanding $\BG_2(s_1,-\lambdared_j)-\BGtwored(s_1,-\lambdared_j)$ in $s_1$ about $-\lambdared_k$ and evaluating at $s_1=\eta_k$ gives
    \begin{align*}
        \Bmred_{k,j}^{\trans}&\left(\BG_2(-\eta_k,-\lambdared_j)-\BGtwored(-\eta_k,-\lambdared_j)\right)(\Bbred_k\otimes\Bbred_j) \\
        =&~\varepsilon e^{\imunit\theta}\Bmred_{k,j}^{\trans}\left(\frac{\partial}{\partial s_1}\BGtwored(-\lambdared_k,-\lambdared_j) -\frac{\partial}{\partial s_1}\BG_2(-\lambdared_k,-\lambdared_j) \right)(\Bbred_k\otimes\Bbred_j) + O\left(\varepsilon^2\right).
    \end{align*}
    To finalize~\cref{eq:pertPole_G2ip}, expand $\BG_2(s_1,-\eta_k)$ in $s_1$ about $-\lambdared_k$ and evaluate $s_1=-\eta_k$:
    \begin{align*}
        \BG_2(-\eta_k, -\eta_k)=\BG_2(-\lambdared_k,-\eta_k) - \varepsilon e^{\imunit \theta}\frac{\partial}{\partial s_1}\BG_2(-\lambdared_k,-\eta_k)+O\left(\varepsilon^2\right).
    \end{align*}
    Then, express $\BG_2(-\lambdared_k,-\eta_k)$ as a series expansion of $\BG_2(-\lambdared_k,s_2)$ in $s_2$ about $-\lambdared_k$, evaluated at $s_2=-\eta_k$:
    \begin{align*}
        \BG_2(-\lambdared_k, -\eta_k)=\BG_2(-\lambdared_k,-\lambdared_k) - \varepsilon e^{\imunit \theta}\frac{\partial}{\partial s_2}\BG_2(-\lambdared_k,-\lambdared_k)+O\left(\varepsilon^2\right).
    \end{align*}
    Because $\BG_2(s_1,s_2)$ is analytic in each argument, its partial derivative $\frac{\partial}{\partial s_1}\BG_2(-\lambdared_k,s_2)$ is analytic in $s_2$, and may also be expressed as a power series about $-\lambdared_k$. Expand about this point and evaluate at $s_2=-\eta_k$:
    \begin{align*}
        \frac{\partial}{\partial s_1}\BG_2(-\lambdared_k,-\eta_k)=\frac{\partial}{\partial s_1}\BG_2(-\lambdared_k,-\lambdared_k) - \varepsilon e^{\imunit\theta}\frac{\partial}{\partial s_2}\frac{\partial}{\partial s_1}\BG_2(-\lambdared_k,-\lambdared_k) + O\left(\varepsilon^2\right).
    \end{align*}
    Putting this all together, we have
    \begin{align*}
        \BG_2(-\eta_k,-\eta_k)&=\BG_2(-\lambdared_k,-\lambdared_k) - \varepsilon e^{\imunit\theta}\left(\frac{\partial}{\partial s_1}\BG_2(-\lambdared_k,-\lambdared_k) + \frac{\partial}{\partial s_2}\BG_2(-\lambdared_k,-\lambdared_k)\right)\\
        &~~~~~~~+ O\left(\varepsilon^2\right).
    \end{align*}
    Applying the exact same logic to the $\BGtwored(-\eta_k,-\eta_k)$ term, we have
    \begin{align*}
        \BGtwored(-\eta_k,-\eta_k)&=\BGtwored(-\lambdared_k,-\lambdared_k) - \varepsilon e^{\imunit\theta}\left(\frac{\partial}{\partial s_1}\BGtwored(-\lambdared_k,-\lambdared_k) + \frac{\partial}{\partial s_2}\BGtwored(-\lambdared_k,-\lambdared_k)\right)\\
        &~~~~~~~+ O\left(\varepsilon^2\right).
    \end{align*}
    Combining these calculations, we have
    \begin{align*}
        \Bmred_{k,k}^{\trans}&\left(\BG_2(-\eta_k,-\eta_k)-\BGtwored(-\eta_k,-\eta_k)\right)(\Bbred_k\otimes\Bbred_k) \\
        =&~\varepsilon e^{\imunit\theta}\Bmred_{k,k}^{\trans}\left(\frac{\partial}{\partial s_1}\BGtwored(-\lambdared_k,-\lambdared_k) -\frac{\partial}{\partial s_1}\BG_2(-\lambdared_k,-\lambdared_k) \right)(\Bbred_k\otimes\Bbred_k) \\
        +&~\varepsilon e^{\imunit\theta}\Bmred_{k,k}^{\trans}\left(\frac{\partial}{\partial s_2}\BGtwored(-\lambdared_k,-\lambdared_k) -\frac{\partial}{\partial s_2}\BG_2(-\lambdared_k,-\lambdared_k) \right)(\Bbred_k\otimes\Bbred_k)+ O\left(\varepsilon^2\right),
    \end{align*}
    and so the expression for inner product~\cref{eq:pertPole_G2ip} ultimately simplifies to 
    \begin{align}
    \begin{split}
    \nonumber
        \langle\BG_2-\BGtwored, \BGtwored-\BGtwohat\rangle_{\CH_2}=&-\sum_{i\neq k}^r\Bmred_{i,k}^{\trans}\bigg(\BG_2(-\lambdared_i,-\eta_k)-\BGtwored(-\lambdared_i,-\eta_k)\bigg)(\Bbred_i\otimes\Bbred_k)\\
        &-\sum_{j\neq k}^r\Bmred_{k,j}^{\trans}\bigg(\BG_2(-\eta_k,-\lambdared_j)-\BGtwored(-\eta_k,-\lambdared_j)\bigg)(\Bbred_k\otimes\Bbred_j)\\
        &-\Bmred_{k,k}^{\trans}\bigg(\BG_2(-\eta_k,-\eta_k)-\BGtwored(-\eta_k,-\eta_k)\bigg)(\Bbred_k\otimes\Bbred_k)
    \end{split}\\
    \begin{split}
    \label{eq:pertPole_G2ip_simp}
        =-\varepsilon e^{\imunit\theta}\bigg[\sum_{\ell=1}^r&\Bmred_{k,\ell}^{\trans}\left(\frac{\partial}{\partial s_1}\BGtwored(-\lambdared_k,-\lambdared_\ell) -\frac{\partial}{\partial s_1}\BG_2(-\lambdared_k,-\lambdared_\ell) \right)(\Bbred_k\otimes\Bbred_\ell)\\
        +\sum_{\ell = 1}^r \Bmred_{\ell,k}^{\trans}\left(\frac{\partial}{\partial s_2}\right.&\left.\BGtwored(-\lambdared_\ell,-\lambdared_k) -\frac{\partial}{\partial s_2}\BG_2(-\lambdared_\ell,-\lambdared_k) \right)(\Bbred_\ell\otimes\Bbred_k) \bigg]+O\left(\varepsilon^2\right).
    \end{split}
    \end{align}
    \end{subequations}
    Note that in passing from the first to the second equality, we have relabeled the sums over $i$ and $j$ to run over $\ell$ to agree with the claim~\cref{eq:H1H2MixedHermiteCon}, and grouped the $(k,k)$-th terms into each of these sums.
    What remains is to deal with the norms in~\cref{eq:inequalityStar} for this case.
    Similar to the previous arguments, we ultimately show that $\|\BGonered-\BGonehat\|_{\CH_2}^2$ and $\|\BGtwored-\BGtwohat\|_{\CH_2}^2$ are $O\left(\varepsilon^2\right)$ terms.
    The calculations required to resolve these terms as $O\left(\varepsilon^2\right)$ are direct, but tedious and not particularly illuminating. Therefore, we refer to~\cite[Appendix~A]{Rei25} for the full calculations.
    Using the fact that these norms are $O\left(\varepsilon^2\right)$,
    substituting the calculations for the inner products~\cref{eq:pertPole_G1ip},~\cref{eq:pertPole_G2ip_simp} together with the definition of $\theta$ in~\cref{eq:pertPole_theta},  inequality~\cref{eq:inequalityStar} simplifies to
    \begin{align*}
        0&\leq -\varepsilon\left|\Bcred_k^{\trans}\left(\frac{d}{ds}\BG_1(-\lambdared_k)-\frac{d}{ds}\BGonered(-\lambdared_k)\right)\Bbred_k\right.\\
        &~~~~~~~~~+\sum_{\ell=1}^r\Bmred_{k,\ell}^{\trans} \left(\frac{\partial}{\partial s_1}\BG_2(-\lambdared_k,-\lambdared_\ell)-\frac{\partial}{\partial s_1}\BGtwored(-\lambdared_k,-\lambdared_\ell)\right)(\Bbred_k\otimes \Bbred_\ell) \\
        &~~~~~~~~~+\left.\sum_{\ell=1}^r\Bmred_{\ell,k}^{\trans} \left(\frac{\partial}{\partial s_2}\BG_2(-\lambdared_\ell,-\lambdared_k)-\frac{\partial}{\partial s_2}\BGtwored(-\lambdared_j,-\lambdared_k)\right)( \Bbred_\ell\otimes\Bbred_k)\right| + O\left(\varepsilon^2\right).
    \end{align*}
    By the same logic used to prove~\cref{eq:H1H2MixedLeftCon}, this inequality yields a contradiction for small values of $\varepsilon>0$,
    and thus the interpolation conditions in~\cref{eq:H1H2MixedHermiteCon} must hold.

%% file: root.bbl
\begin{thebibliography}{10}

\bibitem{Ant05}
{\sc A.~C. Antoulas}, {\em Approximation of Large-Scale Dynamical Systems},
  SIAM, Philadelphia, PA, 2005, \url{https://doi.org/10.1137/1.9780898718713}.

\bibitem{AntBG20}
{\sc A.~C. Antoulas, C.~A. Beattie, and S.~G{\"u}{\u{g}}ercin}, {\em
  Interpolatory Methods for Model Reduction}, Computational Science \&
  Engineering, SIAM, Philadelphia, PA, 2020,
  \url{https://doi.org/10.1137/1.9781611976083}.

\bibitem{AumW23}
{\sc Q.~Aumann and S.~W.~R. Werner}, {\em Structured model order reduction for
  vibro-acoustic problems using interpolation and balancing methods}, J. Sound
  Vib., 543 (2023), p.~117363, \url{https://doi.org/10.1016/j.jsv.2022.117363}.

\bibitem{BalG24}
{\sc L.~Balicki and S.~Gugercin}, {\em Energy-based approximation of linear
  systems with polynomial outputs}, e-prints 2409.19730, arXiv, 2024,
  \url{https://arxiv.org/abs/2409.19730}.

\bibitem{BauBF14}
{\sc U.~Baur, P.~Benner, and L.~Feng}, {\em Model order reduction for linear
  and nonlinear systems: a system-theoretic perspective}, Archives of
  Computational Methods in Engineering, 21 (2014), pp.~331--358,
  \url{https://doi.org/10.1007/s11831-014-9111-2}.

\bibitem{BeaGW12}
{\sc C.~Beattie, S.~Gugercin, and S.~Wyatt}, {\em Inexact solves in
  interpolatory model reduction}, Linear Algebra and its Applications, 436
  (2012), pp.~2916--2943, \url{https://doi.org/10.1016/j.laa.2011.07.015}.

\bibitem{BenB12}
{\sc P.~Benner and T.~Breiten}, {\em Interpolation-based $\mathcal{H}_2$-model
  reduction of bilinear control systems}, SIAM Journal on Matrix Analysis and
  Applications, 33 (2012), pp.~859--885,
  \url{https://doi.org/10.1137/110836742}.

\bibitem{BenGG18}
{\sc P.~Benner, P.~Goyal, and S.~Gugercin}, {\em $\mathcal{H}_2$-quasi-optimal
  model order reduction for quadratic-bilinear control systems}, SIAM Journal
  on Matrix Analysis and Applications, 39 (2018), pp.~983--1032,
  \url{https://doi.org/10.1137/16M1098280}.

\bibitem{BenGPD21}
{\sc P.~Benner, P.~Goyal, and I.~Pontes~Duff}, {\em Gramians, energy
  functionals, and balanced truncation for linear dynamical systems with
  quadratic outputs}, IEEE Transactions on Automatic Control, 67 (2021),
  pp.~886--893, \url{https://doi.org/10.1109/TAC.2021.3086319}.

\bibitem{BenMS05}
{\sc P.~Benner, V.~Mehrmann, and D.~C. Sorensen}, {\em Dimension Reduction of
  Large-Scale Systems}, vol.~45 of Lectures Notes in Computional Science and
  Engineering, Springer, Berlin, Heidelberg, 2005,
  \url{https://doi.org/10.1007/3-540-27909-1}.

\bibitem{BenOCW17}
{\sc P.~Benner, M.~Ohlberger, A.~Cohen, and K.~Willcox}, {\em Model Reduction
  and Approximation: Theory and Algorithms}, SIAM, Philadelphia, PA, 2017,
  \url{https://doi.org/10.1137/1.9781611974829}.

\bibitem{BocC49}
{\sc S.~Bochner and K.~Chandrasekharan}, {\em Fourier Transforms}, Princeton
  University Press, 1949, \url{https://doi.org/10.1515/9781400882243}.

\bibitem{Bre78}
{\sc J.~Brewer}, {\em Kronecker products and matrix calculus in system theory},
  IEEE Transactions on Circuits and Systems, 25 (1978), pp.~772--781,
  \url{https://doi.org/10.1109/TCS.1978.1084534}.

\bibitem{Bu24}
{\sc Y.-P. Bu}, {\em Krylov subspace model order reduction of linear dynamical
  systems with quadratic output}, Transactions of the Institute of Measurement
  and Control, 47 (2025), pp.~827--838,
  \url{https://doi.org/10.1177/01423312241257298}.

\bibitem{BunKVW10}
{\sc A.~Bunse-Gerstner, D.~Kubalinska, G.~Vossen, and D.~Wilczek}, {\em
  $h_2$-norm optimal model reduction for large scale discrete dynamical {MIMO}
  systems}, Journal of Computational and Applied Mathematics, 233 (2010),
  pp.~1202--1216, \url{https://doi.org/10.1016/j.cam.2008.12.029}.
\newblock Special Issue Dedicated to William B. Gragg on the Occasion of His
  70th Birthday.

\bibitem{CaoEtal22}
{\sc X.~Cao, J.~Maubach, W.~Schilders, and S.~Weiland}, {\em
  Interpolation-based model order reduction for quadratic-bilinear systems and
  $\mathcal{H}_2$ optimal approximation}, in Realization and Model Reduction of
  Dynamical Systems: A Festschrift in Honor of the 70th Birthday of Thanos
  Antoulas, Springer, 2022, pp.~117--135,
  \url{https://doi.org/10.1007/978-3-030-95157-3_7}.

\bibitem{DiazHGA23}
{\sc A.~N. Diaz, M.~Heinkenschloss, I.~V. Gosea, and A.~C. Antoulas}, {\em
  Interpolatory model reduction of quadratic-bilinear dynamical systems with
  quadratic-bilinear outputs}, Advances in Computational Mathematics, 49
  (2023), pp.~1--28, \url{https://doi.org/10.1007/s10444-023-10096-2}.

\bibitem{FlaG15}
{\sc G.~Flagg and S.~Gugercin}, {\em Multipoint {V}olterra series interpolation
  and $\mathcal{H}_2$ optimal model reduction of bilinear systems}, SIAM
  Journal on Matrix Analysis and Applications, 36 (2015), pp.~549--579,
  \url{https://doi.org/10.1137/130947830}.

\bibitem{Fla12}
{\sc G.~M. Flagg}, {\em Interpolation Methods for the Model Reduction of
  Bilinear Systems}, {D}issertation, Virginia Tech, 2012,
  \url{https://doi.org/10919/27521}.

\bibitem{Gam03}
{\sc T.~Gamelin}, {\em Complex Analysis}, Springer Science \& Business Media,
  New York, NY, 2003, \url{https://doi.org/10.1007/978-0-387-21607-2}.

\bibitem{GosA19}
{\sc I.~V. Gosea and A.~C. Antoulas}, {\em A two-sided iterative framework for
  model reduction of linear systems with quadratic output}, in 2019 IEEE 58th
  Conference on Decision and Control (CDC), IEEE, 2019, pp.~7812--7817,
  \url{https://doi.org/10.1109/CDC40024.2019.9030025}.

\bibitem{GosG20}
{\sc I.~V. Gosea and S.~Gugercin}, {\em Data-driven modeling of linear
  dynamical systems with quadratic output in the {AAA} framework}, Journal of
  Scientific Computing, 91 (2022), p.~16,
  \url{https://doi.org/10.1007/s10915-022-01771-5}.

\bibitem{GugAB08}
{\sc S.~Gugercin, A.~C. Antoulas, and C.~Beattie}, {\em $\mathcal{H}_2$ model
  reduction for large-scale linear dynamical systems}, SIAM Journal on Matrix
  Analysis and Applications, 30 (2008), pp.~609--638,
  \url{https://doi.org/10.1137/06066612}.

\bibitem{HilU25}
{\sc B.~Hillebrecht and B.~Unger}, {\em $\mathcal{H}_\infty$ model order
  reduction for quadratic output systems}, e-prints 2505.12529, arXiv, 2025,
  \url{https://arxiv.org/abs/2505.12529}.

\bibitem{HolNSU25}
{\sc T.~Holicki, J.~Nicodemus, P.~Schwerdtner, and B.~Unger}, {\em Energy
  matching in reduced passive and port-{H}amiltonian systems}, SIAM Journal on
  Control and Optimization, 63 (2025), pp.~2154--2176,
  \url{https://doi.org/10.1137/23M1600931}.

\bibitem{MagN79}
{\sc J.~R. Magnus and H.~Neudecker}, {\em The commutation matrix: some
  properties and applications}, The Annals of Statistics, 7 (1979),
  pp.~381--394, \url{https://doi.org/10.1214/aos/1176344621}.

\bibitem{MehU23}
{\sc V.~Mehrmann and B.~Unger}, {\em Control of port-{H}amiltonian
  differential-algebraic systems and applications}, Acta Numerica, 32 (2023),
  pp.~395--515, \url{https://doi.org/10.1017/S0962492922000083}.

\bibitem{MeiL67}
{\sc L.~Meier and D.~Luenberger}, {\em Approximation of linear constant
  systems}, IEEE Transactions on Automatic Control, 12 (1967), pp.~585--588,
  \url{https://doi.org/10.1109/TAC.1967.1098680}.

\bibitem{PrzDGB24}
{\sc J.~Przybilla, I.~Pontes~Duff, P.~Goyal, and P.~Benner}, {\em Balanced
  truncation of descriptor systems with a quadratic output}, e-prints
  2402.14716, arXiv, 2024, \url{https://arxiv.org/abs/2402.14716}.

\bibitem{Pul18}
{\sc R.~Pulch}, {\em Model order reduction and low-dimensional representations
  for random linear dynamical systems}, Mathematics and Computers in
  Simulation, 144 (2018), pp.~1--20.

\bibitem{Pul23}
{\sc R.~Pulch}, {\em Energy-based model order reduction for linear stochastic
  {G}alerkin systems of second order}, PAMM, 23 (2023), p.~e202300038,
  \url{https://doi.org/10.1002/pamm.20230003833}.

\bibitem{PulA19}
{\sc R.~Pulch and A.~Narayan}, {\em Balanced truncation for model order
  reduction of linear dynamical systems with quadratic outputs}, SIAM Journal
  on Scientific Computing, 41 (2019), pp.~A2270--A2295,
  \url{https://doi.org/10.1137/17M1148797}.

\bibitem{supRei26}
{\sc S.~Reiter}, {\em Code, data, and results for numerical experiments in
  ``$\mathcal{H}_2$-optimal model reduction of linear quadratic-output systems
  by multivariate rational interpolation'' (version 1.1)}, Mar. 2026,
  \url{https://doi.org/10.5281/zenodo.18829841}.

\bibitem{Reietal25}
{\sc S.~Reiter, I.~P. Duff, I.~V. Gosea, and S.~Gugercin}, {\em ${H}_2$ optimal
  model reduction of linear systems with multiple quadratic outputs}, IEEE
  Transactions on Automatic Control, 71 (2026), pp.~3168 -- 3183,
  \url{https://doi.org/10.1109/TAC.2025.3636441}.

\bibitem{ReiW24}
{\sc S.~Reiter and S.~W.~R. Werner}, {\em Interpolatory model reduction of
  dynamical systems with root mean squared error}, IFAC-PapersOnLine, 59
  (2025), pp.~385--390, \url{https://doi.org/10.1016/j.ifacol.2025.03.066}.

\bibitem{Rei25}
{\sc S.~J. Reiter}, {\em Dimension Reduction in Structured Dynamical Systems:
  Optimal-$\mathcal{H}_2$ Approximation, Data-Driven Balancing, and Real-Time
  Monitoring}, {D}issertation, Virginia Tech, 2025,
  \url{https://doi.org/10919/134223}.

\bibitem{Rug81}
{\sc W.~J. Rugh}, {\em Nonlinear System Theory}, Johns Hopkins University
  Press, Baltimore, MD, 1981.
\newblock {ISBN}: O-8018-2549-0, Web version prepared in 2002.

\bibitem{SonZXUS24}
{\sc Q.-Y. Song, U.~Zulfiqar, Z.-H. Xiao, M.~M. Uddin, and V.~Sreeram}, {\em
  Balanced truncation of linear systems with quadratic outputs in limited time
  and frequency intervals}, e-prints 2402.11445, arXiv, 2024,
  \url{https://arxiv.org/abs/2402.11445}.

\bibitem{VanM10}
{\sc R.~Van~Beeumen and K.~Meerbergen}, {\em Model reduction by balanced
  truncation of linear systems with a quadratic output}, in AIP Conference
  Proceedings, vol.~1281, American Institute of Physics, 2010, pp.~2033--2036,
  \url{https://doi.org/10.1063/1.3498345}.

\bibitem{VanVNLM12}
{\sc R.~Van~Beeumen, K.~Van~Nimmen, G.~Lombaert, and K.~Meerbergen}, {\em Model
  reduction for dynamical systems with quadratic output}, International Journal
  for Numerical Methods in Engineering, 91 (2012), pp.~229--248,
  \url{https://doi.org/10.1002/nme.4255}.

\bibitem{Van06}
{\sc A.~van~der Schaft}, {\em Port-{H}amiltonian systems: an introductory
  survey}, in International congress of mathematicians, European Mathematical
  Society Publishing House (EMS Ph), 2006, pp.~1339--1365,
  \url{https://doi.org/10.4171/022-3/65}.

\bibitem{VanGPA08}
{\sc P.~van Dooren, K.~A. Gallivan, and P.-A. Absil}, {\em {$H_2$}-optimal
  model reduction of {MIMO} systems}, Applied Mathematics Letters, 21 (2008),
  pp.~1267--1273, \url{https://doi.org/10.1016/j.aml.2007.09.015}.

\bibitem{VanGA10}
{\sc P.~van Dooren, K.~A. Gallivan, and P.-A. Absil}, {\em
  $\mathcal{H}_2$-optimal model reduction with higher-order poles}, SIAM
  Journal on Matrix Analysis and Applications, 31 (2010), pp.~2738--2753,
  \url{https://doi.org/10.1137/080731591}.

\bibitem{Wil70}
{\sc D.~A. Wilson}, {\em Optimum solution of model-reduction problem}, in
  Proceedings of the Institution of Electrical Engineers, vol.~117, IET, 1970,
  pp.~1161--1165, \url{https://doi.org/10.1049/piee.1970.0227}.

\bibitem{YanWJ25}
{\sc P.~Yang, Z.-H. Wang, and Y.-L. Jiang}, {\em {$H_2$} optimal model
  reduction of linear dynamical systems with quadratic output by the
  {R}iemannian {BFGS} method}, Mathematics and Computers in Simulation,
  (2025), \url{https://doi.org/10.1016/j.matcom.2025.03.021}.

\bibitem{YueM12}
{\sc Y.~Yue and K.~Meerbergen}, {\em Using {K}rylov-{P}ad{\'e} model order
  reduction for accelerating design optimization of structures and vibrations
  in the frequency domain}, International Journal for Numerical Methods in
  Engineering, 90 (2012), pp.~1207--1232,
  \url{https://doi.org/10.1002/nme.3357}.

\bibitem{YueM13}
{\sc Y.~Yue and K.~Meerbergen}, {\em Accelerating optimization of parametric
  linear systems by model order reduction}, SIAM Journal on Optimization, 23
  (2013), pp.~1344--1370, \url{https://doi.org/10.1137/120869171}.

\end{thebibliography}
